\documentclass[11pt]{amsart}
\usepackage{amsthm,amssymb,amsmath,epsfig,graphics,appendix,bm, mathrsfs,bbm}
\usepackage{graphicx}
\usepackage{hyperref}
\usepackage{bm}
\usepackage[dvipsnames, svgnames, x11names]{xcolor}
\usepackage{stmaryrd} 
\usepackage[left=1in, right=1in, top=1.1in,bottom=1.1in]{geometry}
\usepackage{enumitem}

\def\R{\mathbb R}
\def\E{\mathbb E}

\newtheorem{theorem}{Theorem}[section]
\newtheorem{lemma}[theorem]{Lemma}
\newtheorem{definition}[theorem]{Definition}
\newtheorem{corollary}[theorem]{Corollary}
\newtheorem{remark}{Remark}[section]
\newtheorem{proposition}[theorem]{Proposition}

\newtheorem{assumption}{Assumption}
\newtheorem{assumptionalt}{Assumption}[assumption]
\newenvironment{assumptionp}[1]{
\renewcommand\theassumptionalt{#1}
\assumptionalt
}{\endassumptionalt}

\DeclareMathOperator*{\esssup}{ess\,sup}

\usepackage{geometry}
\usepackage{xcolor}
\usepackage{amsfonts}
\usepackage{lipsum}                     
\usepackage{xargs}         
\usepackage[colorinlistoftodos,prependcaption,textsize=tiny, textwidth=2cm]{todonotes} 

\usepackage{cleveref}
\crefname{assumptionalt}{Assumption}{Assumptions}

\numberwithin{equation}{section}

\newcommand{\ome}{\omega}

\newcommand{\MF}{\mathcal{F}}

\newcommand{\tf}{\widetilde{f}}

\begin{document}
\title[Reflected BSDEs with rough drivers]{Reflected backward stochastic differential equations with rough drivers}

\author[H. Li]{Hanwu Li}\address{Research Center for Mathematics and Interdisciplinary Sciences, Shandong University, Qingdao 266237, China; Frontiers Science Center for Nonlinear Expectations (Ministry of Education), Shandong University, Qingdao 266237, China}
\email{lihanwu@sdu.edu.cn}

\author[H. Zhang]{Huilin Zhang}\address{Research Center for Mathematics and Interdisciplinary Sciences, Frontiers Science Center for Nonlinear Expectations (Ministry of Education), Shandong University, Qingdao 266237, China; Institute of Mathematics, Humboldt University, Berlin, 10099, Germany}
\email{huilinzhang@sdu.edu.cn}

\author[K. Zhang]{Kuan Zhang}\address{Research Center for Mathematics and Interdisciplinary Sciences, Shandong University, Qingdao 266237, China}
\email{201911819@mail.sdu.edu.cn}

\begin{abstract}
In this paper, we investigate reflected backward stochastic differential equations driven by rough paths (rough RBSDEs), which can be viewed as probabilistic representations of nonlinear rough partial differential equations (rough PDEs) or stochastic partial differential equations (SPDEs) with obstacles. Furthermore, we demonstrate that solutions to rough RBSDEs solve the corresponding optimal stopping problems within a rough framework. This development provides effective and practical tools for pricing American options in the context of the rough volatility model, thus playing a crucial role in advancing the understanding and application of option pricing in complex market regimes. 

The well-posedness of rough RBSDEs is established using a variant of the Doss-Sussman transformation. Moreover, we show that rough RBSDEs can be approximated by a sequence of penalized BSDEs with rough drivers. For applications, we first develop the viscosity solution theory for rough PDEs with obstacles via rough RBSDEs. Second, we solve the corresponding optimal stopping problem and establish its connection with an American option pricing problem in the rough path setting.

\end{abstract}



\maketitle
\tableofcontents

\section{Introduction}
\subsection{Background and motivation}

In this paper we focus on the reflected backward stochastic differential equations (RBSDEs) involving an extra rough integral $d\mathbf{X}_t.$ Specifically, for $p\ge2,$ we are interested in the equation
\begin{equation}\label{e:reflec BSDE}
\left\{\begin{aligned}
&Y_{t} = \xi + \int_{t}^{T} f(r,S_r,Y_r,Z_r) dr + \int_{t}^{T} H(S_{r},Y_{r}) d\mathbf{X}_r - \int_{t}^{T} Z_{r} dW_r + K_{T} - K_{t},   \\
&Y_t\ge L_{t},\ t\in[0,T], \textrm{ and } \int_{0}^{T} (Y_{r} - L_{r}) dK_{r} = 0,
\end{aligned}\right.
\end{equation}
where 
\begin{equation*}
\int_{t}^{T}H(S_{r},Y_{r})d\mathbf{X}_{r} = \sum_{i=1}^{l} \int_{t}^{T} H_{i}(S_{r},Y_{r})d\mathbf{X}^{i}_{r} ,
\end{equation*} 
the path $\mathbf{X}$ is a so called $p$-variation {\it geometric rough path}, and the diffusion process $S$ satisfies
\begin{equation*}
S_{t} = x + \int_{0}^{t}b(r,S_r)dr + \int_{0}^{t}\sigma(r,S_r) dW_r,\quad t\in[0,T].
\end{equation*}
In addition, we focus on the case where $f$ exhibits quadratic growth in $z.$ When $H \equiv 0,$ Eq.~\eqref{e:reflec BSDE} degenerates to RBSDEs with quadratic growth, as studied by, e.g., \cite{bayraktar2012quadratic,KLQT,lepeltier2007reflected,sun2022quantitative}.\\

The nonlinear backward stochastic differential equation (BSDE), introduced by Pardoux-Peng \cite{pardoux1990adapted}, has been a famous area of research in stochastic analysis, with significant applications in various fields such as finance, control theory, and mathematical physics.
The BSDE has the following form:
\begin{equation}\label{e:BSDE}
Y_{t} = \xi + \int_{t}^{T} f(r,Y_r,Z_r) dr -\int_{t}^{T} Z_r dW_r,\quad t\in[0,T],
\end{equation}
where $\xi$ is $\mathcal{F}_{T}$-measurable, and the solution of \eqref{e:BSDE} is a pair of adapted processes $(Y,Z),$ with $Z$ term (vaguely speaking) designed to cancel the anticipating part of $\xi+ \int f(...) dr$ so that $Y$ is adapted. The concept of RBSDEs, first introduced by El Karoui et al. \cite{elkaroui1997reflected}, has extended this theory to incorporate constraints represented by the reflection term. Compared with classical BSDEs, the RBSDE requires an additional adapted increasing process $K_t$ such that
\begin{equation}\label{e:rBSDE}
\left\{\begin{aligned}
&Y_{t} = \xi + \int_{t}^{T}f(r,Y_r,Z_r)dr - \int_{t}^{T}Z_r dW_r + K_{T} - K_{t},\\
&Y_{t}\ge L_t, \ t\in[0,T], \textrm{ and } \int_{0}^{T} (Y_{r} - L_{r}) dK_{r} = 0,
\end{aligned}\right.
\end{equation}
where the obstacle $L$ is continuous, adapted and $L_{T}\le \xi.$ \\

Our interest in RBSDE~\eqref{e:reflec BSDE} is motivated by the following two related problems.

(a) The first motivation is to solve the following obstacle problem for semilinear rough partial differential equations (rough PDEs):
\begin{equation}\label{e:rough PDE}
\left\{\begin{aligned}
&d u(t,x) + \mathcal{L}u(t,x)dt + H(x,u(t,x))d\mathbf{X}_t = 0,\\
&u(T,x) = g(x) ,\quad u(t,x)\ge l(t,x),\quad t\in[0,T],
\end{aligned}\right.
\end{equation}
where
\begin{equation*}
\mathcal{L}u(t,x) = \frac{1}{2}\text{Tr}\left\{\sigma\sigma^{\top}(t,x)D^{2}_{xx} u(t,x) \right\} + b^{\top}(t,x)D_{x} u(t,x) + f\big(t,x,u(t,x), \sigma^{\top}(t,x)D_{x} u(t,x)\big).
\end{equation*}
Rough PDEs, which can be viewed both as an extension of classical PDEs with singular drivers and as a pathwise interpretation of SPDEs, have been studied by various authors in different forms; see, e.g., \cite{caruana2009partial,deya2019priori,deya2012non,diehl2017stochastic,hocquet2018energy}. However, with reflections or obstacles, research on such rough differential equations is very limited. Some related topics can be found in Aida \cite{AIDA20153570} and Deya et al. \cite{DEYA20193261}, who study the reflected rough differential equation in the deterministic case. Nevertheless, rough PDEs with reflections in the infinite-dimensional case remain a widely open research area. If the driver $\mathbf{X}$ is a smooth path, it is well understood that the solution $u(t,x)$ to \eqref{e:rough PDE} can be represented as a solution to an RBSDE via the so-called nonlinear Feynman-Kac formula. In the same spirit, when $\mathbf{X}$ is rough, one expects an equivalent relationship between \eqref{e:rough PDE} and rough RBSDEs (see Section 5.1). On the other hand, when the $d\mathbf{X}_t$ term is a pathwise realization of a Brownian integral, \eqref{e:rough PDE} corresponds to an obstacle problem for SPDEs, which has been well studied by Matoussi-Stoica \cite{matoussi2010obstacle}. Thus, the well-posedness of our rough RBSDE, and consequently of \eqref{e:rough PDE}, would nontrivially generalize and solve the obstacle problem for SPDEs driven by rough stochastic noises, such as fractional Brownian motions or strong Markov processes that can be lifted to rough signals (see, e.g., \cite[Chapter~15]{friz2010multidimensional}). \\

(b) The second motivation originates from the American option pricing under rough volatility. Consider a financial market model with two assets $P^{i},i=0,1$. $P^0$ is the price of a riskless bond satisfying
\begin{equation}\label{e:bound}
dP^{0}_{t} = P^{0}_{t} r_{t} dt,
\end{equation}
where $r$ is the short rate. The other represents the price of a stock with dynamics given by 
$$
dP^{1}_{t} = P^{1}_{t} \left(b_{t} dt + \rho v_{t} dW_{t} +  \sqrt{1-\rho^{2}}v_{t} dB_{t} \right),
$$
where $B$ and $W$ are independent Brownian motions, $\rho\in[-1,1],$ and $v_{t} = h(t,\hat{B}_t)$ with $\hat{B}_t$ the Riemann-Liouville fractional Brownian motion (fBm) derived from $B$. The rough Stein–Stein model, the rough Bergomi model, and the rough Heston model can all be represented by the above formulas (see, e.g., Bayer et al. \cite[Example~8.1--Example~8.3]{Christian-Asymptotic}) and they describe the logarithm of realized variance as behaving like a fBm (see, e.g., Jim et al. \cite{Jim-VolatilityIsRough}).

Now we consider a slightly more general case where $(\rho v_{t},\sqrt{1-\rho^2} v_{t})$ is replaced by $(\sigma_{t},\lambda_{t}).$ According to the recent work \cite{bank2023rough} by Bank et al, one may \textit{lift} the term $\sigma^{-1}_{t}\lambda_{t} dB_t$ to a rough driver $d\mathbf{X}_t,$ and thus the above stock price is equivalent to 
\begin{equation}\label{e:stock}
dP^{1}_{t} = P^{1}_{t} \left(b_{t} dt + \sigma_{t} dW_{t} +  \lambda_{t} dB_t \right) = P^{1}_{t} \left(b_{t} dt + \sigma_{t} dW_{t} +  \sigma_{t} d\mathbf{X}_t \right).
\end{equation}
The main advantage of the rough path framework is that it overcomes the problem arising from the non-Markovian nature of the volatility $v_t$, as shown in \cite{bank2023rough} for the European option pricing.

When $\mathbf{X}$ is smooth (say, with bounded variation), El Karoui-Quenez \cite{el1997non} associate the minimal superhedging strategy (and thus the American option price) with an optimal stopping problem, and they solve this problem by analyzing the RBSDE with the reflection given by the payoff. Then, in the rough case, by extending this argument, the American option pricing problem is naturally related to the rough RBSDE \eqref{e:reflec BSDE} (for more details, see Section \ref{subsec:Optimal Stopping}).\\

The theory of rough paths originates from Lyons' seminal work \cite{Lyons1998}. It has been developed by several excellent researchers over the last two decades; see, e.g., \cite{FH20} and references therein. The theory of rough SDEs has proven to be a powerful tool in areas of  SPDEs, nonlinear filtering, and dynamical systems; see, e.g.,  \cite{allan2020pathwise,crisan2013robust,diehl2017control}. To our best knowledge, Diehl-Friz \cite{DF} are the first to introduce rough BSDEs and show their well-posedness by applying a continuity argument. Diehl-Zhang \cite{diehl2017young} consider rough BSDEs with Young drivers ($p$-variation with $p\in[1,2)$) and prove their well-posedness in a strong sense. Very recently, Liang-Tang \cite{liang2023multidimensional} extend the result of \cite{DF} to a multidimensional case and find a strong, unique solution to rough BSDE when the rough integral is linear. For the reflected rough differential equation, except \cite{AIDA20153570,DEYA20193261} mentioned in our motivation, we may refer to \cite{gassiat2021non,richard2020penalisation} and the references therein for more recent works. To the best of our knowledge, this is the first research on RBSDEs driven by rough paths and, consequently, rough PDEs with obstacles.   \\

The contributions of this paper are threefold. First, we systematically explore the characteristics and properties of rough RBSDEs, including the stability, the comparison theorem, and the construction and approximation of solutions via penalization. This penalization argument can be viewed as a numerical approximation for nonlinear rough PDEs or SPDEs with obstacles. Second, as applications, we relate rough RBSDEs to the obstacle problem within the rough path framework, and then we derive a nonlinear Feynman-Kac formula for the corresponding rough PDEs. Third, we delve into an optimal stopping problem that involves RBSDEs and rough BSDEs, and we relate this problem to an American option pricing problem.\\

The main idea for the well-posedness part, developed from \cite{DF}, is to approximate the rough path via {\it{lifted}} smooth paths. Then we solve the approximate RBSDE via  a modified Doss-Sussman transformation.  However, by following this idea, we will encounter fundamental problems when dealing with the penalization approximation and applications for rough RBSDEs, since the convergence of the solution via the prescribed approximation is rather weak. More precisely, the modified Doss-Sussman transformation fails to provide the continuous dependence of reflections $K_t$ on the rough driver, and thus the convergence of $K_t$ does not follow from the convergence of $(Y_t,Z_t).$ In addition, it is unclear if the continuity obtained by the approximation argument is consistent with the penalized approximation or the viscosity solution theory with obstacles. These two problems stem from the fact that the rough integral $\int_{0}^{t}H(...)d\mathbf{X}_{r}$ exists solely in a formal sense, while RBSDE theory relies heavily on It\^{o}'s theory to derive the convergence of $\int_{0}^{t}H(S_{r},Y^{m}_r)d\mathbf{X}_{r}$ when $Y^{m}\rightarrow Y$ uniformly. Moreover, for the obstacle problem involving rough PDEs, the mere pointwise convergence of the approximated solution is not sufficient for the existence of solutions to such rough PDEs. Indeed, the locally uniform convergence of viscosity solutions to the limit should be considered.\\

We handle these problems separately. For the trouble caused by the definition of solutions, we rely on the monotonic property of the diffeomorphism induced by the Doss-Sussman transformation, as well as the continuous property of the obstacle. Indeed, these two properties make the convergence of reflections equivalent to the weak convergence of corresponding measures by analyzing the uniform convergence of cumulative distribution functions. For the similar trouble encountered in the penalization argument, inspired by Zhang \cite[Section~6.3.3]{zhang2017backward}, we transform the problem for convergence of penalization sequences into the study of continuous dependence of rough RBSDEs on varying obstacles, and we demonstrate the convergence of these obstacles using the Doss-Sussman transformation. That is, we treat the sequence of penalized rough BSDEs as a new kind of rough RBSDEs, whose obstacle is the smaller value between the original obstacle and the penalization solution. Then, we employ the continuity and differentiability of the Doss-Sussman transformation to study the trajectory semi-continuity of the limits of penalization solutions. This semi-continuity provides us with the a.s. uniform convergence of the new obstacles. Consequently, the desired convergence of the penalization sequence follows directly from the stability of rough RBSDEs. For the well-posedness of rough PDEs with obstacles, motivated by \cite{DF}, to avoid the lack of uniform convergence, we employ a so-called ``semi-relaxed limits" technique, which is usually used to establish the invariance of viscosity solutions under limits. This technique offers a tool to establish a Dini-type argument, thereby allowing us to derive the uniform convergence of viscosity solutions from the equivalence of the relaxed upper and lower limits. Furthermore, the equivalence of relaxed limits is a consequence of the comparison theorem for viscosity solutions of transformed PDEs, as outlined in \cite[Section~6]{crandall1992user}.
\\

\subsection{Structure of the paper}

The paper is organized as follows. In Section \ref{sec:prelim}, we recall some notations and facts on the geometric rough path
theory. In Section~\ref{sec:rough reflected BSDE}, we establish the well-posedness of rough RBSDEs and introduce some properties of our equations, including the continuity of solution map, the comparison result, and the penalization approximation. In Section~\ref{subsec:PDE obstacle}, we study  the obstacle problem for rough PDEs. In Section~\ref{subsec:Optimal Stopping}, we investigate the optimal stopping problem related to Eq.~\eqref{e:reflec BSDE}, as well as the analogous problem arising from the American option pricing problem. Finally, in \ref{append:A} we recall some concepts and results of RBSDEs and rough BSDEs, and in \ref{append:B} we give some facts about viscosity solutions for obstacle problems and prove the comparison theorem that will be used in Section~\ref{subsec:PDE obstacle}.

\section{Preliminaries}\label{sec:prelim}

\subsection{Notations and basic definitions}
Let $l,m\in\mathbb{N},$ and $(E,d_{E})$ be a complete metric space, and they may vary in the context. Denote the Euclidean norm by $|\cdot|,$ of which the dimension depends on the context. $T>0$ is a fixed time horizon. Let  $\left\{W_t\right\}_{t\in[0,T]}$ be a $d$-dimensional standard Brownian motion in $(\Omega,\mathcal{F},(\mathcal{F}_{t})_{t\in[0,T]}, \mathbb{P})$ with $(\mathcal{F}_{t})_{t\in[0,T]}$ being the augmented filtration of $W.$ \\

{\it Stochastic norms: } For any $k\in[1,\infty)$ and $[s,t]\subset [0,T],$ denote by 

\begin{itemize}
    \item $\mathrm{H}^{k}_{[s,t]}( \R^{l})$ the space of progressively measurable processes $X$ on $[s,t]$ with values in $\R^{l}$, equipped with the norm
$$
\|X\|_{\mathrm{H}^{k};[s,t]}:= \left\{\E\left[\left|\int_s^t|X_r|^2 dr\right|^{\frac{k}{2}}\right]\right\}^{\frac1k};
$$

\item $L^{k}(\MF_t,\R^{l})$ the space of random variables $\xi\in\MF_{t}$ with values in $\R^{l}$, equipped with the norm 
$$\|\xi\|_{L^{k}} := \left\{\E\left[|\xi|^k\right]\right\}^{\frac1k};$$

\item $\mathrm{H}^{\infty}_{[s,t]}(\R^{l})$ the subspace of $\mathrm{H}^{k}_{[s,t]}( \R^{l})$, which has finite essential supremum norm 
$$
\|Y\|_{\mathrm{H}^{\infty};[s,t]}:=\esssup\limits_{\ome\in\Omega,r\in[s,t]}|Y_r(\ome)|;
$$

\item $L^{\infty}(\MF_t,\R^{l})$ the space of all $\xi\in L^{k}(\MF_t,\R^{l})$ such that $\|\xi\|_{L^\infty}<\infty$, where
$$
\|\xi\|_{L^{\infty}}:=\esssup\limits_{\omega\in\Omega}|\xi(\ome)|;
$$

\item $\mathrm{I}^k_{[s,t]}$ the space of continuous increasing adapted processes $K$ on $[s,t]$ with $K_s=0$ and $K_t \in L^k(\MF_t, \R)$.

\end{itemize}

In the following we write $\mathrm{H}^{k}_{[s,t]} $ and $L^k(\MF_t)$ as short for $\mathrm{H}^{k}_{[s,t]}(\R^{l})$ and $L^k(\MF_t,\R^{l})$ if $l=1.$ 
\\

{\it Spaces of functions: } For any open set $U\subset \R^{m},$ $k,k_1,k_2\in\mathbb{N},$ and $\lambda>1,$ denote by
\begin{itemize}
    \item $C(U,E)$ (resp. $C_{b}(U,E)$) the space of continuous (bounded and continuous) functions $g:U\rightarrow E$;\\

    \item $\|\cdot\|_{\infty;U}$ the uniform norm for $g\in C_b(U,\R^l),$ which is defined by $$\|g\|_{\infty;U} := \sup\limits_{r\in U}|g(r)|;$$
    
    
    \item $C^{k}(\R^m,\R^l)$ the space of $k$-th continuously differentiable functions $g\in C(\R^m,\R^l)$.\\
    
    \item $C^{k_1,k_2}(\R\times \R^m,\R^l)$ the space of functions $g\in C(\R\times \R^m,\R^l)$ such that for each $i=1,2,...,k_1$ and $|\alpha|_{\text{index}} \le k_{2},$ $\partial^{i}_{t}\partial^{\alpha}_{x}g(t,x)$ exists and is continuous,  where $\alpha\in\mathbb{N}^{m},$ $|\alpha|_{\text{index}} := \sum\limits_{i=1}^{m} \alpha_{i},$ and $\partial^{\alpha}_{x} g := \frac{\partial^{|\alpha|_{\text{index}}}g}{\partial x_{1}^{\alpha_{1}}\partial x_{2}^{\alpha_{2}}...\partial x_{m}^{\alpha_{m}}}$;\\

    \item $C^{k_{1},k_{2}}([0,T]\times \R^m, \R^l)$ the space of functions $g\in C([0,T]\times\R^m,\R^l),$ which is defined in the same manner as $C^{k_1,k_2}(\R\times \R^m,\R^l),$ requiring that the derivatives continuous on $[0,T]\times \R^m;$ 
    \\
    \item $D_{x}g$ and $D^2_{xx}g$ the gradient and Hessian matrix of $g\in C^{1}(\R^m,\R)$ and $g\in C^{2}(\R^m,\R)$ respectively, and denote $D^{2}_{xy}g:=D_{x}(D_{y} g)$ and $D^{3}_{xxy}g:=D^2_{xx}(D_{y} g)$ for $g\in C^{1,1}(\R^m\times \R,\R)$ and $g\in C^{2,1}(\R^m \times \R,\R)$ respectively;\\

    \item $\text{Lip}^1(U,\R^{l})$ the space of functions $V\in C(U,\R^{l})$ such that $\|V\|_{\text{Lip}^{1}(U,\R^l)}<\infty,$ where
    $$\|V\|_{\text{Lip}^1(U,\R^{l})}:=\sup\limits_{x^1,x^2\in U:x^1\neq x^2}\frac{|V(x^1)-V(x^2)|}{|x^1-x^2|} + \|V(\cdot)\|_{\infty;U};$$
    \item $\text{Lip}^{\lambda}(U,\R^{l})$ the space of functions $V\in C(U,\R^{l})$ such that $\|V\|_{\text{Lip}^{\lambda}(U,\R^{l})}<\infty,$ where
    $$\|V\|_{\text{Lip}^{\lambda}(U,\R^{l})}:=\sum\limits_{|\alpha|_{\text{index}} = \lfloor \lambda \rfloor}\sup\limits_{x^1,x^2\in U:x^1\neq x^2}\frac{|\partial^{\alpha}_{x}V(x^1)-\partial^{\alpha}_{x}V(x^2)|}{|x^1-x^2|^{\lambda - \lfloor \lambda \rfloor}} + \sum\limits_{|\beta|_{\text{index}}\le \lfloor \lambda \rfloor} \|\partial^{\beta}_{x}V(\cdot)\|_{\infty;U},$$  $\lfloor \lambda \rfloor$ represent the largest integer strictly smaller than $\lambda,$ and $\alpha,\beta\in\mathbb{N}^{m}$;\\

    \item $\text{Lip}^{1}([0,T],\R^l)$ and  $\text{Lip}^{\lambda}([0,T],\R^l)$ the space of $V\in C([0,T],\R^l),$ which is defined in the same manner as $\text{Lip}^{1}(U,\R^l)$ and $\text{Lip}^{\lambda}(U,\R^l)$ respectively;\\

    \item $\text{Lip}_{\text{loc}}^{1}(\R^m,\R^l)$ (resp. $\text{Lip}_{\text{loc}}^{\lambda}(\R^m,\R^l)$) the local $1$-Lipschitz ($\lambda$-Lipschitz) continuous space, which is the set of all functions $V\in C(\R^m,\R^l)$ such that $V\in \text{Lip}^{1}(U,\R^l)$ ($\text{Lip}^{\lambda}(U,\R^l)$) for every bounded open subset  $U\subset \R^m$.
\end{itemize}

For $g\in C^{1}([0,T],\R)$ ($g\in C^{1}(\R,\R).$ resp.) we also write  $\dot{g}_t := D_t g(t)$  ($g'(y) := D_y g(y)$ resp.). \\

{\it The $p$-variation metric space and geometric rough path space: } 
For $p\ge 1,$ let $[p]$ be the integer part of $p$ and denote by
\begin{itemize}
    \item $C^{p\text{-var}}([0,T],E)$ the space of functions $f\in C([0,T],E)$ with $\|f\|_{p\text{-var};[0,T]}<\infty,$ where
    $$\|f\|_{p\text{-var};[0,T]} := \sup\limits_{\pi\in \mathcal{P}[0,T]}\left\{\sum\limits_{[t_{i},t_{i+1}]\in\pi}|d_{E}( f_{t_i},f_{t_{i+1}})|^{p}\right\}^{\frac{1}{p}},$$ $\mathcal{P}[0,T]$ is the set of all partitions of the interval $[0,T],$ and $|\pi|$ is the mesh size of a partition $\pi \in \mathcal{P}[0,T]$;\\
    
    \item $G^{[p]}(\R^{l})\subset \oplus_{k=0}^{[ p]}(\R^{l})^{\otimes k}$ the free nilpotent group of step $[p]$ over $\R^{l},$ which is equipped with the product $*$ defined in \cite[Chapter~7.3.2]{friz2010multidimensional} ($\otimes$ in the literature) and the Carnot-Caratheodory metric $d(\cdot,\cdot)$ defined in \cite[Chapter~7.5.2]{friz2010multidimensional};\\

    \item $C^{p\text{-var}}([0,T],G^{[p]}(\R^l))$ the $p$-variation weak geometric rough path space, equipped with the $p$-variation metric\\
    $d_{p\text{-var};[0,T]}(\mathbf{X}^{1},\mathbf{X}^{2}) := \sup\limits_{\pi\in\mathcal{P}[0,T]}\left\{\sum\limits_{[t_{i},t_{i+1}]\in\pi}|d(\mathbf{X}^{1}_{t_{i},t_{i+1}}, \mathbf{X}^{2}_{t_{i},t_{i+1}})|\right\}^{\frac{1}{p}},$ where $\mathbf{X}_{s,t}:=\mathbf{X}^{-1}_{s}* \mathbf{X}_{t}$;
    \\
    
    \item $C^{0,p\text{-var}}([0,T],G^{[p]}(\R^{l}))$ the $p$-variation geometric rough path space, which is the closure of $C^{1\text{-var}}([0,T],G^{[p]}(\R^{l}))$ under $d_{p\text{-var};[0,T]}.$\\
\end{itemize}

\begin{remark}\label{e:weak geometric vs. geometric}
Referring to \cite[Theorem~8.13]{friz2010multidimensional}, $\left(C^{p\text{-var}}([0,T],G^{[p]}(\R^{l})),d_{p\text{-var};[0,T]}\right)$ forms a complete, non-separable metric space. Comparatively, \cite[Proposition~8.25]{friz2010multidimensional} shows that the metric space $\left(C^{0,p\text{-var}}([0,T],G^{[p]}(\R^{l})),d_{p\text{-var};[0,T]}\right)$ forms a polish space.
\end{remark}

Recall the definition of the lift of bounded variation path given by \cite[Definition~7.2]{friz2010multidimensional}.
\begin{definition}\label{def:lift}
The step-$[p]$ signature of $\mathrm{X}\in C^{1\text{-var}}([0,T],\R^{l})$ is given by
\begin{equation*}
\begin{aligned}
S_{[p]}(\mathrm{X})_{0,T} &:= \left(1,\int_{0<u<T}d\mathrm{X}_{u},\int_{0<u_{1}<u_{2}<T}d\mathrm{X}_{u_1}\otimes d\mathrm{X}_{u_2},\cdots,\int_{0<u_1<...<u_{[p]}<T}d\mathrm{X}_{u_{1}}\otimes\cdots \otimes \mathrm{X}_{u_{[p]}}\right)\\
&\in G^{[p]}(\R^{l}).
\end{aligned}
\end{equation*}
The path $\mathbf{X}_{t} := S_{[p]}(\mathrm{X})_{0,t}$ is called the (step-$[p]$) lift of $\mathrm{X}.$
\end{definition}

\begin{remark}\label{rem:geo rough path}
For $\mathbf{X}\in C^{0,p\text{-var}}([0,T],G^{[p]}(\R^{l})),$ according to \cite[Lemma~8.21]{friz2010multidimensional} we can find $\mathrm{X}^{n}\in \bigcap\limits_{k\ge 1}\text{Lip}^{k}([0,T],\R^{l})$ (i.e., $\mathrm{X}^{n}$ are smooth in time) such that the lifts of $\mathrm{X}^{n},$ denoted by $\mathbf{X}^{n},$ converges to $\mathbf{X}$ in $C^{p\text{-var}}([0,T],G^{[p]}(\R^{l})).$
\end{remark}

\bigskip

\subsection{Property of rough flows}
Throughout this article, we make the following assumption: 
\begin{assumptionp}{$(H_{p,\gamma})$}\label{(Hpr)}
Given a natural number $l\ge 1$ and real numbers $\gamma>p\ge 2$, let $H(x,\cdot) := (H_1(x,\cdot),...,H_l(x,\cdot))$ be a collection of vector fields on $\mathbb{R},$ parameterized by $x\in\mathbb{R}^{d}.$ There exists a constant $C_H>0$ such that 
$$|H|_{\rm{Lip}^{\gamma+2}(\mathbb{R}^{d+1},\mathbb{R}^l)}\le C_H.$$
\end{assumptionp}

Under Assumption~\eqref{(Hpr)}, according to \cite[Chapter~11]{friz2010multidimensional}, given any geometric rough path $\mathbf{X}\in C^{0,p\text{-var}}([0,T],G^{[p]}(\R^l))$ and terminal time $T'\in[0,T],$ the following RDE admits a unique solution $\phi^{T'}(t,x,y):$
\begin{equation}\label{e:flow}
\phi^{T'}(t,x,y) = y + \int_{t}^{T'} H(x,\phi^{T'}(r,x,y))d\mathbf{X}_r,\quad t\in[0,T'].
\end{equation} 
In addition, $\phi^{T'}:[0, T'] \times \R^d \times \mathbb{R} \rightarrow \mathbb{R}$ is a continuous $C^{3}$-diffeomorphism flow on $\R$. We denote the space of such functions by $\mathcal{D}_{3}([0,T']\times\R^d;\R),$ i.e.,
\begin{align*}
\mathcal{D}_{3}([0,T']\times\R^d;\R):=\Big\{ 
\phi \Big| \
\forall (t,x) \in[0, T']\times\R^d, & \ \alpha \leq 3: D^\alpha_{y}\phi (t,x,y), D^\alpha_{y} \phi^{-1}(t,x,y)\\
&\text { are continuous on } [0,T'] \times \R^{d+1}
 \Big\}.
\end{align*}
Moreover, the inverse of the map $y\mapsto\phi^{T'}(t,x,y)$ satisfies the following RDE (or ODE when $\mathbf{X}$ is a lift of a smooth path):
\begin{equation}\label{e:phi inverse}
(\phi^{T'})^{-1}(t,x,y)=y-\int_t^{T'} D_y(\phi^{T'})^{-1}(r,x,y)H(x,y) d\mathbf{X}_r,\quad t\in[0,T'].
\end{equation}

In view of \cite[Appendix B]{DF}, we obtain the following estimates for $\phi^{T'}$ and its derivatives. 

\begin{lemma}\label{lem:boundedness of the flow}
Suppose \ref{(Hpr)} holds and $\mathbf{X}\in C^{0,p\text{-var}}([0,T],G^{[p]}(\R^{l}))$ with 
\[\|\mathbf{X}\|_{p\text{-var};[0,T]}\le C_{\mathbf{X}},\text{ for some constant }C_{\mathbf{X}}>0.\]
Then, the solution to \eqref{e:flow} $\phi^{T'}$ belongs to $ \mathcal{D}_{3}([0,T']\times\R^d;\R).$ Moreover, there exists a positive constant $C = C(C_H,C_{\mathbf{X}},T)$ independent of $T'$ such that, uniformly over $t\in[0,T'],$ $x\in\R^d,$ and $y\in\mathbb{R}$ 
\begin{align*}\max\Bigg\{\left|D_{x}\phi^{T'}\right|,\left|D_{y}\phi^{T'}\right|,&\left|\frac{1}{D_y \phi^{T'}}\right|,\left|D^{2}_{yy}\phi^{T'}\right|,\left|D^{2}_{xy}\phi^{T'}\right|,\left|D^{2}_{xx}\phi^{T'}\right|,\\ &\left|D^{3}_{yyy}\phi^{T'}\right|,\left|D^{3}_{xyy}\phi^{T'}\right|,\left|D^{3}_{xxy}\phi^{T'}\right|\Bigg\}\le C,\end{align*}
where the derivatives of $\phi^{T'}$ always take values at $(t,x,y).$ Moreover, for every $\varepsilon>0,$ there exists a positive constant $\delta = \delta(\varepsilon,C_{H},C_{\mathbf{X}})$ independent of $T'$ such that whenever $\|\mathbf{X}\|_{p\text{-var};[s,T']} \le \delta,$  we have for all $t\in[s,T'],$ $x\in\mathbb{R}^{d},$ and $y\in\mathbb{R},$ the following inequality holds
\begin{equation*}
\max\bigg\{\left|D_{x}\phi^{T'}\right|,\left|D_{y}\phi^{T'} - 1\right|,\left|D^{2}_{yy}\phi^{T'}\right|,\left|D^{2}_{xy}\phi^{T'}\right|,\left|D^{2}_{xx}\phi^{T'}\right|,\left|D^{3}_{yyy}\phi^{T'}\right|,\left|D^{3}_{xyy}\phi^{T'}\right|,\left|D^{3}_{xxy}\phi^{T'}\right|\bigg\}\le \varepsilon.
\end{equation*}
\end{lemma}

\begin{remark}\label{rem:flow}
According to \cite[Proposition 11.11]{friz2010multidimensional} we also have for $(t,x,y)\in[0,T']\times\R^{d}\times\R$
\begin{equation*}
\begin{aligned}
\max\bigg\{\left|D_{x}(\phi^{T'})^{-1}\right|,\left|D_{y}(\phi^{T'})^{-1}\right|,&\left|D^{2}_{yy}(\phi^{T'})^{-1}\right|,\left|D^{2}_{xy}(\phi^{T'})^{-1}\right|,\left|D^{2}_{xx}(\phi^{T'})^{-1}\right|,\left|D^{3}_{yyy}(\phi^{T'})^{-1}\right|,\\
&\left|D^{3}_{xyy}(\phi^{T'})^{-1}\right|,\left|D^{3}_{xxy}(\phi^{T'})^{-1}\right|\bigg\}\le C(C_{H},C_{\mathbf{X}},T),
\end{aligned}
\end{equation*}
where the derivatives of $(\phi^{T'})^{-1}$ take values at $(t,x,y).$
\end{remark}

We also have the following stability result for Eq.~\eqref{e:flow}, which is referred to \cite[Lemma~B.2]{DF} and \cite[Theorem~11.13]{friz2010multidimensional}.

\begin{lemma}\label{lem:stability of the flow}
Under the same setting as Lemma \ref{lem:boundedness of the flow}, let $\left\{\mathbf{X}^n,\ n\ge 1\right\}$ be a sequence of $p$-variation geometric rough paths that converges to $\mathbf{X}$ in $p$-variation norm. Denote by $\phi^{n,T'}$ the solution of Eq.~\eqref{e:flow} driven by $\mathbf{X}^n.$ Then locally uniformly on $[0,T']\times\mathbb{R}^{d}\times\mathbb{R}$, i.e. for any compact set $U \subset \R^{d+1},$ the following convergence holds uniformly on $[0,T'] \times U,$
\begin{equation*}
\begin{aligned}
&\left(\phi^{n,T'},D_{x}\phi^{n,T'},D_{y}\phi^{n,T'},\frac{1}{D_y \phi^{n,T'}},D^{2}_{yy}\phi^{n,T'},D^{2}_{xy}\phi^{n,T'},D^{2}_{xx}\phi^{n,T'}\right)\\
&\ \ \ \ \ \  \longrightarrow \left(\phi^{T'},D_{x}\phi^{T'},D_{y}\phi^{T'},\frac{1}{D_y \phi^{T'}},D^{2}_{yy}\phi^{T'},D^{2}_{xy}\phi^{T'},D^{2}_{xx}\phi^{T'}\right).
\end{aligned}
\end{equation*}

In addition, let 
$\psi^{n,T'}(t,x,y) := (\phi^{n,T'})^{-1}(t,x,y)$ and $\psi^{T'}(t,x,y) := (\phi^{T'})^{-1}(t,x,y).$ Then the above locally uniform convergence also holds with $(\phi^{n,T'}, \phi^{T'})$ replaced by $(\psi^{n,T'}, \psi^{T'})$.

\end{lemma}

The following lemma will be useful for performing induction from one subinterval to another.

\begin{lemma}\label{lem:length of the subinterval}
Let $X\in C^{\text{p\text{-var}}}([0,T],E).$ For any fixed $\delta>0,$ set $N:=\left(\left[\frac{\|X\|_{p\text{-var};[0,T]}}{\delta}\right] + 1\right)^{p}.$ Then, there exists a set of positive real numbers $\left\{h^{(i)},i = 0,...,N\right\}$ satisfying 
\begin{equation*}
h^{(i)}\le h^{(i+1)},\quad i=0,\dots,N-1, \quad h^{(N)} = T, \quad h^{(0)} = 0,
\end{equation*}
and 
\begin{equation*}
\|X\|_{p\text{-var};[T-h^{(i)},T-h^{(i-1)}]}\le \delta,\quad \text{for all }i=1,2,...,N.
\end{equation*}
\end{lemma}

\begin{proof}
Fix $\delta>0.$ Let $h^{(0)} := 0$ and
\begin{equation*}
h^{(i)} := \sup\left\{h\in[h^{(i-1)},T];\|X\|_{p\text{-var};[T-h,T-h^{(i-1)}]}\le \delta\right\},\quad i=1,2,...,N.
\end{equation*}
The reader may easily prove that the set $\left\{h^{(i)},i = 0,...,N\right\}$ satisfies the desired conditions of the lemma by the superadditivity  $\|X\|^{p}_{p\text{-var};[0,T]}\ge \sum\limits_{i=1}^{N}\|X\|^{p}_{p\text{-var};[T-h^{(i)},T-h^{(i-1)}]}.$
\end{proof}

\section{Rough reflected BSDEs}\label{sec:rough reflected BSDE}

In this section, we mainly consider the following rough RBSDE driven by $\mathbf{X},$
\begin{equation}\label{e:BSDE main result}
\left\{\begin{aligned}
&Y_{t} = \xi + \int_{t}^{T} f(r,S_r,Y_r,Z_r) dr + \int_{t}^{T} H(S_{r},Y_{r}) d\mathbf{X}_r - \int_{t}^{T} Z_{r} dW_r + K_{T} - K_{t} ,	\\
&Y_t\ge L_{t},\ t\in[0,T] \textrm{ and } \int_{0}^{T} (Y_{r} - L_{r}) dK_{r} = 0,
\end{aligned}\right.
\end{equation}
where the process $S$ satisfies
\begin{equation}\label{e:diffusion process}
S_{t} = x + \int_{0}^{t}b(r,S_{r})dr + \int_{0}^{t}\sigma(r,S_{r})dW_{r},\ t\in[0,T].
\end{equation}
In the rest of our paper, we always assume that $Z$ takes values in the $d$-dimensional row vector.

Throughout this section, we assume that $f:\Omega\times[0,T]\times\R^{d}\times\R\times\R^{d}\rightarrow\R$ is $\mathcal{P}\otimes\mathcal{B}(\R^d)\otimes\mathcal{B}(\R)\otimes\mathcal{B}(\R^d)$-measurable, where $\mathcal{P}$ denotes the predicable $\sigma$-algebra and $\mathcal{B}$ denotes the Borel $\sigma$-algebra. Let $C_0>0$ be a global constant and $c(\cdot)\in C(\R^{+},\R^{+})$ be a strictly positive and increasing continuous function with $\int_0^\infty \frac{1}{c(y)}dy=\infty.$ We propose the following assumption for the function $f$.
\begin{assumptionp}{($H_f$)}\label{(H0)}
\end{assumptionp}
\begin{itemize}
 \item[(A1)] $\mathbb{P}$-a.s., for all $(t,x,y,z)\in[0,T]\times\mathbb{R}^d\times \mathbb{R}\times\mathbb{R}^d$,
\begin{equation*}
|f(t,x,y,z)|\leq c(|y|)+C_0|z|^2.
\end{equation*}
\item[(A2)] For all $M>0,$ there exists a constant $C_{M}>0$ such that $\mathbb{P}$-a.s, 
\begin{equation*}
|D_z f(t,x,y,z)|\leq C_{M}(1+|z|),\text{ for all }(t,x,y,z)\in[0,T]\times\mathbb{R}^d\times [-M,M]\times\mathbb{R}^d.
\end{equation*}
\item[(A3)] For all $M>0$ and $\varepsilon>0,$ there exists a constant $C_{\varepsilon,M}>0$ such that $\mathbb{P}$-a.s.,
\[D_y f(t,x,y,z)\leq C_{\varepsilon,M} + \varepsilon |z|^{2}, \text{ for all }(t,x,y,z)\in [0, T ]\times \mathbb{R}^d \times[-M,M]\times\mathbb{R}^d.\]
\item[(A4)] For each $m\ge 1,$ there exists a random variable $C'_m (\omega)>0$ such that $\mathbb{P}$-a.s.,
\begin{equation*}
|D_y f(t,x,y,z)|\le C'_m, \text{ for all }|(t,x,y,z)|\le m .
\end{equation*}

\end{itemize}

\begin{remark}\label{rem:remark of H0} 

Regarding the above conditions, we provide the following explanations.
\begin{itemize}
 
\item[(i).]  Conditions~(A1) and (A3) are slightly more general compared to the corresponding ones in \cite{DF}.
For example, letting $f(t,x,y,z) = y \log(1+|y|) + |z|^2$, then $f$ satisfies (A1). Indeed, by choosing $c(y) = C + (1+y)\log(1+y)$ with some $C>0$, (A1) holds by applying Young's inequality. For Assumption (A3), $D_y f$ here is assumed to have quadratic growth in $z$ while in \cite[(F2)]{DF} they assume  that $D_{y}f$ has a constant upper bound. 

\item[(ii).]
Condition (A3) is slightly different from \cite[(H3)]{KLQT}. Here $c(\cdot)$ is required to be increasing while (H3) in \cite{KLQT} only requires it to be strictly positive. The additional monotonicity of $c(\cdot)$ ensures that there exists a constant $C>0$ such that 
\begin{equation*}
c(|\phi^{T'}(t,x,y)|)\le C\cdot c(C |y|),
\end{equation*}
which will provide an upper bound convenient for the proof of Lemma~\ref{lem40}.

\item[(iii).]
Here, we assume that the partial derivatives of $f$ exist. However, as stated in \cite{DF}, the existence of partial derivatives is not necessary. This requirement can be replaced by the assumption that $f$ is locally Lipschitz continuous in $(y,z)$, and by bounding the Lipschitz constant accordingly.  

 \item[(iv).]
Compared with \cite{KLQT}, we require an additional condition $(A4)$, which assumes the local boundedness of $D_y f.$ This technical condition will be used in Lemma~\ref{lem:locally uniform conver} to ensure that, a.s.,
\begin{equation*}
f(t,x,y_n,z) \rightarrow f(t,x,y,z),\text{ locally uniformly in }(t,x,z) \text{ as }y_n\rightarrow y,
\end{equation*}
which will play an essential role in the proof of the well-posedness of rough RBSDEs. 

\end{itemize}
\end{remark}

Now, we will introduce the assumptions on $b(\cdot)$ and  $\sigma(\cdot)$ appearing in  Eq.~\eqref{e:diffusion process}.

\begin{assumptionp}{($H_{b,\sigma}$)}\label{(H_b,sigma)}
Assume that $b: \Omega\times[0,T]\times\R^{d}\rightarrow\R^{d}$ and $\sigma:\Omega\times [0,T]\times\R^{d}\rightarrow\R^{d\times d}$ are progressively measurable and satisfy that 
\begin{equation}\label{e:b,sigma}
|D_{x} b(t,x)|\vee | b(t,x)|\vee|D_{x} \sigma(t,x) | \vee | \sigma(t,x)|\le C_{S}, \text{ for all } (\omega,t,x)\in \Omega\times [0,T]\times\R^{d},
\end{equation}
where $C_{S}>0$ is a constant.
\end{assumptionp}

For the terminal value $\xi$ and the obstacle process $L,$ we assume the following assumption.
\begin{assumptionp}{($H_{L,\xi}$)}\label{(H_L,xi)}
Assume that $\xi\in L^{\infty}(\mathcal{F}_{T}),$ and $L:\Omega\times [0,T]\rightarrow \R$ is a bounded, continuous, and adapted process with $L_{T}\le \xi.$ Assume further that 
\begin{equation*}
\|\xi\|_{L^{\infty}}\le C_{\xi},\quad  \|L\|_{\mathrm{H}^{\infty};[0,T]}\le C_{L},
\end{equation*}
where $C_{\xi},C_{L}>0$ are constants.
\end{assumptionp}

\subsection{Well-posedness of rough RBSDEs}\label{subsec:well-posedness}

Firstly, we consider the case where the driver is $d\mathrm{X}_r,$ with  $\mathrm{X}\in\bigcap_{\lambda\ge 1}\text{Lip}^{\lambda}([0,T],\R^{l}).$ And we call such $\mathrm{X}$ a smooth path. Then, the equation is written as follows 
\begin{equation}\label{58}
\left\{\begin{array}{l}
Y_t=\xi+\int_t^T f(r,S_r,Y_r,Z_r)dr+\int_t^T H(S_r,Y_r) \dot{\mathrm{X}}_r dr -\int_t^T Z_r dW_r + K_T-K_t,  \\
\\
Y_t \ge L_t,\ t\in[0,T] \textrm{ and }\int_{0}^{T}(Y_{r} - L_{r})dK_{r} = 0, 
\end{array}\right.
\end{equation}
where the diffusing process $S$ is the solution to Eq.~\eqref{e:diffusion process}.
Here, $H$ satisfies Assumption~\ref{(Hpr)}, $(L,\xi)$ satisfies Assumption~\ref{(H_L,xi)}, $f,$ $b$ and $ \sigma$ satisfy Assumption~\ref{(H0)} and \ref{(H_b,sigma)} respectively. 

\begin{definition}\label{def:smooth case}
A triplet of processes $(Y,Z,K) \in \mathrm{H}^{\infty}_{[0,T]} \times \mathrm{H}^{2}_{[0,T]}(\R^d)\times \mathrm{I}^{1}_{[0,T]}$ is called a solution to the reflected BSDE with parameter $(\xi,f,L,H,\mathrm{X})$ on $[0,T]$ if \eqref{58} holds.
\end{definition}

\begin{remark}\label{rem:well-posedness for smooth case}
Suppose Assumption~\ref{(Hpr)}, \ref{(H0)}, \ref{(H_b,sigma)}, and \ref{(H_L,xi)} hold.  Assume $\mathrm{X}$ is a smooth path. Then Eq.~\eqref{58} with parameter $(\xi,f,L,H,\mathrm{X})$ admits a unique solution in the sense of Definition~\ref{def:smooth case}, which is a consequence of \cite[Theorem~1]{KLQT} and \cite[Corollary~1]{KLQT}.
\end{remark}

\begin{lemma}[Smooth case]\label{lem:smmoth case}
Suppose Assumption~\ref{(Hpr)}, \ref{(H0)}, \ref{(H_b,sigma)}, and \ref{(H_L,xi)} hold. Assume $\mathrm{X}$ is smooth. Let $\phi^{T}(t,x,y)$ be the solution of Eq.~\eqref{e:flow} with $T' = T,$ and $(Y,Z,K)$ be the unique solution to the reflected BSDE with parameter $(\xi,f,L,H,\mathrm{X}).$ Then, the triplet $(\widetilde{Y},\widetilde{Z},\widetilde{K})$ defined by
\begin{equation}\label{e1'}\begin{split}
&\widetilde{Y}_t:=(\phi^{T})^{-1}(t,S_t,Y_t), \  \widetilde{K}_t:=\int_0^t \frac{1}{D_y \phi^{T}(r,S_r,\widetilde{Y}_r)}dK_r,\\
&\widetilde{Z}_t:=-\frac{(D_x \phi^{T}(t,S_t,\widetilde{Y}_t))^{\top}}{D_y\phi^{T}(t,S_t,\widetilde{Y}_t)}\sigma_t+ \frac{1}{D_y \phi^{T}(t,S_t,\widetilde{Y}_t)}Z_t,\ t\in[0,T],
\end{split}\end{equation}
is the solution of the reflected BSDE with parameter $(\xi, \tf, \widetilde{L}, 0,0),$ where (throughout, $\phi^{T}$ and all its derivatives will always be evaluated at $(t,x,\widetilde{y}),$ and $b,\sigma$ will always be evaluated at $(t,x)$) $\widetilde{L}_t:=(\phi^{T})^{-1}(t,S_t,L_t),$ and 
\begin{align*}
\widetilde{f}(t,x,\widetilde{y},\widetilde{z}):=\frac{1}{D_y \phi^{T}}\Big\{f(t,x,\phi^{T},D_y \phi^{T} \widetilde{z}&+ \sigma^{\top}D_x \phi^{T})+ (D_x \phi^{T})^{\top} b+\frac{1}{2}\text{Tr}\left\{D^{2}_{xx}\phi^{T} \sigma\sigma^{\top}\right\} \\ 
&+(D^{2}_{xy}\phi^{T})^{\top} \sigma \widetilde{z}+\frac{1}{2}D^{2}_{yy}\phi^{T} |\widetilde{z}|^2\Big\}.
\end{align*}
\end{lemma}

\begin{proof}
Applying It\^{o}'s formula to $\psi(t,S_t,Y_t),$ where $\psi:=(\phi^{T})^{-1}$ is the $y$-inverse of $\phi^{T}(t,x,y)$ satisfying Eq.~\eqref{e:phi inverse}, we have
\begin{align*}
\psi(t,S_t,Y_t)=&\xi+\int_t^T \bigg[ D_y \psi(\theta_r)f(r,Y_r,Z_r) - (D_x \psi(\theta_r))^{\top}b_r-\frac{1}{2}\text{Tr}\left\{D^{2}_{xx}\psi(\theta_r)\sigma_r\sigma^{\top}_r\right\}\\
&\ \ \ \ \ \ - \frac{1}{2}D^{2}_{yy} \psi(\theta_{r})|Z_r|^2 - (D^{2}_{xy}\psi(\theta_r))^{\top} \sigma_r Z^{\top}_r \bigg]dr\\
&\ \ \ \ \ \  - \int_t^T \left[(D_x\psi(\theta_r))^{\top}\sigma_r+D_y\psi(\theta_r)Z_r\right]dW_r +\int_t^T D_y\psi(\theta_r) dK_r,
\end{align*}
where $\theta_r:=(r,S_r,Y_r),$ $b_r := b(r,S_r),$ and $\sigma_r := \sigma(r,S_r).$ Let $(\widetilde{Y},\widetilde{Z},\widetilde{K})$ be given in \eqref{e1'}. Noting that $\psi(t,x,y)$ is continuously strictly increasing in $y$ and $(Y,Z,K)$ solves \eqref{58}, we have 
\begin{itemize}
\item[(i)] $\widetilde{L}_{T} = L_{T}\le \xi,$
\item[(ii)] $\widetilde{K}_t$ is increasing,
\item[(iii)] $\widetilde{Y}_t\geq \widetilde{L}_t$, a.s., and  $Y_{t} = L_{t} \Leftrightarrow \widetilde{Y}_t = \widetilde{L}_{t}.$
\end{itemize}
It follows that
\begin{equation}\label{e:Skorohod tilde Y}
\int_0^T \mathbf{1}_{\{\widetilde{Y}_{t} = \widetilde{L}_{t}\}}(t) d\widetilde{K}_t = \int_0^T \frac{\mathbf{1}_{\{Y_{t} = L_{t}\}}(t)}{D_y \phi(t,S_t,\widetilde{Y}_t)} dK_t = 0,
\end{equation}
which implies that $\int_{0}^{T}(\widetilde{Y}_{t} - \widetilde{L}_{t})d\widetilde{K}_{t} = 0.$ Therefore, $(\widetilde{Y},\widetilde{Z},\widetilde{K})$ is the solution of the reflected BSDE with parameter $(\xi,\widetilde{f},\widetilde{L},0,0).$ 
\end{proof}

The objective is to provide the definition of the solution to Eq.~\eqref{e:BSDE main result} driven by a geometric rough path $\mathbf{X}.$ Our strategy is to study the transformed RBSDEs with parameter $(\xi,\widetilde{f},\widetilde{L},0,0)$ suggested by the Doss-Sussman transformation. Therefore, we need the following property of $\widetilde{f}(\cdot)$ to ensure that $\widetilde{f}(\cdot)$ fulfills Assumption~\ref{(H1)} for the well-posedness of classical RBSDEs.

\begin{lemma}\label{lem40}
Suppose Assumption \ref{(Hpr)}, \ref{(H0)}, and \ref{(H_b,sigma)} hold. Let $\mathbf{X}\in C^{0,p\text{-var}}([0,T],G^{[p]}(\mathbb{R}^{l})).$ Let $\phi^{T'}(t,x,y)$ be the unique solution to Eq.~\eqref{e:flow}. Then, for the function
\begin{align*}
\widetilde{f}^{T'}(t,x,\widetilde{y},\widetilde{z}):=\frac{1}{D_y \phi^{T'}}\Bigg\{f\big(t,x,\phi^{T'},D_y \phi^{T'} \widetilde{z}& + \sigma^{\top} D_x \phi^{T'} \big) + (D_x \phi^{T'})^{\top} b+\frac{1}{2}\text{Tr}\left\{D^{2}_{xx}\phi^{T'} \sigma\sigma^{\top}\right\}\\
& + (D^{2}_{xy}\phi^{T'})^{\top} \sigma \widetilde{z}+\frac{1}{2}D^{2}_{yy}\phi^{T'} |\widetilde{z}|^2\Bigg\},
\end{align*}
there exists an increasing positive continuous function $\widetilde{c}(\cdot):\R\rightarrow \R$  
depending on $\|\mathbf{X}\|_{p\text{-var};[0,T']}$ such that $\int_{0}^{\infty}\frac{dy}{\widetilde{c}(y)} = \infty$ and 
\begin{itemize}
\item[(i).]  there exists a constant $\widetilde{C}_0 > 0$ depending on 
 $\|\mathbf{X}\|_{p\text{-var};[0,T']}$ such that a.s.,
\begin{align*}
|\widetilde{f}^{T'}(t,x,\widetilde{y},\widetilde{z})|\leq \widetilde{c}(|\widetilde{y}|) + \widetilde{C}_{0}|\widetilde{z}|^2,\text{ for all }(t,x,\widetilde{y},\widetilde{z})\in[0,T']\times \R^{d}\times\R\times\R^{d};
\end{align*}

\item[(ii).] for every $M>0,$ there exists a constant $\widetilde{C}_{M}>0$ depending on 
$\|\mathbf{X}\|_{p\text{-var};[0,T']}$ such that a.s.,
\begin{equation*}
|D_{\widetilde{z}}\widetilde{f}^{T'}(t,x,\widetilde{y},\widetilde{z})|\leq \widetilde{C}_{M}(1+|\widetilde{z}|),\text{ for all }(t,x,\widetilde{y},\widetilde{z})\in[0,T']\times \R^{d}\times[-M,M]\times\R^{d};
\end{equation*}

\item[(iii).] for every $\varepsilon,M>0,$ there exist constants $\delta_{\varepsilon,M},\widetilde{C}_{\varepsilon,M}>0$ depending on 
 $\|\mathbf{X}\|_{p\text{-var};[0,T']}$ such that whenever  $s\in[0,T')$ and
$\|\mathbf{X}\|_{p\text{-var};[s,T']}\le \delta_{\varepsilon,M}$,  we have $a.s.,$ 
\[D_{\widetilde{y}}\widetilde{f}^{T'}(t,x,\widetilde{y},\widetilde{z})\leq \widetilde{C}_{\varepsilon,M}+\varepsilon|\widetilde{z}|^2,\text{ for all }(t,x,\widetilde{y},\widetilde{z})\in[s,T']\times\R^{d}\times[-M,M]\times\R^{d}.\]
\end{itemize}

\end{lemma}

\begin{proof}
In view of \cite[Lemma 7]{DF}, it remains to show the first inequality in $(i)$. By Assumption~\ref{(H0)} we have  
\begin{equation*}
\begin{aligned}
\frac{1}{D_{y}\phi^{T'}}\left|f\big(t,x,\phi^{T'},D_y \phi^{T'} \widetilde{z} +  \sigma^{\top}D_x \phi^{T'}\big)\right|&\le C'c(|\phi^{T'}|) + C'C_0 \left|D_y \phi^{T'} \widetilde{z} + \sigma^{\top} D_x \phi^{T'} \right|^{2}\\
& \le C'c(C''|\widetilde{y}| + C''' ) + C'C_0 \left|D_y \phi^{T'} \widetilde{z} + \sigma^{\top} D_x \phi^{T'}\right|^{2},
\end{aligned}
\end{equation*}
where (the $\sup_{(t,x,\widetilde{y})}$ takes over $[0,T']\times \R^d \times \R$)
\begin{equation*}
C' := \sup_{(t,x,\widetilde{y})}\left|\frac{1}{D_{y}\phi^{T'}(t,x,\widetilde{y})}\right|,\ C'' := \sup_{(t,x,\widetilde{y})}|D_{y}\phi^{T'}(t,x,\widetilde{y})|,\  C''' := \sup\limits_{(t,x)}|\phi^{T'}(t,x,0)|.
\end{equation*}
By Lemma~\ref{lem:boundedness of the flow}, $C'$ and $C''$ are finite and depending on $\|\mathbf{X}\|_{p\text{-var};[0,T']}$. In view of \cite[Theorem~10.14]{friz2010multidimensional} and Assumption~\ref{(Hpr)}, $C'''$ is finite as well. Again by Lemma~\ref{lem:boundedness of the flow} and Assumption~\ref{(H_b,sigma)}, there exist constants $\hat{C}$ and $\widetilde{C}_{0}$ such that
\begin{equation*}
\begin{aligned}
C'C_0 \left|D_y \phi^{T'} \widetilde{z} + \sigma^{\top} D_x \phi^{T'} \right|^{2} & + \frac{1}{D_{y}\phi^{T'}}\Big|(D_{x}\phi^{T'})^{\top}b + \frac{1}{2}\text{Tr}\left\{D^{2}_{xx}\phi^{T'}\sigma\sigma^{\top}\right\}  \\
& \quad + (D^{2}_{xy}\phi^{T'})^{\top}\sigma\widetilde{z} + \frac{1}{2}D^{2}_{yy}\phi^{T'}|\widetilde{z}|^2\Big|\\
&\le  \hat{C} + \widetilde{C}_0|\widetilde{z}|^2.
\end{aligned}
\end{equation*}
Thus, letting $\widetilde{c}(\widetilde{y}) := C' c(C''|\widetilde{y}| + C''') + \hat{C},$ the desired 
bound follows.
\end{proof}

The following lemma demonstrates the continuity property of $\widetilde{f}^{T'}(t,x,\widetilde{y},\widetilde{z})$ defined above with respect to $\mathbf{X}\in C^{0,p\text{-var}}([0,T],G^{[p]}(\R^l)).$ In light of the continuity of derivatives of $\phi^{T'}$ (Lemma~\ref{lem:stability of the flow}), it suffices to discuss the continuity of
\[\mathbf{X}\mapsto f^{T'}(t,x,\phi^{T'}(t,x,\widetilde{y}),D_{y}\phi^{T'}(t,x,\widetilde{y})\widetilde{z} + \sigma^{\top}(t,x)D_x\phi^{T'}(t,x,\widetilde{y}) ).\]

\begin{lemma}\label{lem:locally uniform conver}
Let $F:[0,T]\times\R^d\times\R\times \R^d\rightarrow \R.$ Assume that $\sigma:[0,T]\times\R^d \rightarrow \R^{d\times d}$ satisfies \ref{(H_b,sigma)}, and $\mathbf{X}^{n}\rightarrow\mathbf{X}^{0}$ in $C^{0,p\text{-var}}([0,T],G^{[p]}(\R^l)).$ For $T'\in[0,T]$ let $\phi^{n,T'}(t,x,y)$ be the solution of Eq.~\eqref{e:flow} driven by $\mathbf{X}^n.$ For $m\ge 1$ set 
\[O_{m}:=\left\{(t,x,y,z)\in [0,T]\times\R^d\times\R\times \R^d,\ |(t,x,y,z)|< m\right\}.\]
Assume further that for each $m\ge 1,$ there exists $C_m>0$ such that for all $(t,x,y,z)\in O_{m}$,
\begin{equation*}
|D_y F(t,x,y,z)|\vee |D_z F(t,x,y,z)|\le C_m.
\end{equation*}
For $n\in\mathbb{N}$, set 
\begin{equation*}
F^n(t,x,y,z) := F(t,x,\phi^{n,T'}(t,x,y),D_y\phi^{n,T'}(t,x,y)z + \sigma^{\top}(t,x) D_x \phi^{n,T'}(t,x,y)).
\end{equation*}
Then, we have that locally uniformly on $[0,T']\times\R^d\times\R\times\R^d$,
\begin{equation*}
F^{n}(t,x,y,z) \rightarrow F^{0}(t,x,y,z).
\end{equation*}
\end{lemma}

\begin{proof}
Note that by Lemma~\ref{lem:boundedness of the flow} and the boundedness of $\sigma$ as stated in \eqref{e:b,sigma}, 
\[\Theta:=(t,x,\phi^{n,T'}(t,x,y),D_y\phi^{n,T'}(t,x,y)z + \sigma^{\top}(t,x)D_{x}\phi^{n,T'}(t,x,y))\] 
is bounded uniformly in $n\in\mathbb{N}$ whenever $(t,x,y,z)$ is bounded. Suppose that both $(t,x,y,z),\Theta\in O_{m},$ then there exists a constant $C'_m > 0$ such that
\begin{equation*}
\begin{aligned}
&|F^n(t,x,y,z) - F^0(t,x,y,z)|\\
&\le C'_{m}\Big(|\phi^{n,T'}(t,x,y) - \phi^{0,T'}(t,x,y)|  + |D_{y}\phi^{n,T'}(t,x,y) - D_{y}\phi^{0,T'}(t,x,y)||z|\\
&\quad\quad\quad\quad  + |D_{x}\phi^{n,T'}(t,x,y) - D_{x}\phi^{0,T'}(t,x,y)||\sigma(t,x)|\Big).
\end{aligned}
\end{equation*}
Thus the locally uniform convergence of $F^n$ follows from Lemma~\ref{lem:stability of the flow} and the boundedness of $\sigma.$ 
\end{proof}

In the following we will present a helpful property of increasing functions, which implies that the pointwise convergence of continuous increasing functions can be strengthened to uniform convergence. This property will play an essential role in the proof of the well-posedness of rough RBSDEs.

\begin{lemma}\label{lem:K}
Let $K^{n}:[0,T]\rightarrow \R$, $n \ge 1,$ and $ K\in C([0,T],\R)$ be positive increasing functions such that $K^n_{0} = K_{0} = 0.$ Assume that $\lim_{n}K^{n}_{t} = K_{t}$ for every $t\in[0,T].$ Let $y_t\in C([0,T],\R)$ be a positive function. Then it follows that
\begin{equation*}
\int_{0}^{t}y_s dK^{n}_{s} \rightarrow \int_{0}^{t}y_s dK_{s}\text{ uniformly in }t\in[0,T]. 
\end{equation*}
\end{lemma}

\begin{proof}
Denote $t_0 := \sup\{t\in[0,T];K_{t}=0\}.$
Since the equality $K_{t_0} - K_{0} = 0$ leads to that
\begin{equation*}
\left|\int_{0}^{t}y_{s}d(K^{n} - K)_s\right| = \left|\int_{0}^{t}y_{s}dK^{n}_s\right| \le \|y\|_{\infty;[0,t_0]}|K^{n}_{t_0} - K^{n}_0|,
\end{equation*}
we have as $n\rightarrow\infty$
\begin{equation}\label{e:Kn -> K; [0,t0]}
\int_{0}^{t}y_{s}dK^{n}_{s} \rightarrow \int_{0}^{t}y_s dK_{s}\text{ uniformly in }t\in[0,t_0].
\end{equation}
For $t\in(t_0,T],$ from the convergence $\lim_n K^{n}_{\cdot} = K_{\cdot},$ it follows that for all $r\in[t_0,t],$
\begin{equation*}
(K^{n}_{r} - K^{n}_{t_0}) \rightarrow (K_{r} - K_{t_0} ) > 0\text{ and } \int_{t_0}^{r}\frac{dK^{n}_s}{K^{n}_t  - K^{n}_{t_0}} \rightarrow \int_{t_0}^{r}\frac{dK_s}{K_t  - K_{t_0}}.
\end{equation*}
Thus, we have that 
\begin{equation}\label{e:weak conver}
dK^{n}_s / (K^{n}_t - K^{n}_{t_0}) \rightarrow dK_s / (K_t  - K_{t_0}),\ \text{weakly in probability measure on }[t_0,t].
\end{equation}
Since $y$ is continuous, by the definition of weak convergence in measure we have
\begin{equation}\label{e:y/K}
\int_{t_0}^{t}\frac{y_s dK^{n}_s}{K^{n}_t  - K^{n}_{t_0}} \rightarrow \int_{t_0}^{t}\frac{y_s dK_s}{K_t  - K_{t_0}},\text{ for every }t\in(t_0,T].
\end{equation}
Let 
\begin{equation*}
\hat{K}^{n}_{t}:=\int_{t_0}^{t}y_s dK^{n}_s,\ \hat{K}_{t} := \int_{t_0}^{t} y_s dK_s,\ t\in[t_0,T].
\end{equation*}
According to the convergence $\lim_n (K^{n}_t - K^{n}_{t_0}) = K_{t} - K_{t_0}$ and \eqref{e:y/K}, we have 
\begin{equation}\label{e:hat Kn -> hat K}
\hat{K}^{n}_{t}\rightarrow \hat{K}_{t},\text{ for every }t\in(t_0,T].
\end{equation}
Furthermore, by the positivity of $y,$ $\hat{K}^n$ and $\hat{K}$ are positive increasing functions with $\hat{K}^n_{t_0} = \hat{K}_{t_0} = 0.$ And by the continuity of $K,$ $\hat{K}$ is also continuous.  Set $t_0' := \sup\{t\in[t_0,T];\hat{K}_{t} = 0\}.$ Then by repeating the proof of \eqref{e:Kn -> K; [0,t0]} with $(K^n, K)$ replaced by $(\hat{K}^n,\hat{K}),$ and letting the integrand be $1,$ we have that 
\begin{equation}\label{e:y/K'}
\int_{t_0}^{t}y_{s}dK^{n}_{s}  \rightarrow \int_{t_0}^{t}y_s dK_{s}\text{ uniformly in }t\in[t_0,t'_0].
\end{equation}
By repeating the proof of \eqref{e:weak conver} with $t = T,$ we have
\begin{equation}\label{e:weak conver'}
d\hat{K}^{n}_s / (\hat{K}^{n}_T - \hat{K}^{n}_{t'_0}) \rightarrow d\hat{K}_s / (\hat{K}_T  - \hat{K}_{t'_0}), \text{ weakly in probability measure on }[t'_0,T].
\end{equation}
Note that the cumulative distribution function (CDF) of $d\hat{K}_s / (\hat{K}_T  - \hat{K}_{t'_0})$ is continuous. By applying \cite[Chapter~10, Exercise~5.2]{parzen1960modern} to \eqref{e:weak conver'}, the weak convergence of $d\hat{K}^n_s/(\hat{K}^n_{T} - \hat{K}^{n}_{t'_{0}})$ implies the uniform convergence of CDFs, i.e.,
\begin{equation}\label{e:y/K+}
\int_{t'_0}^{t}\frac{y_{s}dK^{n}_{s}}{\hat{K}^{n}_{T} - \hat{K}^{n}_{t'_0}}\rightarrow \int_{t'_0}^{t}\frac{y_{s}dK_{s}}{\hat{K}_{T} - \hat{K}_{t'_0}},\text{ uniformly in }t\in[t'_{0},T].
\end{equation}
According to \eqref{e:hat Kn -> hat K}, $\hat{K}^{n}_{T} - \hat{K}^{n}_{t_0'}\rightarrow \hat{K}_{T} - \hat{K}_{t'_{0}},$ and then by using \eqref{e:y/K+} we have
\begin{equation}\label{e:y/K''}
\int_{t'_0}^{t}y_{s}dK^{n}_{s}\rightarrow \int_{t'_0}^{t}y_{s}dK_{s},\text{ uniformly in }t\in[t'_{0},T].
\end{equation}
By \eqref{e:Kn -> K; [0,t0]},\eqref{e:y/K'}, and \eqref{e:y/K''}, the proof is complete.
\end{proof}

Now we are ready to give the proof of our first main result.

\begin{theorem}\label{thm:well-posedness}
Assume \ref{(Hpr)}, \ref{(H0)}, \ref{(H_b,sigma)}, and \ref{(H_L,xi)} hold. Let $\mathbf{X}\in C^{0,p\text{-var}}([0,T],G^{[p]}(\R^{l}))$ and let $\mathrm{X}^{n}\in\bigcap_{\lambda\ge 1}\text{Lip}^{\lambda}([0,T],\R^l)$ for each $n \ge 1.$ Let $\mathbf{X}^{n}$ be the lift of $\mathrm{X}^{n}$, and $(Y^n,Z^n,K^n)$ be the unique solutions to Eq.~\eqref{e:BSDE main result} with parameter $(\xi,f,L,H,\mathrm{X}^{n}).$ If $\mathbf{X}^{n} \rightarrow \mathbf{X}$ in $C^{0,p\text{-var}}([0,T],G^{[p]}(\mathbb{R}^l)),$ then there exists a triplet $(Y,Z,K)\in \mathrm{H}^{\infty}_{[0,T]}\times \mathrm{H}^{2}_{[0,T]}(\mathbb{R}^d) \times \mathrm{I}^{1}_{[0,T]}$ such that
\begin{equation*}
(Y^{n}_{\cdot},K^{n}_{\cdot}) \rightarrow (Y_{\cdot},K_{\cdot})\  \text{ uniformly on }\ [0,T]\  \text{ a.s., }\ 
Z^n \rightarrow Z \ \text{ in }\ \mathrm{H}^{2}_{[0,T]}(\mathbb{R}^d),
\end{equation*}
and 
\[\int_{0}^{T}(Y_{r} - L_{r})dK_{r} = 0.\]

In addition, the triplet $(Y,Z,K)$ is unique in the sense that it does not depend on the choice of sequence $\{\mathrm{X}^{n},n\ge 1\}.$
In this case we call this $(Y,Z,K)$ the unique solution of Eq.~\eqref{e:BSDE main result} with parameter $(\xi,f,L,H,\mathbf{X}).$ 
\end{theorem}

\begin{proof} 

\textbf{Step 1.} We first show that there exists a constant $M$ independent of $n$, such that 
\begin{align*}
    \esssup_{(t,\omega)\in[0,T]\times\Omega}|Y^{n}_t|\le M.
\end{align*}
In the following, we will denote by $(Y^0,Z^0,K^0)$ and $\mathbf{X}^0$ the solution $(Y,Z,K)$ and the rough path $\mathbf{X},$ respectively. For each $n\ge 0,$ let $\phi^{n,T'}(t,x,y)$  be the following solution flow parameterized by $x\in\R^{d}$ with terminal time $T'\in[0,T]$
\begin{equation}\label{e:phi^n}
\phi^{n,T'}(t,x,y)=y+\int_t^{T'} H(x,\phi^n(r,x,y))d\mathrm{X}^n_r,\quad t\in[0,T'].
\end{equation}
Due to the smoothness of $\mathrm{X}^{n}$ for $n\geq 1,$ the above equation is an ordinary differential equation (ODE) when $n\ge 1$; and it is an RDE when $n=0.$ By Lemma~\ref{lem:boundedness of the flow}, for every $x\in\mathbb{R}$ $\phi^{n,T'}(t,x,\cdot)$ is a flow of $C^3$-diffeomorphisms in $\R.$ In addition, Lemma~\ref{lem:boundedness of the flow} shows that $\phi^{n,T'}(t,\cdot,\cdot)$ and its derivatives up to order three are bounded, and the same holds true for the inverse $y$-inverse $(\phi^{n,T'})^{-1}(t,\cdot,\cdot).$ Furthermore, by Lemma~\ref{lem:stability of the flow}, we have the following locally uniform convergence in $[0,T']\times \mathbb{R}^d\times\mathbb{R}$
\begin{equation}\label{63}
\begin{aligned}
&(\phi^{n,T'},\frac{1}{D_y\phi^{n,T'}},D_y\phi^{n,T'},D^{2}_{yy}\phi^{n,T'},D_x\phi^{n,T'},D^{2}_{xx}\phi^{n,T'},D^{2}_{xy}\phi^{n,T'})\\
&\longrightarrow
(\phi^{0,T'},\frac{1}{D_y\phi^{0,T'}},D_y\phi^{0,T'},D^{2}_{yy}\phi^{0,T'},D_x\phi^{0,T'},D^{2}_{xx}\phi^{0,T'},D^{2}_{xy}\phi^{0,T'}).
\end{aligned}
\end{equation}

For $n\ge 0,$ set
\begin{equation}\label{e:tilde f^n,T'}
\begin{aligned}
&\widetilde{f}^{n,T'}(t,x,\widetilde{y},\widetilde{z}):=\frac{1}{D_y \phi^{n,T'}}\Bigg\{f(t,x,\phi^{n,T'},D_y \phi^{n,T'} \widetilde{z} + \sigma^{\top} D_x \phi^{n,T'} ) + (D_x \phi^{n,T'} )^{\top}b \\
&\qquad\qquad\qquad\qquad\qquad+ \frac{1}{2}\text{Tr}\left\{D^{2}_{xx}\phi^{n,T'} \sigma\sigma^{\top}\right\} + (D^{2}_{xy}\phi^{n,T'})^{\top} \sigma \widetilde{z}+\frac{1}{2}D^{2}_{yy}\phi^{n,T'} |\widetilde{z}|^2\Bigg\},\\
&\widetilde{L}^{n,T'}_{t}:=(\phi^{n,T'})^{-1}(t,S_{t},L_{t}),\quad (t,x,\widetilde{y},\widetilde{z})\in[0,T']\times\R^{d}\times\R\times\R^{d},
\end{aligned}
\end{equation}
where $\phi^{n,T'}$ and all its derivatives are evaluated at $(t,x,\widetilde{y}),$ and $b,\sigma$ are evaluated at $(t,x).$ For $n\ge 1,$ let $(Y^{n},Z^{n},K^{n})$ be the solution of Eq.~\eqref{58} with terminal time $T$ and parameter $(\xi,f,L,H,\mathrm{X}^{n}).$ 
By Lemma~\ref{lem:smmoth case}, 
the triplet $(\widetilde{Y}^{n,T'},\widetilde{Z}^{n,T'},\widetilde{K}^{n,T'})$ solves Eq.~\eqref{58} with terminal time $T'$ and parameter $(Y^{n}_{T'},\widetilde{f}^{n,T'},\widetilde{L}^{n,T'},0,0),$ where 
\begin{equation}\label{e1}
\begin{aligned}
&\widetilde{Y}^{n,T'}_t:=(\phi^{n,T'})^{-1}(t,S_t,Y^n_t), \quad  \widetilde{K}^{n,T'}_t:=\int_0^t \frac{1}{D_y \phi^{n,T'}(s,S_s,\widetilde{Y}^{n,T'}_s)}dK^n_s,\\
&\widetilde{Z}^{n,T'}_t:=-\frac{(D_x \phi^{n,T'}(t,S_t,\widetilde{Y}^{n,T'}_t))^{\top}}{D_y\phi^{n,T'}(t,S_t,\widetilde{Y}^{n,T'}_t)}\sigma_t+ \frac{1}{D_y \phi^{n,T'}(t,S_t,\widetilde{Y}^{n,T'}_t)}Z^n_t,\quad t\in[0,T'].
\end{aligned}
\end{equation}
By Lemma~\ref{lem40}, there exist $\widetilde{C}_{0}>0$ and an increasing positive continuous function $\widetilde{c}(\cdot)$  such that a.s. 
\begin{equation}\label{e:tilde f <= 1 + |z|^2}
|\widetilde{f}^{n,T'}(t,x,\widetilde{y},\widetilde{z})|\le \widetilde{c}(|\widetilde{y}|) + \widetilde{C}_0 |\widetilde{z}|^2, \text{ for all }(t,x,\widetilde{y},\widetilde{z})\in [0,T']\times\R^d\times\R\times\R^d.
\end{equation}
In view of \eqref{e:tilde f <= 1 + |z|^2} and Lemma~\ref{lem:A Prior Estimate}, 
\begin{equation}\label{e:bounde of Yn}
\begin{aligned}
\esssup_{(t,\omega)\in[0,T']\times\Omega}|\widetilde{Y}^{n,T'}_t|&\le \widetilde{M},\text{ and then }
\esssup_{(t,\omega)\in[0,T]\times\Omega}|Y^{n}_t|\le M ,
\end{aligned}
\end{equation}
where $\widetilde{M} = \bar{M},$ which is given by estimate \eqref{e:prior estimate} with $c(\cdot)$ replaced by $\widetilde{c}(\cdot)$, and 
\begin{equation*}
M := \sup\limits_{n\ge 1}\left\{\widetilde{M}\|D_{y}\phi^{n,T}(\cdot,\cdot,\cdot)\|_{\infty;[0,T]\times\R^d\times\R} + \|\phi^{n,T}(\cdot,\cdot,0)\|_{\infty;[0,T]\times\R^d}\right\}.
\end{equation*}
By Lemma~\ref{lem:boundedness of the flow} and the uniform boundedness of $ \|\mathbf{X}^{n}\|_{p\text{-var};[0,T]},$ $M$ defined above is finite.    \\

\textbf{Step 2.} We show that $(Y^n,Z^n,K^n)$ converges to some $(Y^0,Z^0,K^0)$ on a small interval. Note that $\sup_{n}\|\mathbf{X}^{n}\|_{p\text{-var};[0,T]}<\infty$. By Lemma~\ref{lem40}, for every $\varepsilon,M>0$ there exist $\delta,\widetilde{C}>0$ independent of $n$ and $T'$, such that 
\begin{itemize}
\item for every $M>0$ there exits $\widetilde{C}_{M}>0$ such that a.s.
\begin{equation*}
|D_{\widetilde{z}}\widetilde{f}^{n,T'}(t,x,\widetilde{y},\widetilde{z})|\le \widetilde{C}_{M}(1+|\widetilde{z}|),\text{ for all }(t,x,\widetilde{y},\widetilde{z})\in[0,T']\times \R^d \times [-M,M] \times \R^d;
\end{equation*}

\item for every $\varepsilon,M>0,$ there exist $\delta_{\varepsilon,M},\widetilde{C}_{\varepsilon,M}>0$ such that whenever $s\in[0,T')$ and $\|\mathbf{X}^{n}\|_{p\text{-var};[s,T']}\le \delta_{\varepsilon,M}$, we have a.s.
\begin{equation*}
|D_{\widetilde{y}}\widetilde{f}^{n,T'}(t,x,\widetilde{y},\widetilde{z})|\le \widetilde{C}_{\varepsilon,M} + \varepsilon |\widetilde{z}|^2\text{ for all }(t,x,\widetilde{y},\widetilde{z})\in[s,T']\times \R^d \times [-M,M] \times\R^d.
\end{equation*}
\end{itemize}
Then, for each $T'\in(0,T]$ and $n\ge 1,$ in view of \eqref{e:tilde f <= 1 + |z|^2} and the above facts, $\widetilde{f}^{n,T'}$ satisfies the requirement of Proposition~\ref{prop:RBSDE willposedness}. Thus, by Proposition~\ref{prop:RBSDE willposedness}, there exists $\varepsilon>0$ such that for every $s\in[0,T')$ satisfying $\|\mathbf{X}^{n}\|_{p\text{-var};[s,T']}\le \delta_{\varepsilon,\widetilde{M}},$ Eq.~\eqref{58} with terminal time $T'$ and parameter $(Y^{n}_{T'},\widetilde{f}^{n,T'},\widetilde{L}^{n,T'},0,0)$ admits a unique solution $(\widetilde{Y}^{n,T'}_{\cdot},\widetilde{Z}^{n,T'}_{\cdot},\widetilde{K}^{n,T'}_{\cdot} - \widetilde{K}^{n,T'}_{s})$ on $[s,T']$. Now fix the $\varepsilon$ mentioned above, and denote $\delta := \delta_{\varepsilon,\widetilde{M}}$. Since $\lim_{n}d_{p\text{-var};[0,T]}(\mathbf{X}^{n}, \mathbf{X}^{0}) = 0,$ we can
assume without loss of generality that $\sup_{n\ge 1}d_{p\text{-var};[0,T]}(\mathbf{X}^{n},\mathbf{X}^{0})\le \frac{\delta}{2}$. By Lemma~\ref{lem:length of the subinterval}, there exists a set of increasing real numbers $\left\{h^{(i)},i=0,1,...,N\right\}$ such that $[T-h^{(N)},T-h^{(0)}] = [0,T]$ and 
\begin{equation*}
\|\mathbf{X}^{0}\|_{p\text{-var};[T-h^{(i+1)},T-h^{(i)}]}\le \frac{\delta}{2},
\end{equation*}
where $N$ only depends on $\|\mathbf{X}^{0}\|_{p\text{-var};[0,T]}$ and $\delta.$ Then, for $i=0,1,2,...,N-1,$ it follows that
\begin{equation*}
\sup_{n\ge 0}\|\mathbf{X}^{n}\|_{p\text{-var};[T-h^{(i+1)},T-h^{(i)}]}\le \delta.
\end{equation*}
Thus, the Eq.~\eqref{58} with terminal time $T-h^{(i)}$ and 
parameter $(Y^{n}_{T-h^{(i)}},\widetilde{f}^{n,T-h^{(i)}},\widetilde{L}^{n,T-h^{(i)}},0,0)$ admits a unique solution $(\widetilde{Y}^{n,T-h^{(i)}}_{\cdot},\widetilde{Z}^{n,T-h^{(i)}}_{\cdot},\widetilde{K}^{n,T-h^{(i)}}_{\cdot} - \widetilde{K}^{n,T-h^{(i)}}_{T-h^{(i+1)}})$ on the interval $[T-h^{(i+1)},T-h^{(i)}].$  

On the other hand, by Lemma~\ref{lem:locally uniform conver} with $F(t,x,y,z) := f(\omega,t,x,y,z)$ and Lemma~\ref{lem:stability of the flow}, $\widetilde{f}^{n,T}$
converges to $\widetilde{f}^{0,T}$ locally uniformly in $(t,x,\widetilde{y},\widetilde{z})$ a.s., and $\widetilde{L}^{n,T}$ converges to $\widetilde{L}^{0,T}$ uniformly in $t$ a.s. as $n\rightarrow \infty.$ 
In view of Lemma~\ref{lem:continu solution map} we have 
\begin{equation}\begin{split}\label{65}
&\widetilde{Y}^{n,T}_{\cdot}\rightarrow \widetilde{Y}^{0,T}, \  \widetilde{K}^{n,T}_{\cdot} - \widetilde{K}^{n,T}_{T-h^{(1)}}\rightarrow \widetilde{K}^{0,T}_{\cdot} - \widetilde{K}^{0,T}_{T-h^{(1)}} \textrm{ uniformly on } [T-h^{(1)},T] \textrm{ a.s.,}\\
&\widetilde{Z}^{n,T}\rightarrow \widetilde{Z}^{0,T} \textrm{ in } \mathrm{H}^{2}_{[T-h^{(1)},T]},
\end{split}\end{equation}
where the triplet $(\widetilde{Y}^{0,T},\widetilde{Z}^{0,T},\widetilde{K}^{0,T})$ uniquely solves Eq.~\eqref{58} with parameter $(\xi,\widetilde{f}^{0,T},\widetilde{L}^{0,T},0,0)$ on interval $[T-h^{(1)},T].$ For $t\in[T-h^{(1)},T]$, we define
\begin{equation}\label{e:inverse transformation}
\begin{aligned}
& Y^0_t:=\phi^{0,T}(t,S_t,\widetilde{Y}^{0,T}_t), \ K^{0,T}_t:=\int_{T-h^{(1)}}^t D_y \phi^{0,T}(s,S_s,\widetilde{Y}^{0,T}_s)d\widetilde{K}^{0,T}_s,\\ 
&  Z^0_t:=D_y\phi^{0,T}(t,S_t,\widetilde{Y}^{0,T}_t)\left[\widetilde{Z}^{0,T}_t+\frac{(D_x\phi^{0,T}(t,S_t,\widetilde{Y}^{0,T}_t))^{\top}}{D_y\phi^{0,T}(t,S_t,\widetilde{Y}^{0,T}_t)}\sigma_t\right].
\end{aligned}
\end{equation}
Since $\widetilde{Y}^{0,T}_t\geq \widetilde{L}^{0,T}_t$ for $t\in[T-h^{(1)},T],$ and $\int_{T-h^{(1)}}^T (\widetilde{Y}^{0,T}_t-\widetilde{L}^{0,T}_t)d\widetilde{K}^{0,T}_t=0,$ we have for $t\in[T-h^{(1)},T],$ $Y^0_t\geq L_t.$ Therefore $\int_{T-h^{(1)}}^T (Y^0_t-L_t)dK^{0,T}_t=0$ by the same calculation as \eqref{e:Skorohod tilde Y}. In addition, by Eq.~\eqref{e1} we have for $t\in[T-h^{(1)},T]$
\begin{equation*}
\begin{aligned}
&       Y^n_t:=\phi^{n,T}(t,S_t,\widetilde{Y}^{n,T}_t), \ K^{n}_t:=\int_{0}^t D_y \phi^{n,T}(s,S_s,\widetilde{Y}^{n,T}_s)d\widetilde{K}^{n,T}_s,\\ 
&  Z^{n}_t:=D_y\phi^{n,T}(t,S_t,\widetilde{Y}^{n,T}_t)\left[\widetilde{Z}^{n,T}_t+\frac{(D_x\phi^{n,T}(t,S_t,\widetilde{Y}^{n,T}_t))^{\top}}{D_y\phi^{n,T}(t,S_t,\widetilde{Y}^{n,T}_t)}\sigma_t\right].
\end{aligned}
\end{equation*}
Fix a sample point $\omega,$ and note that the difference between  $K^{n}_{\cdot} - K^{n}_{T-h^{(1)}}$ and $K^{0,T}_{\cdot}$ is 
\begin{equation}\label{e:first second term}
\begin{aligned}
&\int_{T-h^{(1)}}^{\cdot} \left(D_{y}\phi^{n,T}(s,S_{s},\widetilde{Y}^{n,T}_{s}) - D_{y}\phi^{0,T}(s,S_{s},\widetilde{Y}^{0,T}_{s})\right)d\widetilde{K}^{n,T}_{s}\\
&+ \int_{T-h^{(1)}}^{\cdot} D_{y}\phi^{0,T}(s,S_{s},\widetilde{Y}^{0,T}_{s})d(\widetilde{K}^{n,T} - \widetilde{K}^{0,T})_{s}.
\end{aligned}
\end{equation}
By the uniform convergence of $D_{y}\phi^{n,T}(t,S_t,\widetilde{Y}^{n,T}_t) \rightarrow D_{y}\phi^{0,T}(t,S_t,\widetilde{Y}^{0,T}_t)$ and the uniform boundedness of $\widetilde{K}^{n,T}_{T},$ the sequence in the first term of \eqref{e:first second term} uniformly converges to $0$.
In addition, note that $t\mapsto D_{y}\phi^{0,T}(t,S_{t},\widetilde{Y}^{0,T}_{t})$ is  continuous and positive, and note that $\widetilde{K}^{n,T}_t - \widetilde{K}^{n,T}_{T-h^{(1)}} \rightarrow \widetilde{K}^{0,T}_t - \widetilde{K}^{0,T}_{T-h^{(1)}}$ for every $t\in [T-h^{(1)},T].$ By Lemma~\ref{lem:K} with $K^n_t = \widetilde{K}^{n,T}_t - \widetilde{K}^{n,T}_{T-h^{(1)}},$ $K_t = \widetilde{K}^{0,T}_t - \widetilde{K}^{0,T}_{T-h^{(1)}},$ and $y_t = D_{y}\phi^{0,T}(t,S_{t},\widetilde{Y}^{0,T}_{t}),$ we have
\begin{equation*}
\int_{T-h^{(1)}}^{t} D_{y}\phi^{0,T}(s,S_{s},\widetilde{Y}^{0,T}_{s})d\widetilde{K}^{n,T}_s \rightarrow \int_{T-h^{(1)}}^{t} D_{y}\phi^{0,T}(s,S_{s},\widetilde{Y}^{0,T}_{s})d\widetilde{K}^{0,T}_{s}\text{ uniformly in }t\in[T-h^{(1)},T].
\end{equation*}
Hence, \eqref{e:first second term} converges to $0$ uniformly on $[T-h^{(1)},T].$ Combining \eqref{63} and \eqref{65} we obtain that
\begin{equation}\begin{split}\label{66}
&{Y}^n_{\cdot}\rightarrow {Y}^0_{\cdot}, \  {K}^n_{\cdot} - K^n_{T-h^{(1)}}\rightarrow {K}^{0,T} \textrm{ uniformly on } [T-h^{(1)},T] \textrm{ a.s.,}\\
&{Z}^n\rightarrow {Z}^0 \textrm{ in } \mathrm{H}^{2}_{[T-h^{(1)},T]}.
\end{split}\end{equation}

\textbf{Step 3.} We will show that $(Y^n,Z^n,K^n)$ converges to $(Y,Z,K)$ on the entire interval and that the Skorokhod condition holds by induction. Fix an $i\in\left\{0,1,...,N-1\right\}$ and assume that we already have 
\begin{equation*}
{Y}^n_{\cdot}\rightarrow {Y}^0_{\cdot}, \textrm{ uniformly on } [T-h^{(i)},T] \textrm{ a.s..}
\end{equation*}
First, we show the above convergence holds on $[T-h^{(i+1)},T]$.
Denote by $(\widetilde{Y}^{0,T-h^{(i)}},\widetilde{Z}^{0,T-h^{(i)}},\widetilde{K}^{0,T-h^{(i)}})$ the unique solution of Eq.~\eqref{58} on the interval $[T-h^{(i+1)},T-h^{(i)}]$ with terminal time $T-h^{(i)}$ and parameter $(Y^{0}_{T-h^{(i)}},\widetilde{f}^{0,T-h^{(i)}},\widetilde{L}^{0,T-h^{(i)}},0,0).$  Note that $Y^{n}_{T-h^{(i)}}$ converges to $Y^{0}_{T-h^{(i)}}$ a.s. and $(\widetilde{f}^{n,T-h^{(i)}},\widetilde{L}^{n,T-h^{(i)}})$ converges to $(\widetilde{f}^{0,T-h^{(i)}},\widetilde{L}^{0,T-h^{(i)}})$ locally uniformly in $(t,x,\widetilde{y},\widetilde{z}).$ Thus, by Lemma~\ref{lem:continu solution map}, it follows that
\begin{equation*}
\begin{aligned}
&\widetilde{Y}^{n,T-h^{(i)}}_{\cdot}\rightarrow \widetilde{Y}^{0,T-h^{(i)}}_{\cdot}, \  \widetilde{K}^{n,T-h^{(i)}}_{\cdot} - \widetilde{K}^{n,T-h^{(i)}}_{T-h^{(i+1)}}\rightarrow \widetilde{K}^{0,T-h^{(i)}}_{\cdot} - \widetilde{K}^{0,T-h^{(i)}}_{T-h^{(i+1)}} \\
&\qquad\qquad\qquad\qquad\qquad\qquad\textrm{ uniformly on } [T-h^{(i+1)},T-h^{(i)}] \textrm{ a.s.,}\\
&\widetilde{Z}^{n,T-h^{(i)}}\rightarrow \widetilde{Z}^{0,T-h^{(i)}} \textrm{ in } \mathrm{H}^{2}_{[T-h^{(i+1)},T-h^{(i)}]}.
\end{aligned}
\end{equation*}
For $t\in[T-h^{(i+1)},T-h^{(i)}]$, set 
\begin{equation*}
\begin{aligned}
&Y^0_t:=\phi^{0,T-h^{(i)}}(t,S_t,\widetilde{Y}^{0,T-h^{(i)}}_t), \ K^{0,T-h^{(i)}}_t:=\int_{T-h^{(i+1)}}^t D_y \phi^{0,T-h^{(i)}}(s,S_s,\widetilde{Y}^{0,T-h^{(i)}}_s)d\widetilde{K}^{0,T-h^{(i)}}_s,\\ 
&Z^0_t:=D_y\phi^{0,T-h^{(i)}}(t,S_t,\widetilde{Y}^{0,T-h^{(i)}}_t)\left[\widetilde{Z}^{0,T-h^{(i)}}_t+\frac{(D_x\phi^{0,T-h^{(i)}}(t,S_t,\widetilde{Y}^{0,T-h^{(i)}}_t))^{\top}}{D_y\phi^{0,T-h^{(i)}}(t,S_t,\widetilde{Y}^{0,T-h^{(i)}}_t)}\sigma_t\right].
\end{aligned}
\end{equation*}
Repeating the same argument as $i=0,$ we have the following convergence in the form of \eqref{66}
\begin{equation}\label{66'}
\begin{aligned}
&{Y}^n_{\cdot}\rightarrow {Y}^0_{\cdot}, \  {K}^n_{\cdot} - K^n_{T-h^{(i+1)}}\rightarrow {K}^{0,T-h^{(i)}}_{\cdot} \textrm{ uniformly on } [T-h^{(i+1)},T-h^{(i)}] \textrm{ a.s.,}\\
&{Z}^n\rightarrow {Z}^0 \textrm{ in } \mathrm{H}^{2}_{[T-h^{(i+1)},T-h^{(i)}]},
\end{aligned}
\end{equation}
and then
\begin{equation*}
Y^{n}_{\cdot}\rightarrow Y^{0}_{\cdot},\text{ uniformly on }[T-h^{(i+1)},T] \text{ a.s..}
\end{equation*}
Thus, the induction is valid, and then \eqref{66'} holds for every $i=0,1,...,N-1.$ Let $K^0_0:=0,$ and for each $i=0,1,...,N-1$ and $t\in(T-h^{(i+1)},T-h^{(i)}]$, let
\begin{equation*}
K^0_t:= \sum_{j=1}^{N-i-1}K^{0,T-h^{(N-j)}}_{T-h^{(N-j)}}+K^{0,T-h^{(i)}}_t,
\end{equation*}
where we stipulate that $\sum\limits_{j=1}^{0} K^{0,T-h^{(N-j)}}_{T-h^{(N-j)}} = 0.$ Clearly, for $t\in(T-h^{(i+1)},T-h^{(i)}],$ we have 
\begin{equation*}
\left|K^{n}_{t} - K^{0}_{t}\right|= \left|\sum_{j=1}^{N-i-1}\left[K^{n}_{T-h^{(N-j)}} - K^{n}_{T-h^{(N-j+1)}} - K^{0,T-h^{(N-j)}}_{T-h^{(N-j)}}\right] + K^{n}_{t} - K^{n}_{T-h^{(i+1)}} - K^{0,T-h^{(i)}}_{t} \right|.
\end{equation*}
Therefore, as $n\rightarrow\infty,$ we have
\begin{equation}\label{e:sum k}
\begin{aligned}
&\sup_{t\in[0,T]}\left|Y^n_t-Y^0_t\right|\leq \sum_{i=0}^{N-1}\sup_{T-h^{(i+1)}< t\le T-h^{(i)}}\left|Y^n_t-Y^0_t\right|\rightarrow 0,\quad\text{a.s.,}\\
&\sup_{t\in[0,T]}\left|K^n_t-K^0_t\right|\leq \sum_{i=0}^{N-1}\sup_{T-h^{(i+1)}< t\le T-h^{(i)}}\left|K^n_t - K^{n}_{T-h^{(i+1)}} - K^{0,T-h^{(i)}}_t\right|\rightarrow 0,\quad\text{a.s.,}\\
&\mathbb{E}\left[\int_0^T \left|Z^n_t-Z^0_t\right|^2 dt\right]\le \sum_{i=0}^{N-1}\mathbb{E}\left[\int_{T-h^{(i+1)}}^{T-h^{(i)}} \left|Z^n_t -Z^0_t\right|^2 dt\right]\rightarrow 0.
\end{aligned}
\end{equation}
Furthermore, by the same calculation as in \eqref{e:Skorohod tilde Y}, we have for $t\in[0,T],$ 
\begin{equation*}
Y^0_t\geq L_t,\  \text{and }\int_{0}^{T} (Y^{0}_{r} - L_{r}) dK^{0}_{t}=0.
\end{equation*}

Finally, $(Y^0,Z^0,K^0)$ is independent of the approximation sequence since the sequence of $\mathrm{X}^n$ is arbitrarily chosen in the beginning. The proof is complete.
\end{proof}

In the above theorem, we have shown that $\|Z\|_{\mathrm{H}^{2};[0,T]}<\infty.$ In the following we improve the integrability via proving the boundedness of the BMO norm of $Z$, where  
\begin{equation*}
\|Z\|_{\text{BMO}} :=  \esssup_{(t,\omega)\in[0,T]\times\Omega}\left|\mathbb{E}_{t}\left[\int_{t}^{T}|Z_{r}|^{2}dr\right]\right|^{\frac{1}{2}}.
\end{equation*}
Consequently, we have the $L^2$-boundedness of $K,$ which is essential to the penalization method introduced in Subsection \ref{subsec:penalization}.

\begin{proposition}\label{prop:BMO}
Under the same assumption as in Theorem~\ref{thm:well-posedness}, let $(Y,Z,K)$ be the solution of Eq.~\eqref{58} with parameter $(\xi,f,L,H,\mathbf{X}).$ Then there exists a constant $C>0$ depending on $\|\mathbf{X}\|_{p\text{-var};[0,T]}$ 
such that 
\begin{equation*}
\|Z\|_{\text{BMO}}\le C,\ \text{ and }\ \|K_{T}\|_{L^{2}}\le C.
\end{equation*}
\end{proposition}

\begin{proof}
For each $n\ge 1$ let $(Y^n,Z^n,K^n)$ be given in Theorem~\ref{thm:well-posedness}, and $(\widetilde{Y}^{n,T},\widetilde{Z}^{n,T},\widetilde{K}^{n,T})$ be defined in \eqref{e1}.  
We first claim that there exits a constant $\widetilde{C}_{\text{BMO}}>0$ such that 
\begin{equation}\label{e:||tilde Z||_BMO}
\sup_{n\ge 1}\|\widetilde{Z}^{n,T}\|_{\text{BMO}}\le\widetilde{C}_{\text{BMO}}.
\end{equation}
If so, by the representation~\eqref{e:inverse transformation}, and boundedness of the inverse transformation shown by Lemma~\ref{lem:boundedness of the flow}, the above boundedness implies that $\sup_{n\ge 1}\|Z^{n}\|_{\text{BMO}}\le C_{\text{BMO}}$ with some constant $C_{\text{BMO}}>0.$ In addition, recalling that $Z$ is the limit of $Z^n$ in $\mathrm{H}^{2}_{[0,T]}(\R^d),$ it follows that $\int_{t}^{T}|Z^{n}_{r}|^2dr\rightarrow \int_{t}^{T}|Z_{r}|^2dr $ in probability. Then we obtain $\|Z\|_{\text{BMO }}\le C_{\text{BMO}}$ by Fatou’s Lemma for conditional expectations. Indeed, \eqref{e:||tilde Z||_BMO} can be proved similarly as the proof of \cite[Proposition 2.2]{sun2022quantitative} so we omit the proof\footnote{Although in \cite{sun2022quantitative} parameters of the corresponding RBSDE are assumed to be Markovian, the boundedness of BMO norm can be extended to the case satisfying $(A1)$ of Assumption~\ref{(H1)}.}. 

To complete the proof, it remains to show the boundedness of  $\|K\|_{\mathrm{I}^{2};[0,T]}.$ 
Firstly, according to \cite[Corollary~2.1]{kazamaki1994continuous}, the estimate \eqref{e:||tilde Z||_BMO} implies that there exist a constant $C'>0$ such that
\begin{equation*}
\sup_{n\ge 1}\|\widetilde{Z}^{n,T}\|_{\mathrm{H}^{4}_{[0,T]}}\le C'.
\end{equation*}
Secondly, by \eqref{e:bounde of Yn} we have $\sup_{n\ge 1}|\widetilde{Y}^{n,T}_{t}|\le \widetilde{M}$ with some constant $\widetilde{M}>0.$ Finally, noting that
\begin{equation*}
\widetilde{K}^{n,T}_{T} = \widetilde{Y}^{n,T}_{0} - \xi - \int_{0}^{T}\widetilde{f}^{n,T}(r,S_r,\widetilde{Y}^{n,T}_{r},\widetilde{Z}^{n,T}_{r})dr + \int_{0}^{T}\widetilde{Z}^{n,T}_{r} dW_{r},
\end{equation*}
we obtain the following inequality by the condition (A1) in \ref{(H0)}
\begin{equation*}
\sup_{n\ge 1}\mathbb{E}\left[\left|\widetilde{K}^{n,T}_{T}\right|^{2}\right] \le C'' + C'' \sup_{n\ge 1}\mathbb{E}\left[\left|\int_{0}^{T}|\widetilde{Z}^{n,T}_{r}|^{2}dr\right|^{2}\right] 
\end{equation*}
with some constant $C''>0,$ and then the boundedness of $\mathbb{E}\left[|K_{T}|^{2}\right]$ follows.
\end{proof}

The following comparison theorem is a direct consequence of the comparison theorem for RBSDEs with quadratic growth (we refer 
it to \cite[Proposition~3.2]{KLQT}) and Theorem~\ref{thm:well-posedness}. 

\begin{corollary}\label{cor:comparison}
Suppose that for each $j=1,2,$ the parameter $(\xi^{j},f^{j},L^{j},H,\mathbf{X})$ satisfy the conditions in Theorem~\ref{thm:well-posedness}. Assume further that  $\xi^{1}\le \xi^{2},$ $f^{1}\le f^{2},$ and $L^{1}\le L^{2}.$ Let $(Y^{1},Z^{1},K^{1})$ and $(Y^{2},Z^{2},K^{2})$ be the solution of Eq.~\eqref{58} with parameters $(\xi^{1},f^{1},L^{1},H,\mathbf{X})$ and $(\xi^{2},f^{2},L^{2},H,\mathbf{X})$ respectively. Then, a.s., for all $t\in[0,T],$
\begin{equation*}
    Y^{1}_t\le Y^{2}_t.
\end{equation*}
\end{corollary}

Now we present a very useful stability result for rough RBSDEs, which serves as an essential tool for obtaining the results in the following sections.

\begin{proposition}\label{prop:conti of (Y^0,Z^0,K^0)}
Suppose that for each $n\ge 0,$ $(\xi^{(n)},f^{(n)},L^{(n)},H^{(n)},\mathbf{X}^{(n)})$ satisfies the same conditions as in Theorem~\ref{thm:well-posedness}.
Assume further that the sequence $f^n$ converges to $f^{0}$ locally uniformly in $(t,x,y,z)\in[0,T]\times\mathbb{R}^d\times \mathbb{R}\times\mathbb{R}^{d}$ a.s., $\xi^n$ converges to $\xi^{0}$ a.s., $L^{n}$ converges to $L^{0}$ uniformly in $t\in[0,T]$ a.s., $H^{(n)}$ converges to $H^{(0)}$ in $\text{Lip}^{\gamma+2}(\R^{d+1},\R^{l}),$ and $\mathbf{X}^{(n)}$ converges to $\mathbf{X}^{(0)}$ in $C^{0,p\text{-var}}([0,T],G^{[p]}(\R^{l})).$ Let $(Y^{(n)},Z^{(n)},K^{(n)})$ be the unique solution to Eq.~\eqref{58} with parameter $(\xi^{(n)},f^{(n)},L^{(n)},H^{(n)},\mathbf{X}^{(n)})$. Then, we have as $n\rightarrow \infty$ 
\begin{equation*}
\begin{aligned}
(Y^{(n)},K^{(n)}) &\rightarrow (Y^{(0)},K^{(0)}) \text{ uniformly on } [0,T],\ \mathbb{P} \text{-a.s.,}\\
Z^{(n)} &\rightarrow Z^{(0)} \text{ in }\mathrm{H}^{2}_{[0,T]}(\mathbb{R}^d).
\end{aligned}
\end{equation*}
\end{proposition}
\begin{proof}
The idea is to apply Lemma~\ref{lem:continu solution map} to the transformed equations. For $T'\in[0,T)$ and $t\in[T',T],$ for $n\ge 0$ denote by $\phi^{(n),T'}(t,x,\widetilde{y})$ the unique solution of the following RDE on $t\in[0,T']$
\begin{equation*}
\phi^{(n),T'}(t,x,\widetilde{y}) = \widetilde{y} + \int_{t}^{T'}H^{(n)}(x,\phi^{(n),T'}(r,x,\widetilde{y}))d\mathbf{X}^{(n)}_{r}.
\end{equation*} 
In addition, set 
\begin{equation*}
\begin{aligned}
&\widetilde{Y}^{(n),T'}_{t}:=(\phi^{(n),T'})^{-1}(t,S_{t},Y^{(n)}_{t}), \  \widetilde{K}^{(n),T'}_t:=\int_0^t \frac{1}{D_{y} \phi^{(n),T'}(s,S_{s},\widetilde{Y}^{(n),T'}_s)}dK^{(n)}_s,\\
&\widetilde{Z}^{(n),T'}_t := -\frac{(D_x \phi^{(n),T'}(t,S_t,\widetilde{Y}^{(n),T'}_t))^{\top}}{D_y\phi^{(n),T'}(t,S_t,\widetilde{Y}^{(n),T'}_t)}\sigma_t+ \frac{1}{D_y \phi^{(n),T'}(t,S_t,\widetilde{Y}^{(n),T'}_t)}Z^{(n)}_t,\\
&\widetilde{f}^{(n),T'}\big(t,x,\widetilde{y},\widetilde{z}):=\frac{1}{D_y \phi^{(n),T'}}\bigg\{f^{(n)}(t,x,\phi^{(n),T'},D_y \phi^{(n),T'} \widetilde{z} + \sigma^{\top} D_x \phi^{(n),T'} \big)\\
&\qquad + (D_x \phi^{(n),T'})^{\top} b + \frac{1}{2}\text{Tr}\left\{D^{2}_{xx}\phi^{(n),T'} \sigma\sigma^{\top}\right\} + (D^{2}_{xy}\phi^{(n),T'})^{\top} \sigma \widetilde{z}+\frac{1}{2}D^{2}_{yy}\phi^{(n),T'} |\widetilde{z}|^2\bigg\},
\end{aligned}
\end{equation*}
and $\widetilde{L}^{(n),T'} := (\phi^{(n),T'})^{-1}(t,S_{t},L^{(n)}_{t}).$ Here, $\phi^{(n),T'}$ and all its derivatives are valued at $(t,x,\widetilde{y}).$ By the same argument as the step~2 of the proof of Theorem~\ref{thm:well-posedness}, there exists a sequence of increasing positive numbers $\left\{h^{(i)},i=0,1,2,...,N\right\}$ with $[T-h^{(N)},T-h^{(0)}] = [0,T]$
such that, for each $n\ge 0$ and each $i\in\left\{0,1,2,...,N-1\right\},$  
\begin{equation*}
(\widetilde{Y}^{(n),T-h^{(i)}}_{\cdot},\widetilde{Z}^{(n),T-h^{(i)}}_{\cdot},\widetilde{K}^{(n),T-h^{(i)}}_{\cdot} - \widetilde{K}^{(n),T-h^{(i)}}_{T-h^{(i+1)}})
\end{equation*}
uniquely solves Eq.~\eqref{58} on $[T-h^{(i+1)},T-h^{(i)}]$ with terminal time $T-h^{(i)}$ and parameter $(Y^{(n),T-h^{(i)}}_{T-h^{(i)}},\widetilde{f}^{(n),T-h^{(i)}},\widetilde{L}^{(n),T-h^{(i)}},0,0).$

Now we verify conditions to apply Lemma~\ref{lem:continu solution map} to $(Y^{(n),T-h^{(i)}}_{T-h^{(i)}},\widetilde{f}^{(n),T-h^{(i)}},\widetilde{L}^{(n),T-h^{(i)}})$ on $[T-h^{(i+1)},T-h^{(i)}]$. Indeed, Lemma~\ref{lem:boundedness of the flow} implies that for each $i\in\left\{0,1,...,N-1\right\},$ the terminal values $Y^{(n)}_{T-h^{(i)}},$ the obstacles $\widetilde{L}^{(n),T-h^{(i)}},$ and the functions $\widetilde{f}^{(n),T-h^{(i)}}$ satisfy (i) and (ii) in Lemma~\ref{lem:continu solution map}. 
In addition, set 
\begin{equation*}
\left\{\begin{aligned}
\hat{\phi}^{(n),T-h^{(i)}}_t &:= \Big(x,\phi^{(n),T-h^{(i)}},D_{x}\phi^{(n),T-h^{(i)}},D_{y}\phi^{(n),T-h^{(i)}},D^{2}_{xy}\phi^{(n),T-h^{(i)}},\\
&\qquad\qquad\qquad\qquad\qquad\qquad\qquad D^{2}_{xx}\phi^{(n),T-h^{(i)}},D^{2}_{yy}\phi^{(n),T-h^{(i)}}\Big),\\
\hat{\phi}^{(n),T-h^{(i)}}_T &= (x,y,0,1,0,0,0),
\end{aligned}
\right.
\end{equation*}
where $\phi^{(n),T-h^{(i)}}$ and its derivatives are valuated at $(t,x,y).$ Then, according to the proof of \cite[Theorem~11.12]{friz2010multidimensional}, $\hat{\phi}^{(n),T-h^{(i)}}$ is the unique solution to an RDE along some collection of vector fields $\hat{H}^{(n)} \in \text{Lip}^{\gamma - 1}(\R^{d^2 + 3d + 3},\R^{(d^2 + 3d + 3)\times l})$ and driven by $d\mathbf{X}^{(n)}_t .$ By \cite[Proposition~11.5]{friz2010multidimensional}, the derivatives of $\phi^{(n),T-h^{(i)}}$ satisfies the RDE obtained by formal differentiation. Then by this observation, each element of $\hat{H}^{(n)}$ is multiplication of some derivatives of $H^{(n)}$ and some truncated linear functions. Since $H^{(n)}\rightarrow H^{(0)}$ in $\text{Lip}^{\gamma + 1}(\R^d\times\R,\R^l)$ we have 
\begin{equation}\label{e:hat H^n}
\hat{H}^{(n)}\rightarrow \hat{H}^{(0)} \text{ in }\text{Lip}^{\gamma - 1}(\R^{d^2 + 3d + 3},\R^{(d^2 + 3d + 3)\times l}).
\end{equation}
In view of \eqref{e:hat H^n} and the convergence of $\mathbf{X}^{(n)}\rightarrow \mathbf{X}^{0}$ in $C^{0,p\text{-var}}([0,T],G^{[p]}(\R^l)),$ by \cite[Corollary~10.27]{friz2010multidimensional} 
\begin{equation*}
\hat{\phi}^{(n),T-h^{(i)}}\rightarrow \hat{\phi}^{(0),T-h^{(i)}}\text{ locally uniformly in }(t,x,y)\in [0,T]\times\R^d\times\R.
\end{equation*}
Similarly, we also have the above locally convergence with $\phi^{(n),T-h^{(i)}}$ replaced by $(\phi^{(n),T-h^{(i)}})^{-1}.$ Thus, in view of the convergence of derivatives of $\phi^{(n),T-h^{(i)}}$ and its inverse, and by adapting the proof of Lemma~\ref{lem:locally uniform conver} with $\phi^{n,T'}$ replaced by 
$\phi^{(n),T-h^{(i)}},$ the sequence of $(\widetilde{f}^{(n),T-h^{(i)}},\widetilde{L}^{(n),T-h^{(i)}})$ satisfies the requirements (iii) in Lemma~\ref{lem:continu solution map}.

Finally, we finish the proof by Lemma~\ref{lem:continu solution map} and induction argument. Note that $T = T - h^{(0)}$ and 
\begin{equation*}
Y^{(n)}_{T} = \xi^{n}\rightarrow\xi = Y^{(0)}_{T} \ \text{ a.s..}
\end{equation*} 
For $i=0,1,2,...,N-1,$ assume that the terminal value $Y^{(n)}_{T-h^{(i)}}\rightarrow Y^{(0)}_{T-h^{(i)}}$ a.s.. Then by Lemma~\ref{lem:continu solution map}
\begin{equation*}
\begin{aligned}
&(\widetilde{Y}^{(n),T-h^{(i)}}_{\cdot},\widetilde{K}^{(n),T-h^{(i)}}_{\cdot} - \widetilde{K}^{(n),T-h^{(i)}}_{T-h^{(i+1)}})\rightarrow (\widetilde{Y}^{(0),T-h^{(i)}}_{\cdot},\widetilde{K}^{(0),T-h^{(i)}}_{\cdot} - \widetilde{K}^{(0),T-h^{(i)}}_{T-h^{(i+1)}})\\
&\qquad\qquad\qquad\qquad\qquad\text{ uniformly on }[T-h^{(i+1)},T-h^{(i)}]\ \text{ a.s.,}\\
&\widetilde{Z}^{(n),T-h^{(i)}}\rightarrow \widetilde{Z}^{(0),T-h^{(i)}}\  \text{ in }\ \mathrm{H}^{2}_{[T-h^{(i+1)},T-h^{(i)}]}(\R^{d}).
\end{aligned}
\end{equation*}
As a consequence, we can repeat the proof of \eqref{66} to obtain that
\begin{equation*}
\begin{aligned}
(Y^{(n)}_{\cdot},K^{(n)}_{\cdot} - K^{(n)}_{T-h^{(i+1)}})&\rightarrow (Y^{(0)}_{\cdot},K^{(0)}_{\cdot} - K^{(0)}_{T-h^{(i+1)}})\ \text{ uniformly on }\ [T-h^{(i+1)},T-h^{(i)}]\ \text{ a.s.,}\\
Z^{(n)}&\rightarrow Z^{(0)}\ \text{ in }\ \mathrm{H}^{2}_{[T-h^{(i+1)},T-h^{(i)}]}(\R^{d}).
\end{aligned}
\end{equation*}
This convergence implies that 
\begin{equation*}
Y^{(n)}_{T-h^{(i+1)}}\rightarrow Y^{(0)}_{T-h^{(i+1)}}\text{ a.s. .}
\end{equation*}
Hence, by induction, it follows that  $Y^{(n)}_{T-h^{(i)}} \rightarrow Y^{(0)}_{T-h^{(i)}}$ a.s. for all $i\in\{0,1,...,N\}$, 
and then
\begin{equation*}
\begin{aligned}
(Y^{(n)}_{\cdot},K^{(n)}_{\cdot})&\rightarrow (Y^{(0)}_{\cdot},K^{(0)}_{\cdot})\ \text{ uniformly on }\ [0,T]\ \text{ a.s.,}\\
Z^{(n)}&\rightarrow Z^{(0)}\  \text{ in }\ \mathrm{H}^{2}_{[0,T]}(\R^{d}).
\end{aligned}
\end{equation*}
\end{proof}

\subsection{Approximation via penalization}\label{subsec:penalization}

In this subsection, we show that the solution of rough RBSDE can be constructed by a penalization method. Recall the rough BSDE as follows
\begin{equation}\label{59-}
Y_{t} = \xi + \int_{t}^{T} f(r,S_{r},Y_{r},Z_{r})dr + \int_{t}^{T} H(S_{r},Y_{r})d\mathbf{X}_{r} -\int_{t}^{T} Z_{r} dW_{r},\quad t\in[0,T].
\end{equation}
Suppose that the assumptions \ref{(Hpr)}, \ref{(H0)}, and \ref{(H_b,sigma)} hold, and that $\xi\in L^{\infty}(\mathcal{F}_{T}).$ \cite{DF} shows that Eq.~\eqref{59-} admits a unique solution $(Y,Z)\in \mathrm{H}^{\infty}_{[0,T]}\times \mathrm{H}^{2}_{[0,T]}(\R^{d})$ (see Proposition~\ref{prop:BSDE with rough driver}). 

\begin{theorem}\label{thm:punishment}
Assume the same conditions as in Theorem~\ref{thm:well-posedness}. Let $(Y,Z,K)$ be the unique solution to Eq.~\eqref{e:BSDE main result}. For each $m\ge 1,$ let  $(Y^{m},Z^{m})$ be the solution to the following rough BSDE on $[0,T]$
\begin{equation}\label{e:rough BSDE punish}
Y^{m}_{t} = \xi + \int_{t}^{T}f(r,S_{r},Y^{m}_{r},Z^{m}_{r}) dr + \int_{t}^{T}H(S_{r},Y^{m}_{r})d\mathbf{X}_{r} + m\int_{t}^{T}(Y^{m}_{r} - L_{r})^{-}dr - \int_{t}^{T}Z^{m}_{r}dW_{r}.
\end{equation}
Then $(Y^{m}_{\cdot},K^{m}_{\cdot}) \rightarrow (Y_{\cdot},K_{\cdot})$ uniformly on $[0,T]$ a.s., and $Z^{m}\rightarrow Z$ in $\mathrm{H}^{2}_{[0,T]}(\R^d),$ where $K^{m}_{t} := m\int_{0}^{t}(Y^{m}_{r} - L_{r})^{-} dr.$ 
\end{theorem}

\begin{proof}
Denote $f^{m}(t,x,y,z) := f(t,x,y,z) + m(y-L_{t})^{-}.$ Recall that for each $m\ge 1$, $(Y^{m},Z^{m})$ is the unique solution to Eq.~\eqref{59-} with parameter $(\xi,f^m,H,\mathbf{X}).$ Denote $L^{m}_{t} := Y^{m}_{t}\land L_{t}.$ \\

\textbf{Step 1.} We will show that for $m\ge 1$, the unique solution to BSDE~\eqref{e:rough BSDE punish} is also the unique solution to the following RBSDE with rough driver $d\mathbf{X}_r$
\begin{equation}\label{e:BSDE is also RBSDE}
\left\{\begin{aligned}
&Y^m_{t} = \xi + \int_{t}^{T}f(r,S_{r},Y^{m}_{r},Z^{m}_{r}) dr + \int_{t}^{T}H(S_{r},Y^{m}_{r})d\mathbf{X}_{r} - \int_{t}^{T}Z^{m}_{r}dW_{r} + K^{m}_{T} - K^{m}_{t},\\
&Y^{m}_{t}\ge L^{m}_{t},\quad \int_{0}^{T} (Y^{m}_{r} - L^{m}_{r}) dK^{m}_r = 0,\quad t\in[0,T].
\end{aligned}\right.
\end{equation}
Let $\left\{\mathbf{X}^{n},n\ge 1\right\}$ be the same lifts sequence of the smooth paths given in Theorem~\ref{thm:well-posedness} such that $\mathbf{X}^{n}\rightarrow\mathbf{X}\in C^{0,p\text{-var}}([0,T],G^{[p]}(\R^{l})).$ For each $m\ge 1,$ by Proposition~\ref{prop:BSDE with rough driver} we have that $(Y^{m},Z^{m}),$ the solution of Eq.~\eqref{e:rough BSDE punish} with parameter $(\xi,f^{m},H,\mathbf{X}),$ is the limit of solutions of Eq.~\eqref{e:rough BSDE punish} with $(\xi,f^{m},H,\mathbf{X}^{n}),$ $n\ge 1.$ And we denote by $(\bar{Y}^{m,n},\bar{Z}^{m,n})$ the solution of Eq.~\eqref{e:rough BSDE punish} with $(\xi,f^{m},H,\mathbf{X}^{n}).$ For $m,n\ge 1,$ set $\bar{K}^{m,n}_{t} := m\int_{0}^{t}(\bar{Y}^{m,n}_{r} - L_{r})^{-}dr$ and $\bar{L}^{m,n}_{t} := \bar{Y}^{m,n}_{t}\land L_{t}.$ By Definition~\ref{def:smooth case}, for each $ m\ge 1 $ and $ n\ge 1$ $(\bar{Y}^{m,n},\bar{Z}^{m,n},\bar{K}^{m,n})$ also solves RBSDE~\eqref{58} with parameter $(\xi,f^{m},\bar{L}^{m,n},H,\mathbf{X}^{n}).$ Note $\bar{L}^{m,n}_{\cdot}$ converges to $L^{m}_{\cdot}$ uniformly on $[0,T]$ a.s.  as $n\rightarrow\infty$ by the same convergence of $\bar{Y}^{m,n}\rightarrow \bar{Y}^{m}.$ Then the triplet $(\bar{Y}^{m,n},\bar{Z}^{m,n},\bar{K}^{m,n})$ converges to the solution of Eq.~\eqref{e:BSDE is also RBSDE} as $n\rightarrow \infty,$ which is a consequence of the stability result shown in Proposition~\ref{prop:conti of (Y^0,Z^0,K^0)}. Recalling that $(\bar{Y}^{m,n},\bar{Z}^{m,n})\rightarrow(Y^{m},Z^{m})$ as $n\rightarrow\infty,$ we also have 
\begin{equation*}
\bar{K}^{m,n}_{\cdot} = m\int_{0}^{\cdot}(\bar{Y}^{m,n}_{r} - L_{r})^{-} dr\rightarrow m\int_{0}^{\cdot}(Y^{m}_{r} - L_{r})^{-} dr = K^{m}_{\cdot},\ \text{ uniformly on }\ [0,T]\ \text{ a.s., }\  n\rightarrow \infty.
\end{equation*} 
Thus, the triplet $(Y^{m},Z^{m},K^{m})$ solves Eq.~\eqref{e:BSDE is also RBSDE}, or equally, it solves Eq.~\eqref{58} with parameter $(\xi,f,L^{m},H,\mathbf{X}).$\\

\textbf{Step 2.} According to Proposition~\ref{prop:conti of (Y^0,Z^0,K^0)}, it is sufficient to show that $L^{m}_{\cdot} \rightarrow L_{\cdot}$ uniformly on $[0,T]$ a.s. to prove the convergence $(Y^{m},Z^{m},K^{m})\rightarrow(Y,Z,K).$ Noting that $Y^{m}$ is increasing by Proposition~\ref{prop:comparison-BSDE with rough driver} and the limit process $L$ is continuous, Dini's Lemma implies that the uniform convergence of $L^{m}:= Y^{m} \land L $ is equal to the convergence by point. To this end, let $Y^{0} := \lim_{m\rightarrow\infty} Y^{m},$ it suffices to prove that
\begin{equation}\label{e:Y^{0}>=L}
Y^{0}_{t}\ge L_{t},\text{ for all }t\in[0,T].
\end{equation} 
We will demonstrate the above inequality through the following three steps. \\

\textbf{Step 3.} We will show that the essential supremum of $Y^m$ and the $\mathrm{H}^{4}$-norm of $Z^{m}$ are uniformly bounded in $m\ge 1.$ Let $(\hat{Y},\hat{Z})$ be the unique solution to Eq.~\eqref{59-} with parameter $(\xi,f,H,\mathbf{X}).$ Recalling that $(Y^{m},Z^{m})$ is the unique solution to Eq.~\eqref{59-} with parameter $(\xi,f^m,H,\mathbf{X}).$ By  Proposition~\ref{prop:comparison-BSDE with rough driver}, we have a.s.
\begin{equation*}
\hat{Y}_{t} \le Y^{m}_{t},\ \text{ for all }\ t\in[0,T].
\end{equation*}
In view of \cite[Theorem~3]{DF} and Remark~\ref{rem:superlienar}, an estimate similar to \eqref{e:bounde of Yn} can be established for $\hat{Y}$. Consequently, there exists a constant $\hat{M}>0$ such that 
\[\esssup\limits_{(t,\omega)\in[0,T]\times\Omega}|\hat{Y}_t|\le \hat{M}.\] 
Thus, we have $(-\hat{M})\land L_{t}\le L^{m}_{t}\le L_{t}.$ According to the step 1, for each $m\ge 1$ the triplet $(Y^{m},Z^{m},K^{m})$ is also the unique solution to Eq.~\eqref{58} with parameter $(\xi,f,L^{m},H,\mathbf{X}).$ By the same argument leading to \eqref{e:bounde of Yn}, there exists a constant $M>0$ such that
\begin{equation}\label{e:ess |Y^m|}
\sup_{m\ge 1}\esssup_{(t,\omega)\in[0,T]\times \Omega}|Y^{m}_{t}|\le M.
\end{equation}
In addition, Proposition~\ref{prop:BMO} shows that $\|\int_{0}^{\cdot}Z^{m}_{r}dW_{r}\|_{\text{BMO}}$ are uniformly bounded in $m\ge 1.$ By \cite[Corollary~2.1]{kazamaki1994continuous}, we have
\begin{equation}\label{e:|int Z^2|^{2}}
\sup_{m\ge 1}\mathbb{E}\left[\left|\int_{0}^{T}|Z^{m}_{r}|^{2}dr\right|^{2}\right]<\infty.
\end{equation}

\textbf{Step 4.} We will show that a.s., $Y^{0}_{t}\ge L_t$ $dt$-a.e.. For the same reason as in Proposition~\ref{prop:BMO}, we have 
\begin{equation*}
\sup\limits_{m\ge 1}\mathbb{E}\left[\left|K^{m}_{T}\right|^{2}\right]<\infty.
\end{equation*} Hence, since $|Y^{m}|$ is uniformly bounded in $m\ge 1,$ it follows from the dominated convergence theorem that 
\begin{equation*}
\mathbb{E}\left[\int_{0}^{T}(Y^{0}_{t} - L_{t})^{-}dt\right] = \lim_{m\rightarrow \infty}\mathbb{E}\left[\int_{0}^{T}(Y^{m}_{t} - L_{t})^{-}dt\right] = \lim_{m\rightarrow \infty}\frac{1}{m}\mathbb{E}\left[K^{m}_{T}\right] = 0.
\end{equation*}
This implies that $a.s.,$ $Y^{0}_{t}\ge L_{t}$ for a.e. $t\in[0,T].$   \\

\textbf{Step 5.} We will show that a.s. the right limit of $Y^{0}_{t}$ (denoted by $Y^{0}_{t+},$) exists for all $t\in[0,T],$ and it satisfies a.s., $ Y^{0}_{t}\ge Y^{0}_{t+}$ for all $t\in[0,T).$ If this claim is true, combined with the fact that a.s., $Y^{0}_{t} \ge L_{t}$ for a.e. $t\in[0,T]$, it follows that a.s., $Y^{0}_{t}\ge L_{t}$ for all $t\in[0,T],$ which is exactly \eqref{e:Y^{0}>=L}. For any fixed $(T',x,\widetilde{y})\in[0,T]\times\R^d\times\R$, let $\phi^{T'}(\cdot,x,\widetilde{y})$ be the unique solution of the following RDE
\begin{equation*}
\phi^{T'}(t,x,\widetilde{y}) = \widetilde{y} + \int_{t}^{T'} H(x,\phi^{T'}(r,x,\widetilde{y})) d\mathbf{X}_{r},\quad t\in[0,T'].
\end{equation*}
For each $m\ge 1,$ $T'\in[0,T],$ and $t\in[0,T'],$ denote by $(\widetilde{Y}^{m,T'}_{t},\widetilde{Z}^{m,T'}_{t},\widetilde{K}^{m,T'}_{t})$ the following
\begin{equation*}
\begin{aligned}
\widetilde{Y}^{m,T'}_{t}&:=(\phi^{T'})^{-1}(t,S_{t},Y^{m}_{t}), \quad  \widetilde{K}^{m,T'}_t:=\int_0^t \frac{1}{D_{y} \phi^{T'}(s,S_{s},\widetilde{Y}^{m,T'}_s)}dK^{m}_s,\\
\widetilde{Z}^{m,T'}_t&:=-\frac{(D_x \phi^{T'}(t,S_t,\widetilde{Y}^{m,T'}_t))^{\top}}{D_y\phi^{T'}(t,S_t,\widetilde{Y}^{m,T'}_t)}\sigma_t+ \frac{1}{D_y \phi^{T'}(t,S_t,\widetilde{Y}^{m,T'}_t)}Z^{m}_t.\\
\end{aligned}
\end{equation*}
And for $(t,x,\widetilde{y},\widetilde{z})\in[0,T']\times\R^d\times\R\times\R^d$, denote
\begin{equation*}
\begin{aligned}
&\widetilde{f}^{T'}(t,x,\widetilde{y},\widetilde{z}):=\frac{1}{D_y \phi^{T'}}\bigg\{f\big(t,x,\phi^{T'},D_y \phi^{T'} \widetilde{z} + \sigma^{\top} D_x \phi^{T'} \big) + (D_x \phi^{T'})^{\top} b\\
&\qquad\qquad\qquad\qquad\qquad + \frac{1}{2}\text{Tr}\left\{D^{2}_{xx}\phi^{T'} \sigma\sigma^{\top}\right\} + (D^{2}_{xy}\phi^{T'})^{\top} \sigma \widetilde{z}+\frac{1}{2}D^{2}_{yy}\phi^{T'} |\widetilde{z}|^2\bigg\},\\
&\widetilde{L}^{m,T'}_t := (\phi^{T'})^{-1}(t,S_{t},L^{m}_{t}),
\end{aligned}
\end{equation*}
where $\phi^{T'}$ and its derivatives take values at $(t,x,\widetilde{y}),$ and $b,$ $\sigma$ take values at $(t,x).$ By \eqref{e:ess |Y^m|} and Lemma~\ref{lem:boundedness of the flow}, there exists a constant $\widetilde{M}>0$ such that
\begin{equation*}
\sup_{m\ge 1}\esssup_{(t,\omega)\in[0,T']\times\Omega}|\widetilde{Y}^{m,T'}_{t}| \le \widetilde{M}.
\end{equation*}
Thus, by the same argument in the step~2 of the proof pf Theorem~\ref{thm:well-posedness}, there exist an integer $N>0$ and a sequence of $\left\{h^{(i)},i=0,2,...,N\right\}$ such that $[T-h^{(N)},T-h^{(0)}] = [0,T],$ and that, in addition, 
the triplet $(\widetilde{Y}^{m,T-h^{(i)}}, \widetilde{Z}^{m,T-h^{(i)}}, \widetilde{K}^{m,T-h^{(i)}})$ uniquely solves Eq.~\eqref{58} on $[T-h^{(i+1)},T-h^{(i)}]$ with terminal time $T-h^{(i)}$ and parameter $(Y^{m}_{T-kh},\widetilde{f}^{T-h^{(i)}},L^{m,T-h^{(i)}},0,0).$ 
For any $t\in[T-h^{(i+1)},T-h^{(i)}],$ denote
\begin{equation}\label{e:tilde{Y}^0 = ..}
\widetilde{Y}^{0,T-h^{(i)}}_{t} := (\phi^{T-h^{(i)}})^{-1}(t,S_{t},Y^{0}_{t}).
\end{equation}
By the definition of $Y^{0}$ we have that a.s., $\widetilde{Y}^{m,T-h^{(i)}}_{t}\rightarrow \widetilde{Y}^{0,T-h^{(i)}}_{t}$ for all $t\in[T-h^{(i+1)},T-h^{(i)}]$ as $m\rightarrow\infty.$ In addition, by \eqref{e:ess |Y^m|} and \eqref{e:|int Z^2|^{2}}, we have that both $\|\widetilde{Z}^{m,T-h^{(i)}}_{t}\|_{\mathrm{H}^{2};[T-h^{(i+1)},T-h^{(i)}]}$ and $\|\widetilde{f}^{T-h^{(i)}}(t,S_{t},\widetilde{Y}^{m,T-h^{(i)}}_{t},\widetilde{Z}^{m,T-h^{(i)}}_{t})\|_{\mathrm{H}^{2};[T-h^{(i+1)}, T-h^{(i)}]}$ are uniformly bounded in $m\ge 1$ and $i=0,1,2,...,N-1.$ This boundedness implies that there exists a pair of processes $(\widetilde{Z}^{0,T-h^{(i)}},\widetilde{f}^{0,T-h^{(i)}})$ such that, by passing to a subsequence if necessary,
\begin{equation*}
\begin{aligned}
&(\widetilde{Z}^{m,T-h^{(i)}}_{\cdot},\widetilde{f}^{m,T-h^{(i)}}_{\cdot}):=(\widetilde{Z}^{m,T-h^{(i)}}_{\cdot},\widetilde{f}^{T-h^{(i)}}(\cdot,S_{\cdot},\widetilde{Y}^{m,T-h^{(i)}}_{\cdot},\widetilde{Z}^{m,T-h^{(i)}}_{\cdot}))\\
&\qquad\qquad\overset{\text{weakly}}{\longrightarrow} (\widetilde{Z}^{0,T-h^{(i)}}_{\cdot},\widetilde{f}^{0,T-h^{(i)}}_{\cdot})\ 
\text{ in }\ \mathrm{H}^{2}_{[T-h^{(i+1)},T-h^{(i)}]}(\R^d)\times \mathrm{H}^{2}_{[T-h^{(i+1)},T-h^{(i)}]}(\R).
\end{aligned}
\end{equation*}
For $i=0,1,...,N-1$ and $t\in[T-h^{(i+1)},T-h^{(i)}],$ define $\widetilde{K}^{0,T-h^{(i)}}_{t}$ by
\begin{equation}\label{e:after weak conver}
\widetilde{K}^{0,T-h^{(i)}}_{t} := \widetilde{Y}^{0,T-h^{(i)}}_{T-h^{(i+1)}} - \widetilde{Y}^{0,T-h^{(i)}}_{t} - \int_{T-h^{(i+1)}}^{t}\widetilde{f}^{0,T-h^{(i)}}_{r}dr + \int_{T-h^{(i+1)}}^{t}\widetilde{Z}^{0,T-h^{(i)}}_{r}dW_{r}.
\end{equation}
Noting that $\widetilde{Y}^{m,T-h^{(i)}}_{\tau} \rightarrow \widetilde{Y}^{0,T-h^{(i)}}_{\tau}$ a.s. as $m\rightarrow\infty,$ and for $\zeta\in L^{2}(\mathcal{F}_{T})$
\begin{equation*}
\mathbb{E}\left[\int_{T-h^{(i+1)}}^{\tau}\widetilde{f}^{m,T-h^{(i)}}_{r}dr \cdot \zeta \right] = \mathbb{E}\left[\int_{T-h^{(i+1)}}^{T}\widetilde{f}^{m,T-h^{(i)}}_{r}\cdot \mathbb{E}_{r}[\zeta \mathbf{1}_{\{r<\tau}\}] dr \right],
\end{equation*}
we have for each stopping time $\tau$ taking values in $[T-h^{(i+1)},T-h^{(i)}]$  
\begin{equation*}
\widetilde{K}^{m,T-h^{(i)}}_{\tau} \overset{\text{weakly}}{\longrightarrow} \widetilde{K}^{0,T-h^{(i)}}_{\tau}\ \text{ in }\ L^{2}(\mathcal{F}_{T}).
\end{equation*}
Since that for each $m\ge 1$ and $i\in\{0,1,...,N-1\}$ the process $\widetilde{K}^{m,T-h^{(i)}}_{t}$ is increasing in $t\in[T-h^{(i+1)},T-h^{(i)}]$ a.s., Lemma~\ref{lem:app of sction theorem}
implies that  $\widetilde{K}^{0,T-h^{(i)}}_{t}$ is also increasing in $t\in[T-h^{(i+1)},T-h^{(i)}]$ a.s.. In view of the equality~\eqref{e:after weak conver}, for any $i\in\{0,1,...,N-1\}$ and $t\in\left[T-h^{(i+1)},T-h^{(i)}\right)$, we obtain that $\widetilde{Y}^{0,T-h^{(i)}}_{t+}$ exists and $\widetilde{Y}^{0,T-h^{(i)}}_{t}\ge \widetilde{Y}^{0,T-h^{(i)}}_{t+}.$ Thus, $Y^{0}_{t+} = \phi^{T-h^{(i)}}(t+,S_{t+},\widetilde{Y}^{0,T-h^{(i)}}_{t+})$ exists by the continuity of $\phi^{T-h^{(i)}}(t,x,\widetilde{y})$ in $(t,x,\widetilde{y}).$ And then by the representation~\eqref{e:tilde{Y}^0 = ..} and the positivity of $\frac{1}{D_{y}\phi^{T-h^{(i)}}}$ from Lemma~\ref{lem:boundedness of the flow}, it follows that $Y^{0}_{t}\ge Y^{0}_{t+}.$ The proof is complete.
\end{proof}

\section{Applications}\label{sec:applications}

\subsection{Obstacle problem for rough PDEs}\label{subsec:PDE obstacle}

In this subsection, we suppose that $f:[0,T]\times \R^{d} \times \R \times \R^{d}\rightarrow \R$ is deterministic and jointly continuous.  Given $b:[0,T]\times\R^{d}\rightarrow\R^{d},$ $\sigma:[0,T]\times\R^{d}\rightarrow\R^{d\times d},$ $l:[0,T]\times\R^{d}\rightarrow \R,$  $g:\R^{d}\rightarrow \R$ and $\mathbf{X}\in C^{0,p\text{-var}}([0,T],G^{[p]}(\R^{l})),$ consider the following system of forward-backward stochastic differential equations
\begin{equation}\label{e:Markovian}
\left\{\begin{aligned}
&Y^{t,x}_{s} = g(S^{t,x}_{T}) + \int_{s}^{T}f(r,S^{t,x}_r,Y^{t,x}_r,Z^{t,x}_r)dr + \int_{s}^{T}H(S^{t,x}_r,Y^{t,x}_r)d\mathbf{X}_{r}\\
&\quad \quad \quad \quad  - \int_{s}^{T}Z^{t,x}_r dW_{r} + K^{t,x}_{T} - K^{t,x}_{s}, \\
& Y^{t,x}_{s}\ge l(s,S^{t,x}_{s}), \quad \int_{t}^{T} \left(Y^{t,x}_{r} - l(r,S^{t,x}_{r})\right) dK^{t,x}_{r} = 0 , \quad s\in[t,T], \\
&S^{t,x}_{s} = x + \int_{t}^{s} b(r,S^{t,x}_{r}) dr + \int_{t}^{s} \sigma(r,S^{t,x}_{r}) dW_r,\quad s\in[t,T],
\end{aligned}
\right.
\end{equation}
and consider the corresponding rough PDE
\begin{equation}\label{e:rough PDE'}
\left\{\begin{aligned}
&d v(t,x) + \mathcal{L}v(t,x)dt + H(x,v(t,x))d\mathbf{X}_t = 0,\\
&v(T,x) = g(x) ,\quad v(t,x)\ge l(t,x),
\end{aligned}\right.
\end{equation}
where (we stipulate that $(\sigma,b)$ takes value at $(t,x)$)
\begin{equation*}
\mathcal{L}v(t,x) := \frac{1}{2}\text{Tr}\left\{(\sigma\sigma^{\top}D^{2}_{xx}) v(t,x) \right\} + (b^{\top}D_{x}) v(t,x) + f\big(t,x,v(t,x), \sigma^{\top}D_{x}v(t,x)\big).
\end{equation*}

Assume \ref{(Hpr)}, \ref{(H0)}, \ref{(H_b,sigma)} hold for $(H,f,b,\sigma).$ Assume further $(l,g) \in C_b([0,T]\times \R^{d},\R) \times C_b(\R^{d},\R)$ such that $g(x)\ge l(T,x).$ For a smooth driver $d\mathrm{X}_t$ (i.e. $\mathrm{X}\in\bigcap_{\lambda\ge 1}\text{Lip}^{\lambda}([0,T],\R^l)$), $f'(t,x,y,z) := f(t,x,y,z) + H(x,y)\dot{\mathrm{X}}_t$ also satisfies \ref{(H0)}. Thus, by 
Proposition~\ref{prop:exist of PDE}, the PDE \eqref{e:rough PDE'} driven by $d\mathrm{X}_t$ (i.e. with $d\mathbf{X}_t$ replaced by $d\mathrm{X}_t$) admits a unique viscosity solution.
On the other hand, according to Theorem~\ref{thm:well-posedness}, the system~\eqref{e:Markovian} admits a unique solution $(Y^{t,x},Z^{t,x},K^{t,x})$. We aim to establish the relation between $Y^{t,x}$ and the solution of rough PDE~\eqref{e:rough PDE'}. For this purpose, we first introduce the following definition inspired by \cite[Theorem~12]{DF}.

\begin{definition}\label{def:rough PDE}
For any $\mathbf{X} \in C^{0,p\text{-var}}([0,T],G^{[p]}(\mathbb{R}^{l})),$ for $n \ge 1$ let $\mathrm{X}^{n}$ be a smooth path such that their lifts $\mathbf{X}^{n} \rightarrow \mathbf{X}$ in $C^{0,p\text{-var}}([0,T],G^{[p]}(\mathbb{R}^{l})).$ Assume that $u^{n}(t,x)$ is the unique viscosity solution of obstacle PDE~\eqref{e:rough PDE'} driven by $d\mathrm{X}^{n}_t$ (i.e. with $\mathbf{X}$ replaced by $\mathrm{X}^{n}$), and there exists $u \in C([0,T]\times\mathbb{R}^d,\mathbb{R})$ such that
\[ u^{n} \rightarrow u, \quad \text{locally uniformly.} \]
Then we call $u(t,x)$ a solution of the rough PDE~\eqref{e:rough PDE'}. Moreover, if $u(t,x)$ does not depend on the choice of the approximation sequence $\left\{\mathrm{X}^n,n\ge 1\right\},$ we call $u(t,x)$ the unique solution of Eq.~\eqref{e:rough PDE'}.
\end{definition}

\begin{theorem}
\label{thm:rough PDE}
 
Assume that $(H,f)$ satisfies Assumption~\ref{(Hpr)}, \ref{(H0)}, $(l,g) \in C_b([0,T]\times \R^{d},\R) \times C_b(\R^{d},\R)$ such that $g(x)\ge l(T,x),$ and $(b,\sigma)$ satisfies \ref{(H_b,sigma)}. In addition, suppose that $f$ satisfies 
\begin{equation*}
|D_{x}f(t,x,y,z)|\le C'( 1 + |z|^{2} ), 
\end{equation*} 
for some $C'>0$. For any $\mathbf{X}\in C^{0,\text{p-var}}([0,T],G^{[p]}(\R^l))$, let $u(t,x):=Y^{t,x}_t,$ where $(Y^{t,x},Z^{t,x},K^{t,x})$ is the unique solution of Eq.~\eqref{e:Markovian}. Then $u(t,x)$ is the unique solution of the rough PDE~\eqref{e:rough PDE'}.
\end{theorem}

\begin{proof}
First we introduce some notations that will be used for our proof. Denote $\mathbf{X}^{0} := \mathbf{X}$ and $(b,\sigma):=(b(t,x),\sigma(t,x)).$ For $T'\in[0,T],$ let $\phi^{n,T'},n\in\mathbb{N}$ be the unique solution of Eq.~\eqref{e:phi^n}, and let $\mathrm{X}^{n}$ be smooth with their lifts $\mathbf{X}^{n} \rightarrow \mathbf{X}^0$ in $C^{0,p\text{-var}}([0,T],G^{[p]}(\R^l)).$
In view of Proposition~\ref{prop:exist of PDE}, for 
$n\ge 1,$ Eq.~\eqref{e:rough PDE'} driven in $d\mathrm{X}^n_t$ admits a unique viscosity solution $u^n.$ On the other hand, by the monotonicity of $\phi^{n,T'}(t,x,y)$ and $(\phi^{n,T'})^{-1}(t,x,y)$ in $y,$  
\begin{equation}\label{e:u^n > l}
u^n(t,x)>l(t,x) \Leftrightarrow (\phi^{n,T'})^{-1}(t,x,u^n(t,x)) > (\phi^{n,T'})^{-1}(t,x,l(t,x)).
\end{equation}
Let $\widetilde{f}^{n,T'}(t,x,\widetilde{y},\widetilde{z})$ be defined by \eqref{e:tilde f^n,T'}, and $\widetilde{l}^{n,T'}(t,x) : = (\phi^{n,T'})^{-1}(t,x,l(t,x)).$ According to \eqref{e:u^n > l} and \cite[Lemma~6]{friz2014rough}, it follows that 
for $n\ge 1,$ $u^{n}$ is a viscosity solution to Eq.~\eqref{e:rough PDE'} driven by $d\mathbf{X}^{n}_t$ if and only if $v^{n,T'}(t,x) := (\phi^{n,T'})^{-1}(t,x,u^{n}(t,x)) $ is a viscosity solution of the following equation
\begin{equation}\label{e:transformed PDE}
\left\{\begin{aligned}
&d v(t,x) + \left[\frac{1}{2}\text{Tr}\Big\{(\sigma\sigma^{\top}D^{2}_{xx}v(t,x))\right\} + b^{\top}D_{x} v(t,x)\\
&\ \ \ \ \ \  + \widetilde{f}^{n,T'}(t,x,v(t,x), \sigma^{\top}D_{x} v(t,x)) \Big]dt = 0\\
&v(T',x) = u^{n}(T',x) ,\quad v(t,x)\ge \widetilde{l}^{n,T'}(t,x),\quad (t,x)\in [0,T']\times \R^{d}.
\end{aligned}\right.
\end{equation}
For $n=0,$ we denote $\widetilde{l}^{0,T'}(t,x) : = (\phi^{0,T'})^{-1}(t,x,l(t,x)),$ and define $ \widetilde{f}^{0,T'}(t,x,\widetilde{y},\widetilde{z})$ by \eqref{e:tilde f^n,T'} with $n=0$ therein. Now, we will divide our proof into the following three parts.\\

\textbf{Step 1.}  
Assuming additional $u^{n}(T',\cdot)\rightarrow u(T',\cdot),$
we will use the method of semi-relaxed limits (refer to \cite[Section~6]{crandall1992user}) to show that the limit
\begin{equation*}
\overline{v}^{0,T'} := \limsup_{n\rightarrow \infty}{}^{*} v^{n,T},\quad (\text{resp. }\underline{v}_{0,T'} := \liminf_{n\rightarrow \infty}{}_{*} v^{n,T'}),
\end{equation*}
is a viscosity sub  (resp. super) solution of Eq.~\eqref{e:transformed PDE} with $n=0,$ where 
\begin{equation}\label{e:semi-relaxed}
\begin{aligned}
\limsup_{n\rightarrow \infty}{}^{*} v^{n,T'}(t,x) &:= \lim_{n\rightarrow\infty}\sup\left\{v^{m,T'}(s,y): |s-t|\vee|y-x|\le\frac{1}{n},\ \text{and }m\ge n\right\},\\
\liminf_{n\rightarrow \infty}{}_{*} v^{n,T'}(t,x) &:= \lim_{n\rightarrow\infty}\inf\left\{v^{m,T'}(s,y): |s-t|\vee|y-x|\le\frac{1}{n},\ \text{and }m\ge n\right\}.
\end{aligned}
\end{equation}
To see the finiteness of the limits in \eqref{e:semi-relaxed}, for each $n\ge 1$, the Feynman–Kac representation shown in Proposition~\ref{prop:exist of PDE} and the prior bound of the corresponding BSDEs given by Lemma~\ref{lem:A Prior Estimate} imply that $u^{n}(t,x)$ are uniformly bounded in $(t,x)\in [0,T]\times \R^d$ and $n\ge 1.$ In addition, due to the uniform bound of $\|D_y (\phi^{n,T'})^{-1}\|_{\infty;[0,T']\times\R^d\times\R}$ in $n\ge 1,$ the uniform boundedness of $u^{n}(t,x)$ implies that of $v^{n,T},$ i.e., 
\begin{equation*}
\sup_{n\ge 1}\|v^{n,T'}(\cdot,\cdot)\|_{\infty;[0,T']\times\R^d}<\infty.
\end{equation*}
Now, we will show that $\overline{v}^{0,T'}(\cdot,\cdot)$ (resp. $\underline{v}^{0,T'}(\cdot,\cdot)$) is a subsolution (resp. supersolution). Firstly, by \cite[Lemma~1.5 Chapter~V]{bardi1997optimal} we have that $\overline{v}^{0,T'}(\cdot,\cdot)$ (resp. $\underline{v}^{0,T'}(\cdot,\cdot)$) is upper (resp. lower) semicontinuous. 
Secondly, by Lemma~\ref{lem:stability of the flow}, $\widetilde{l}^{n,T'}(\cdot,\cdot)\rightarrow \widetilde{l}^{0,T'}(\cdot,\cdot)$ locally uniformly; and by Lemma~\ref{lem:locally uniform conver}, $\widetilde{f}^{n,T'}\rightarrow \widetilde{f}^{0,T'}$ locally uniformly. In view of $u^{n}(T',\cdot)\rightarrow u(T',\cdot)$ and these two facts,  Proposition~\ref{prop:preserve terminal value} applies, and we have $\overline{v}^{0,T'}(T',x) = \underline{v}^{0,T'}(T',x) = u(T',x).$ Thirdly, assume that $(t,x)\in(0,T')\times\R^{d}$ satisfies 
\begin{equation}\label{e:v0 > l0}
\overline{v}^{0,T'}(t,x)>\widetilde{l}^{0,T'}(t,x),
\end{equation} 
and let $(a,p,\Lambda)\in P^{2,+}\overline{v}^{0,T'}(t,x)$, where the semi-jets $P^{2,+}$ are defined in \ref{append:B}. Using the same argument as in \cite[Lemma~6.1]{crandall1992user}, there exist sequences 
\begin{equation*}
n_{j}\rightarrow \infty,\ 
(t_{j},x_{j})\in(0,T')\times\R^{d},\ (a_{j},p_{j},\Lambda_{j})\in P^{2,+}v^{n_{j},T'}(t_{j},x_{j}),\quad j\ge 1,
\end{equation*}
such that
\begin{equation*}
(t_{j},x_{j},v^{n_{j},T'}(t_{j},x_{j}),a_{j},p_{j},\Lambda_{j})\rightarrow (t,x,\overline{v}^{0,T'}(t,x),a,p,\Lambda),\quad j\rightarrow \infty.
\end{equation*}
In view of the locally uniform convergence of $\widetilde{l}^{n_{j},T'}(\cdot,\cdot) \rightarrow \widetilde{l}^{0,T'}(\cdot,\cdot)$ 
and the inequality \eqref{e:v0 > l0},
we have that $v^{n_{j},T'}(t_{j},x_{j})>\widetilde{l}^{n_{j},T'}(t_{j},x_{j})$ for sufficiently large $j.$ Due to the viscosity property of $v^{n,T'},$ we have 
\begin{equation*}
a_{j} + (p_{j}^{\top}b(t_{j},x_{j})) + \frac{1}{2}\text{Tr}\left\{\sigma\sigma^{\top}(t_{j},x_{j})\Lambda_{j}\right\} + \widetilde{f}^{n_{j},T'}(t_{j},x_{j},v^{n_{j},T'}(t_{j},x_{j}),\sigma^{\top}(t_{j},x_{j})p_{j})\ge 0.
\end{equation*}
According to the fact that $\widetilde{f}^{n,T'}\rightarrow \widetilde{f}^{0,T'}$ locally uniformly, by sending $j\rightarrow \infty$ we have
\begin{equation*}
a + (p^{\top}b(t,x)) + \frac{1}{2}\text{Tr}\left\{\sigma\sigma^{\top}(t,x)\Lambda\right\} + \widetilde{f}^{0,T'}(t,x,\overline{v}^{0}(t,x),\sigma^{\top}(t,x)p)\ge 0,
\end{equation*}
which implies that $\overline{v}^{0,T'}$ is a viscosity subsolution to Eq.~\eqref{e:transformed PDE} with $n=0.$ In addition, the uniform continuity of $l(\cdot,\cdot)$ implies that $\underline{v}^{0,T'}(\cdot,\cdot)\ge \widetilde{l}^{0,T'}(\cdot,\cdot),$ and then we can repeat the above procedure to show that $\underline{v}^{0,T'}$ is a viscosity supersolution to Eq.~\eqref{e:transformed PDE} with $n=0.$\\ 

\textbf{Step 2.}
Following the above step, we will show that $v^{n,T'}(\cdot,\cdot)$ locally uniformly converges to a limit.
Note that $\overline{v}^{0,T'}$ ($\underline{v}^{0,T'}$ resp.) is the viscosity subsolution (supersolution resp.). In view of Lemma~\ref{lem40}, we see that comparison theorem Proposition~\ref{prop:comparison of PDE} applies. Thus, there exists $\delta>0$ such that for whenever $s\in[0,T')$ and $\|\mathbf{X}^{n}\|_{p\text{-var};[s,T']}\le \delta,$  
\begin{equation*}
\overline{v}^{0,T'}(t,x)\le \underline{v}^{0,T'}(t,x)\text{ for }(t,x)\in(s,T']\times\R^{d}.
\end{equation*}
And then by the definition we have $\overline{v}^{0,T'} = \underline{v}^{0,T'}.$  Following the Dini-type argument \cite[Remark~6.4]{crandall1992user}, 
for every $r\in(s,T')$ we have that 
\begin{equation}\label{e:vn uniform converg}
v^{n,T'}(\cdot,\cdot)\rightarrow \overline{v}^{0,T'}(\cdot,\cdot) \text{ locally uniformly on }
 [r,T']\times\R^{d}.
\end{equation}

\textbf{Step 3.}
Now we will prove the convergence of $u^n(\cdot,\cdot).$ Let $\widetilde{u}^{0,T'}(t,x) := \phi^{0,T'}(t,x,\overline{v}^{0,T'}(t,x)).$ Then, Lemma~\ref{lem:stability of the flow} and \eqref{e:vn uniform converg} yield that 
\begin{equation*}
u^{n}(\cdot,\cdot)\rightarrow \widetilde{u}^{0,T'}(\cdot,\cdot) \text{ locally uniformly on }
[r,T']\times\R^{d}.
\end{equation*}
On the other hand, for $n\ge 1,$ denote by $Y^{n;t,x}$ the unique solution of RBSDE 
\eqref{e:Markovian} with $\mathbf{X} = \mathbf{X}^{n}.$ Recall that $u(t,x) = Y^{t,x}_{t}.$ 
Then by Proposition~\ref{prop:exist of PDE} and Theorem~\ref{thm:well-posedness}, for $n\ge 1,$ $u^{n}(t,x) = Y^{n;t,x}_{t}$ and $u^{n}(t,x) \rightarrow Y^{t,x}_{t} = u(t,x).$ Thus, assuming that the locally uniform convergence of $u^{n}(T',\cdot)\rightarrow u(T',\cdot)$ holds, we have that 
$\widetilde{u}^{0,T'}(t,x) = u(t,x)$ on $[r,T']\times\R^d$ and hence
\begin{equation}\label{e:vn uniform converg'}
u^{n}(\cdot,\cdot)\rightarrow u(\cdot,\cdot) \text{ locally uniformly on } [r,T']\times\R^{d}.
\end{equation}
Let $\theta\in(0,T)$ satisfying $\sup_n\|\mathbf{X}^n\|_{p\text{-var};[0,\theta]}\le \delta/2.$ By the fact that $\mathbf{X}^n\rightarrow \mathbf{X}^0$ in $p$-variation metric, and Lemma~\ref{lem:length of the subinterval} with $[0,T]$ replaced by $[\theta,T],$ there exists a sequence of increasing numbers $\{h^{(i)},i=0,1,...,N\}$ satisfying 
\begin{equation*}
h^{(i)} < h^{(i+1)},\ i=0,1,...,N-1,\ h^{(N)} = T - \theta,\ h^{(0)} = 0,
\end{equation*}
and (take $n$ large enough if necessary)
\begin{equation}\label{e:<=delta/2}
\sup_{n}\|\mathbf{X}^n\|_{p\text{-var};[T-h^{(i)},T-h^{(i-1)}]}\le \frac{\delta}{2},\text{ for all }i=1,2,...,N.
\end{equation}
By \eqref{e:<=delta/2} and the fact that $\|\mathbf{X}^n\|_{p\text{-var};[0,\theta]}\le \delta/2,$ we have for $i < N$
\begin{equation*}
\|\mathbf{X}^n\|_{p\text{-var};[T-h^{(i+1)},T-h^{(i-1)}]}\le \delta,\text{ and }\|\mathbf{X}^n\|_{p\text{-var};[0,T-h^{(N-1)}]}\le\delta.
\end{equation*}
Then in view of the above inequalities and the fact that 
\[T-h^{(i)}\in(T-h^{(i+1)},T-h^{(i-1)}),\ T-h^{(N)}\in(0,T-h^{(N-1)}),\]
by \eqref{e:vn uniform converg'} it follows that whenever $u^{n}(T - h^{(i-1)},\cdot)\rightarrow u(T - h^{(i-1)},\cdot)$ locally uniformly,  
\begin{equation}\label{e:u^n -> u locally for i}
u^{n}(\cdot,\cdot)\rightarrow u(\cdot,\cdot)\text{ locally uniformly on }[T-h^{(i)},T-h^{(i-1)}]\times\R^d.
\end{equation}
Since $u^{n}(T,\cdot) = g(\cdot) = u(\cdot)$ for all $n,$ \eqref{e:u^n -> u locally for i} holds for $i=1.$ Thus by induction, the uniform convergence of $u^{n}(T-h^{(i-1)},\cdot)\rightarrow u(T-h^{(i-1)},\cdot)$ can be deduced, and then \eqref{e:u^n -> u locally for i} holds for every $i>1.$ Since $[T-h^{(N)},T-h^{(0)}] = [\theta,T],$ we have 
\begin{equation*}
u^{n}(\cdot,\cdot)\rightarrow u(\cdot,\cdot)\text{ locally uniformly on }[\theta,T]\times\R^d.
\end{equation*}
And then the proof is complete by taking the terminal time $T$ replaced $T+\theta,$ and doing the transition of time $t\mapsto t+\theta.$
\end{proof}

\begin{remark}\label{rem:the point except for 0}

In the above proof, by the same reason mentioned in \cite[Theorem~12]{DF},  the convergence at time $t=0$ is not immediate, due to that the comparison result in Proposition~\ref{prop:comparison of PDE} may not hold at the left endpoint.
However, we can take the transition of time: $t\mapsto t + \theta,$ and extend the coefficients as well as the rough path up to time $0.$ Then the rough PDEs are considered on $[0,T+\theta]$ to get the locally uniform convergence of corresponding $u^n(t-\theta,x)$ on the interval $[\theta,T+\theta].$ Therefore, the convergence of $u^n$ on $[0,T]$ follows by doing the reverse transition of time: $t\mapsto t - \theta.$ 
\end{remark}

\begin{remark}\label{rem:coincide with classical PDE}
Assume $(H,f,l,g,b,\sigma)$ satisfies the requirements in Proposition~\ref{prop:exist of PDE}, and $\mathrm{X}$ is smooth. Let $u$ be the unique viscosity solution of obstacle PDE~\eqref{e:rough PDE'} driven by $d\mathrm{X}_t.$ Then from Theorem~\ref{thm:rough PDE}, it follows that $u$ is also the unique solution in the sense of Definition~\ref{def:rough PDE}.
\end{remark}

\subsection{Optimal stopping problem}\label{subsec:Optimal Stopping}

In this subsection, we aim to show that the solution $Y_{t}$ of Eq.~\eqref{58} with parameter  $(\xi,f,L,H,\mathbf{X})$ corresponds to an optimal stopping time problem in the context of BSDEs with rough drivers $d\mathbf{X}_t$. In the classical case, one of the tasks of optimal stopping problem is to characterize the following value function
\begin{equation*}
\esssup_{\tau\in\mathcal{T}_{t}}\mathbb{E}_{t}\left[G_{\tau}\right],
\end{equation*}
where $G$ is an $\mathcal{F}$-adapted process, $\mathcal{T}_{t}$ is the set of all stopping times taking values in $[t,T].$ Assuming $(Y,Z,K)$ uniquely solves RBSDE~\eqref{e:rBSDE}, it follows from El Karoui et al. \cite[Theorem~5.2]{el1997non} that
\begin{equation*}
Y_{t}  = \esssup_{\tau \in \mathcal{T}_{t}} B_{t}(\tau, L'_{\tau}),
\end{equation*}
where $B_{t}(\tau,L'_{\tau})$ represents the solution to BSDE~\eqref{e:BSDE} with terminal time $\tau,$ terminal value $L'_{\tau},$ and
\begin{equation}\label{L'}
L'_{s} := L_{s}\mathbf{1}_{\{s<T\}} + \xi\mathbf{1}_{\{s=T\}}.
\end{equation}
In the following, we will establish the well-posedness of rough RBSDEs with random terminal times and provide a characterization of the optimal stopping problem in the rough path setting.

Fix a $\tau\in \mathcal{T}_0.$ For $\xi\in L^{\infty}(\mathcal{F}_{\tau}),$ consider the following equation 
\begin{equation}\label{58'}
\left\{\begin{aligned}
Y_{t} &= \xi + \int_{t\land \tau}^{\tau} f(r,S_{r},Y_{r},Z_{r})dr + \int_{t\land \tau}^{\tau} H(S_{r},Y_{r})d\mathbf{X}_{r} - \int_{t\land \tau}^{\tau} Z_{r} dW_{r} + K_{\tau} - K_{t\land \tau}\\
Y_{t}&\ge L_{t\land \tau},\quad \int_{0}^{\tau}(Y_{r} - L_{r}) dK_{r} = 0 ,\quad t\in[0,T],
\end{aligned}\right.
\end{equation}
where $L_{\tau}\le \xi.$ Set  
\begin{equation}\label{e:hat f, hat L, hat H}
\begin{aligned}
\left(\hat{f}(t,x,y,z),\hat{L}_{t},\hat{H}(t,x,y)\right) := \left(f(t,x,y,z) \mathbf{1}_{\{t\le \tau\}},L_{t\land \tau},H(x,y) \mathbf{1}_{\{t\le \tau\}}\right).
\end{aligned}
\end{equation}

\begin{definition}\label{def:stopping terminal}
Let $ \{\mathrm{X}^{n}\}_{n \ge 1}$ be a sequence of smooth paths with lifts $\mathbf{X}^{n} \rightarrow \mathbf{X}$ in $C^{0,p\text{-var}}([0,T],G^{[p]}(\R^{l})).$ Assume that for each $n\ge 1$ Eq.~\eqref{58} with parameter $(\xi,\hat{f},\hat{L},\hat{H},\mathrm{X}^{n})$ admits a unique solution $(Y^{n},Z^{n},K^{n})\in \mathrm{H}^{\infty}_{[0,T]}\times \mathrm{H}^{2}_{[0,T]}(\R^{d})\times\mathrm{I}^{1}_{[0,T]}$. We call a triple $(Y,Z,K)\in \mathrm{H}^{\infty}_{[0,T]}\times \mathrm{H}^{2}_{[0,T]}(\R^{d})\times\mathrm{I}^{1}_{[0,T]}$
 a solution of Eq.~\eqref{58'} with terminal time $\tau$ and parameter $(\xi,f,L,H,\mathbf{X})$ if 
\begin{equation*}
(Y^{n}_{\cdot},K^{n}_{\cdot}) \rightarrow (Y_{\cdot},K_{\cdot})\  \text{ uniformly on }\ [0,T]\  \text{ a.s., }\ 
Z^n \rightarrow Z \ \text{ in }\ \mathrm{H}^{2}_{[0,T]}(\mathbb{R}^d),
\end{equation*}
and 
\begin{equation*}
\int_{0}^{\tau}(Y_{r} - L_{r})dK_{r} = 0,
\end{equation*} 
If  $(Y,Z,K)$ is independent of the choice of $\left\{\mathrm{X}^{n},n\ge 1\right\},$ we call $(Y,Z,K)$ the unique solution of Eq.~\eqref{58'} with terminal time $\tau$ and parameter $(\xi,f,L,H,\mathbf{X}).$
\end{definition}

Now we establish the well-posedness of Eq.~\eqref{58'} by adapting the arguments from Theorem~\ref{thm:well-posedness} to a stopping time scenario.

\begin{proposition}\label{prop:extension of the well-posedness}
Under the same assumptions as in Theorem~\ref{thm:well-posedness}, and assume further that $\xi\in L^{\infty}(\mathcal{F}_{\tau}),$ where $\tau \in \mathcal{T}_{0}.$ Then there exists a unique solution $(Y,Z,K)$ of Eq.~\eqref{58'} with terminal time $\tau$ and parameter $(\xi,f,L,H,\mathbf{X}).$ 
\end{proposition}

\begin{proof}
Let $\left\{\mathrm{X}^{n}:[0,T]\rightarrow\R^{l},n\ge 1\right\}$ be a sequence of smooth paths such that their lifts $\mathbf{X}^{n} \rightarrow \mathbf{X}$ in $C^{0,p\text{-var}}([0,T],G^{[p]}(\R^{l})).$ Denoted $\mathbf{X}^{0}:=\mathbf{X}.$ For $n\ge 1$, let $(Y^{n},Z^{n},K^{n})$ be the unique solution of Eq.~\eqref{58} with parameter $(\xi,\hat{f},\hat{L},\hat{H},\mathrm{X}^{n}),$ where $(\hat{f},\hat{L},\hat{H})$ is defined in \eqref{e:hat f, hat L, hat H}. For $n\ge 1,$ $T'\in[0,T],$ and $t\in[0,T'],$ we define  $(\widetilde{Y}^{n,T'},\widetilde{Z}^{n,T'},\widetilde{K}^{n,T'})$ by
\begin{equation}\label{e:tilde Y^n,T'_tau}
\begin{aligned}
&\widetilde{Y}^{n,T'}_{t}:=(\phi^{n,T'})^{-1}(t\land \tau,S_{t\land \tau},Y^n_{t\land \tau}), \  \widetilde{K}^{n,T'}_{t}:=\int_{0}^{t\land \tau} \frac{1}{D_y \phi^{n,T'}(s,S_s,\widetilde{Y}^{n,T'}_s)}dK^n_s,\\
&\widetilde{Z}^{n,T'}_t:= \mathbf{1}_{\{t\le \tau\}}\left(-\frac{(D_x \phi^{n,T'}(t,S_{t},\widetilde{Y}^{n,T'}_{t}))^{\top}}{D_y\phi^{n,T'}(t,S_{t},\widetilde{Y}^{n,T'}_{t})}\sigma_{t} + \frac{1}{D_y \phi^{n,T'}(t,S_{t},\widetilde{Y}^{n,T'}_{t})}Z^{n}_{t}\right),
\end{aligned}
\end{equation}
where $\phi^{n,T'}$ satisfies Eq.~\eqref{e:phi^n}. For $n\ge 0$ let $(\widetilde{f}^{n,T'},\widetilde{L}^{n,T'})$ be given by \eqref{e:tilde f^n,T'}. For $n\ge 1$, applying It\^o's formula to $(\phi^{n,T'})^{-1}(t,S_{t},Y^{n}_{t})$ and using \eqref{e:Skorohod tilde Y}, we have for $t\in[0,T']$
\begin{equation*}
\left\{\begin{aligned}
\widetilde{Y}^{n,T'}_{t\land \tau} &= (\phi^{n,T'})^{-1}(T'\land\tau,S_{T'\land\tau},Y^{n}_{T'\land \tau}) + \int_{t\land \tau}^{T'\land \tau}\widetilde{f}^{n,T'}(r,S_{r},\widetilde{Y}^{n,T'}_{r},\widetilde{Z}^{n,T'}_{r}) dr\\
&\qquad\qquad\qquad   - \int_{t\land\tau}^{T'\land \tau} \widetilde{Z}^{n,T'}_{r} dW_{r} + \widetilde{K}^{n,T'}_{T'\land \tau} - \widetilde{K}^{n,T'}_{t\land \tau},\\
\widetilde{Y}^{n,T'}_{t\land \tau} &\ge \widetilde{L}^{n,T'}_{t\land \tau},\quad \int_{0}^{T'\land \tau}\left(\widetilde{Y}^{n,T'}_{r} - \widetilde{L}^{n,T'}_r \right) d\widetilde{K}^{n}_{r} = 0,
\end{aligned}\right.
\end{equation*}
and then the fact $(\widetilde{Y}^{n,T'}_{t},\widetilde{Z}^{n,T'}_{t},\widetilde{K}^{n,T'}_{t}) = (\widetilde{Y}^{n,T'}_{t\land \tau},\widetilde{Z}^{n,T'}_{t}\mathbf{1}_{\{t\le \tau\}},\widetilde{K}^{n,T'}_{t\land \tau})$ yields that for $t\in[0,T']$
\begin{equation*}
\left\{\begin{aligned}
\widetilde{Y}^{n,T'}_{t} &= (\phi^{n,T'})^{-1}(T'\land\tau,S_{T'\land\tau},Y^{n}_{T'\land \tau}) + \int_{t}^{T'}\widetilde{f}^{n,T'}(r,S_{r},\widetilde{Y}^{n,T'}_{r},\widetilde{Z}^{n,T'}_{r})\mathbf{1}_{\{r\le \tau\}} dr  \\
&\qquad\qquad\qquad - \int_{t}^{T'} \widetilde{Z}^{n,T'}_{r} dW_{r} + \widetilde{K}^{n,T'}_{T} - \widetilde{K}^{n,T'}_{t},\\
\widetilde{Y}^{n,T'}_{t} &\ge \widetilde{L}^{n,T'}_{t\land\tau},\quad \int_{0}^{T'}\left(\widetilde{Y}^{n,T'}_{r} - \widetilde{L}^{n,T'}_{r\land \tau}\right)d\widetilde{K}^{n}_{r} = 0,
\end{aligned}\right.
\end{equation*}
which implies that $(\widetilde{Y}^{n,T'},\widetilde{Z}^{n,T'},\widetilde{K}^{n,T'})$ is a solution of Eq.~\eqref{58} with terminal time $T'$ and parameter $(\widetilde{\xi}^{n,T'},\widetilde{f}^{n,T'}(t,x,\widetilde{y},\widetilde{z})\mathbf{1}_{\{t\le \tau\}},\widetilde{L}^{n,T'}_{t\land\tau},0,0),$ where
\begin{equation}\label{e:xi^n,T'}
\widetilde{\xi}^{n,T'}:= (\phi^{n,T'})^{-1}(T'\land\tau,S_{T'\land\tau},Y^{n}_{T'\land\tau}).
\end{equation}

The rest of the proof is similar to the procedure of \eqref{e:bounde of Yn}--\eqref{e:sum k} in the proof of Theorem~\ref{thm:well-posedness}, so we will just illustrate it briefly. By the same argument used to obtain $\eqref{e:bounde of Yn},$ there exists a constant $\widetilde{M}>0$ such that for all $n\ge 1$ and $T'\in[0,T]$ 
\begin{equation*}
\esssup_{(t,\omega)\in[0,T']\times\Omega}|\widetilde{Y}^{n,T'}_t|\le \widetilde{M}.
\end{equation*}
In addition, by the same argument in the step 2 of the proof of Theorem~\ref{thm:well-posedness}, there exist an integer $N>0$ and a set of increasing real numbers $\left\{h^{(i)},i=0,1,...,N\right\}$ such that $[T-h^{(N)},T-h^{(0)}] = [0,T].$ Furthermore, the Eq.~\eqref{58'} with terminal time $T - h^{(i)}$ and parameter 
\[(\widetilde{\xi}^{n,T - h^{(i)}},\widetilde{f}^{n,T - h^{(i)}}(t,x,\widetilde{y},\widetilde{z})\mathbf{1}_{\{t\le \tau\}},\widetilde{L}^{n,T - h^{(i)}}_{t\land\tau},0,0)\] 
admits a unique solution on the interval $[T-h^{(i+1)},T - h^{(i)}].$ For $n\ge 1$ this unique solution coincides with $(\widetilde{Y}^{n,T-h^{(i)}},\widetilde{Z}^{n,T-h^{(i)}},\widetilde{K}^{n,T-h^{(i)}});$ for $n=0,$ we denote this unique solution by $(\widetilde{Y}^{0,T-h^{(i)}},\widetilde{Z}^{0,T-h^{(i)}},\widetilde{K}^{0,T-h^{(i)}}).$ 
For each $i=0,1,...,N-1$ and $t\in[T-h^{(i+1)},T-h^{(i)}],$ define by  
\begin{equation}\label{e:Y0 Z0}
\begin{aligned}
&Y^0_t:=\phi^{0,T-h^{(i)}}(t\land \tau,S_{t\land \tau},\widetilde{Y}^{0,T-h^{(i)}}_{t}), \\
&K^{0,T-h^{(i)}}_t:=\int_{T-h^{(i+1)}}^{t\land \tau} D_y \phi^{0,T-h^{(i)}}(s,S_s,\widetilde{Y}^{0,T-h^{(i)}}_s)d\widetilde{K}^{0,T-h^{(i)}}_s,\\ 
&Z^0_t:= \mathbf{1}_{\{t\le \tau\}}\cdot D_y\phi^{0,T-h^{(i)}}(t,S_t,\widetilde{Y}^{0,T-h^{(i)}}_t)\left[\widetilde{Z}^{0,T-h^{(i)}}_t+\frac{(D_x\phi^{0,T-h^{(i)}}(t,S_t,\widetilde{Y}^{0,T-h^{(i)}}_t))^{\top}}{D_y\phi^{0,T-h^{(i)}}(t,S_t,\widetilde{Y}^{0,T-h^{(i)}}_t)}\sigma_t\right].
\end{aligned}
\end{equation}
Then, using  Lemma~\ref{lem:continu solution map}, and performing an induction indexed by $i=0,1,...,N-1$ similar to the one in the proof of Theorem~\ref{thm:well-posedness}, it follows that
\begin{equation*}
\begin{aligned}
&\widetilde{Y}^{n,T-h^{(i)}}_{\cdot}\rightarrow \widetilde{Y}^{0,T-h^{(i)}}_{\cdot}, \  \widetilde{K}^{n,T-h^{(i)}}_{\cdot} - \widetilde{K}^{n,T-h^{(i)}}_{T-h^{(i+1)}}\rightarrow \widetilde{K}^{0,T-h^{(i)}}_{\cdot} - \widetilde{K}^{0,T-h^{(i)}}_{T-h^{(i+1)}} \\
&\qquad\qquad\qquad\qquad\qquad\qquad\textrm{ uniformly on } [T-h^{(i+1)},T-h^{(i)}] \textrm{ a.s.,}\\
&\widetilde{Z}^{n,T-h^{(i)}}\rightarrow \widetilde{Z}^{0,T-h^{(i)}} \textrm{ in } \mathrm{H}^{2}_{[T-h^{(i+1)},T-h^{(i)}]},
\end{aligned}
\end{equation*}
and then 
\begin{equation*}
\begin{aligned}
&{Y}^n_{\cdot}\rightarrow {Y}^0_{\cdot}, \  {K}^n_{\cdot}\rightarrow {K}^{0}_{\cdot} \textrm{ uniformly on } [0,T] \textrm{ a.s.,}\\
&{Z}^n\rightarrow {Z}^0 \textrm{ in } \mathrm{H}^{2}_{[0,T]},
\end{aligned}
\end{equation*}
where $K^{0}_{0} := 0$ and
\begin{equation*}
K^{0}_{t} := \sum_{j=1}^{N-j-1}K^{0,T-h^{(N-j)}}_{T-h^{(N-j)}} + K^{0,T-h^{(i)}}_{t},\quad t\in(T-h^{(i+1)}, T-h^{(i)}],\quad i=0,1,...,N-1.
\end{equation*}
\end{proof}

\begin{remark}\label{rem:property of solution by stopping}
For the solutions obtained by Proposition~\ref{prop:extension of the well-posedness}, we claim that similar results to those in Corollary~\ref{cor:comparison} and Proposition~\ref{prop:conti of (Y^0,Z^0,K^0)} (i.e., the comparison property and the stability property) still hold.
\end{remark}

Now, for a stopping time $\tau\in\mathcal{T}_0$ and $\xi \in L^{\infty}(\mathcal{F}_{\tau}),$ consider a BSDE with random terminal time  
\begin{equation}\label{59'}
Y_{t} = \xi + \int_{t\land\tau}^{\tau}f(r,S_{r},Y_{r},Z_{r})dr + \int_{t\land\tau}^{\tau} H(S_{r},Y_{r}) d\mathbf{X}_{r} - \int_{t\land\tau}^{\tau} Z_{r} dW_{r},\  t\in[0,T].
\end{equation}
Similar to Definition~\ref{def:stopping terminal} and Proposition~\ref{prop:extension of the well-posedness}, we provide the definition for the solution to a BSDE with random terminal time and its existence and uniqueness result.
\begin{definition}\label{def:stopping terminal'}
Let the parameter $(\hat{f},\hat{H})$ be defined in \eqref{e:hat f, hat L, hat H}, and $ \{\mathrm{X}^{n}\}_{n \ge 1}$ be a sequence of smooth paths with lifts $\mathbf{X}^{n} \rightarrow \mathbf{X}$ in $C^{0,p\text{-var}}([0,T],G^{[p]}(\R^{l})).$ Assume that for each $n\ge 1,$ Eq.~\eqref{59-} with parameter $(\xi,\hat{f},\hat{H},\mathrm{X}^{n})$ admits a unique solution $(Y^{n},Z^{n})$ in $\mathrm{H}^{\infty}_{[0,T]}\times \mathrm{H}^{2}_{[0,T]}(\R^{d}).$ Call $(Y,Z)\in \mathrm{H}^{\infty}_{[0,T]}\times \mathrm{H}^{2}_{[0,T]}(\R^{d})$ a solution of Eq.~\eqref{58'} with terminal time $\tau$ and parameter $(\xi,f,H,\mathbf{X})$ if
\begin{equation*}
Y^{n}_{\cdot} \rightarrow Y_{\cdot}\  \text{ uniformly on }\ [0,T]\  \text{ a.s., }\ 
Z^n \rightarrow Z \ \text{ in }\ \mathrm{H}^{2}_{[0,T]}(\mathbb{R}^d).
\end{equation*}
If  $(Y,Z)$ is independent of the choice of $\left\{\mathrm{X}^{n},n\ge 1\right\},$ we call $(Y,Z)$ the unique solution of Eq.~\eqref{59'}.  
\end{definition}

\begin{remark}\label{rem:relation of BSDE and RBSDE}
In our case (i.e., under the assumption as in Proposition \ref{prop:extension of the well-posedness}), Definition \ref{def:stopping terminal} is a generalization of Definition~\ref{def:stopping terminal'} with $K=0$. By Remark~\ref{rem:superlienar} and the estimate  \eqref{e:bounde of Yn}, a uniform prior estimate can be established for $Y^n,$ and thus $Y$ is uniformly bounded. Then a solution to \eqref{59'} is a solution to \eqref{58'} with $L\equiv -M$ for some sufficiently large $M>0$ and $K=0.$


\end{remark}

According to Remark~\ref{rem:relation of BSDE and RBSDE}, we obtain the well-posedness of \eqref{59'} by Proposition~\ref{prop:extension of the well-posedness}.

\begin{proposition}\label{prop:extension of the well-posedness'}
Assume $H$ and $f$ satisfy the conditions in \ref{(Hpr)} and \ref{(H0)} respectively. Assume $\xi\in L^{\infty}(\mathcal{F}_{\tau}),$ and $(\sigma,b)$ satisfies \ref{(H_b,sigma)}.  Then there exists a unique solution of Eq.~\eqref{59'}.

\end{proposition}



Now we are ready to show the main result of this subsection.
For each $\tau\in \mathcal{T}_{t},$ let the pair of processes $(B_{s}(\tau,L'_{\tau}),\pi_{s}(\tau,L'_{\tau}))$ be the unique solution of the BSDE~\eqref{59'} with terminal time $\tau$ and parameter $(L'_{\tau},f,H,\mathbf{X})$, where $L'$ is defined in \eqref{L'}.

\begin{theorem}\label{thm:optimal stopping}
Assume $(\xi,f,L,H,\mathbf{X})$ satisfies the conditions in Theorem~\ref{thm:well-posedness}. Let $(Y,Z,K)$ be the unique solution of the rough RBSDE \eqref{e:BSDE main result}
with parameter $(\xi,f,L,H,\mathbf{X}).$ Then, for each $t\in[0,T],$
\begin{equation*}
Y_t=B_t(D_t,L'_{D_{t}})=\esssup_{\tau\in\mathcal{T}_t} B_t(\tau,L'_{\tau}),
\end{equation*}
where $D_t=\inf\{u\geq t; Y_u=L'_u\}.$
\end{theorem}

\begin{proof}
Denote $\mathbf{X}^{0} := \mathbf{X}.$ Let $\left\{\mathrm{X}^{n}:[0,T]\rightarrow \R^{l},n\ge 1\right\}$ be a sequence of smooth paths with their lifts $\mathbf{X}^{n}\rightarrow \mathbf{X}^{0}$ in $C^{0,p\text{-var}}([0,T],G^{[p]}(\R^{l})).$ For each $n\ge 1,$  according to \cite{lepeltier1998existence} and \cite{kobylanski2000backward},  the BSDE with stopping terminal time $\tau$ and parameter $(L'_{\tau},f,H,\mathrm{X}^{n})$ admits a unique solution, and we denote it by $(B^{n}_{t}(\tau,L'_{\tau}),\pi^{n}_{t}(\tau,L'_{\tau})).$
Using the same argument as in the proof of Proposition~\ref{prop:extension of the well-posedness}, an estimate similar to \eqref{e:bounde of Yn} can be established, i.e., there exists a constant $M>0$ such that, for all $\tau\in\mathcal{T}_{t}$
\begin{equation}\label{e:|B^n| <= M}
\sup_{n\ge 1} \esssup_{(t,\omega)\in[0,T]\times\Omega}|B^{n}_{t}(\tau,L'_{\tau})| \le M.
\end{equation}
Let $\overline{L}_{t} := (-M-1)\land L_{t}.$ By  \eqref{e:|B^n| <= M} we have $B^{n}_{t}(\tau,L'_{\tau}) > \overline{L}_{t}.$ According to the uniqueness result shown by \cite[Corollary~1]{KLQT}, the triplet $(B^{n}_{t}(\tau,L'_{\tau}),\pi^{n}_{t}(\tau,L'_{\tau}),0)$ uniquely solves the RBSDE~\eqref{58'} with terminal time $\tau$ and parameter $(L'_{\tau},f,\overline{L},H,\mathrm{X}^{n}).$ Consequently, by Definition~\ref{def:stopping terminal}, \ref{def:stopping terminal'}, and Proposition~\ref{prop:extension of the well-posedness'}, $(B_{t}(\tau,L'_{\tau}),\pi_{t}(\tau,L'_{\tau}),0)$ uniquely solves the RBSDE~\eqref{58'} associated with terminal time $\tau$ and parameter $(L'_{\tau},f,\overline{L},H,\mathbf{X}).$ Since $Y_{\tau}\ge L'_{\tau}$ and $L_{\cdot}\ge \overline{L}_{\cdot}$ a.s., in view of Remark~\ref{rem:property of solution by stopping}, we have $Y_{t}\ge B_{t}(\tau,L'_{\tau}).$ Thus, we have
\begin{equation*}
Y_{t} \ge \esssup_{\tau\in \mathcal{T}_{t}}B_{t}(\tau,L'_{\tau}).
\end{equation*}

It remains to show $Y_{t} = B_{t}(D_{t},L'_{D_{t}}).$ Firstly, Let $(Y^{0},Z^{0},K^{0})$ be the unique solution of Eq.~\eqref{58'} with terminal time $\tau = D_{t}$ and parameter $(Y_{D_{t}},f,L,H,\mathbf{X}).$ By Definition~\ref{def:stopping terminal}, $(Y^{0},Z^{0},K^{0})$ is the limit of $(Y^{n},Z^{n},K^{n})$, which is the unique  solution of Eq.~\eqref{58'} with terminal time $\tau = D_{t}$ and parameter $(Y_{D_{t}},f,L,H,\mathrm{X}^{n}).$ Meanwhile,  by Theorem~\ref{thm:well-posedness}, $(Y,Z,K)$ is the limit of $(Y^{(n)},Z^{(n)},K^{(n)})$, which is the unique solution of Eq.~\eqref{e:BSDE main result} with parameter $(\xi,f,L,H,\mathrm{X}^{n})$. Then $(Y_{\cdot\land D_{t}},Z_{\cdot}\mathbf{1}_{\{\cdot\le D_{t}\}},K_{\cdot\land D_{t}})$ is the limit of $(Y^{(n)}_{\cdot\land D_{t}},Z^{(n)}_{\cdot}\mathbf{1}_{\{\cdot\le D_{t}\}},K^{(n)}_{\cdot\land D_{t}}),$ which is the unique  solution of Eq.~\eqref{58'} with terminal time $\tau = D_{t}$ and parameter $(Y^{(n)}_{D_{t}},f,L,H,\mathrm{X}^{n}).$ Since $Y^{(n)}_{D_{t}}\rightarrow Y_{D_{t}}$ a.s., by the stability of solutions implied by Remark~\ref{rem:property of solution by stopping},  we have
\begin{equation}\label{e:T <=> D_t}
(Y^{0}_{s},Z^{0}_{s},K^{0}_{s}) = (Y_{s\land D_{t}},Z_{s}\mathbf{1}_{\{s\le D_{t}\}},K_{s\land D_{t}}),\ \text{ for all }\ s\in[0,T].
\end{equation} 
Secondly, we will show that $Y^{0}_{s} = B_{s}(D_{t},L'_{D_{t}})$ for $s\in[t,T],$ which implies our conclusion. For any $T'\in[t,T]$ and $s\in[t,T'],$ let $\phi^{0,T'}$ be the unique solution of Eq.~\eqref{e:phi^n} with $n=0,$
and  $(\widetilde{f}^{0,T'},\widetilde{L}^{0,T'})$ be given in Eq.~\eqref{e:tilde f^n,T'} with $n=0.$ 
Let $(\widetilde{Y}^{0,T'},\widetilde{Z}^{0,T'},\widetilde{K}^{0,T'})$ be defined by \eqref{e:tilde Y^n,T'_tau} with $n=0$ and $\tau = D_t.$ By the proof of Proposition~\ref{prop:extension of the well-posedness}, we see the boundedness of $Y^{0}$, and thus there exists a constant $\widetilde{M}>0$ such that for all $T'\in[t,T]$ and $s\in[t,T']$
\begin{equation*}
\left|\widetilde{Y}^{0,T'}_{s}\right|\le \widetilde{M}.
\end{equation*}
Let the random variable  $\widetilde{\xi}^{0,T'}$ be defined in \eqref{e:xi^n,T'} with $n=0$ and $\tau = D_t.$ By the proof of Proposition~\ref{prop:extension of the well-posedness}, there exist an integer $N>0$ and a set of increasing real numbers $\left\{h^{(i)},i=0,1,...,N\right\}$ such that  $[T-h^{(N)},T-h^{(0)}] = [t,T].$ Furthermore, the triplet 
\begin{equation*}
(\widetilde{Y}^{0,T-h^{(i)}}_{\cdot},\widetilde{Z}^{0,T-h^{(i)}}_{\cdot},\widetilde{K}^{0,T-h^{(i)}}_{\cdot} - \widetilde{K}^{0,T-h^{(i)}}_{T-h^{(i)}})
\end{equation*}
is the unique solution of  Eq.~\eqref{58'} on $[T-h^{(i+1)},T-h^{(i)}]$ with terminal time $(T - h^{(i)}) \land D_{t}$ and parameter $(\widetilde{\xi}^{0,T-h^{(i)}},\widetilde{f}^{0,T - h^{(i)}},\widetilde{L}^{0,T - h^{(i)}},0,0).$ 
On each subinterval $[T-h^{(i+1)},T-h^{(i)}],$ since $Y_{s} > L_{s}$ for all $s\in[t,D_{t}),$ we have $\widetilde{Y}^{0,T-h^{(i)}}_{s} > \widetilde{L}^{0,T-h^{(i)}}_{s},$ $s\in[T-h^{(i+1)},T-h^{(i)}]\cap [t,D_{t}).$ Thus, $\widetilde{K}^{0,T-h^{(i)}}_{\cdot}\equiv 0.$ 
Note that by \cite[Theorem~A.2]{DF}, we see that Eq.~\eqref{59'} with terminal time $D_{t}$ and parameter $(\widetilde{\xi}^{0,T-h^{(i)}},\widetilde{f}^{0,T-h^{(i)}},0,0)$ has a unique solution. Then it follows that $(\widetilde{Y}^{0,T-h^{(i)}},\widetilde{Z}^{0,T-h^{(i)}})$ is the solution. In view of \eqref{e:Y0 Z0}, we have that $K^{0,T-h^{(i)}}=0$ and  $(Y^0,Z^0)$ is exactly the solution constructed in the proof of \cite[Theorem~3]{DF}.\footnote{The case for rough BSDEs with stopping terminal times is not involved in the proof of \cite[Theorem~3]{DF}. However, according to Proposition~\ref{prop:extension of the well-posedness}, their proof can be adapted to the stopping terminal case without difficulties.} It follows that $(Y^0,Z^0)$ is the unique solution of Eq.~\eqref{59'} with terminal time $D_{t}$ and parameter  $(L'_{D_{t}},f,H,\mathbf{X})$ which implies our claim and completes the proof.
\end{proof}

Now we consider the American option pricing problem under the rough volatility model. Recall that the market has two assets with prices $P^{0}$ and $P^{1},$ given in \eqref{e:bound} and \eqref{e:stock} respectively, where $r_{t}$ in \eqref{e:bound} is the short rate of a riskless bond, and $b_t$ (resp. ($\sigma_t$,$\lambda_t$)) in \eqref{e:stock} represents the interest rate (resp. the volatility) of a stock.

We assume that the market is imperfect, so the investor can borrow money from the bond under a higher interest rate $R_{t}>r_{t}$ to invest in the stock. Let $\{\pi_t\}_{t\in[0,T]}$ be the portfolio and $\{C_t\}_{t\in[0,T]}$ be the cumulative consumption plan. Then, The wealth process $V$ associated with $(\pi,C)$ satisfies (note that $\pi^{0}_{t} = [V_{t} - \pi^{1}_{t}]^{+}$)
\begin{equation*}
dV_{t} = \pi^{0}_{t}\frac{dP^{0}_{t}}{P^{0}_{t}} + \pi^{1}_{t}\frac{dP^{1}_{t}}{P^{1}_{t}} - R_{t}\left[ V_{t} - \pi^{1}_{t}\right]^{-}dt - dC_{t}.
\end{equation*}
Denote $\theta_{t} := \sigma_{t}^{-1}(b_{t} - r_{t})$ where we recall $\sigma=\rho v.$ Substituting Eq.~\eqref{e:bound} and \eqref{e:stock}
into the above equation, we have 
\begin{equation*}
dV_{t} = \left(r_{t}V_{t} + \pi^{1}_{t}\sigma_{t}\theta_{t} - (R_{t} - r_{t})\left[ V_{t} - \pi^{1}_{t}\right]^{-} \right)dt +  \pi^{1}_{t}\sigma_{t} dW_{t} + \pi^{1}_{t}\sigma_{t}d\mathbf{X}_{t} - dC_{t},
\end{equation*}
where $\mathbf{X}_t:=(1,I_t, \int_0^t I_r dI_r)$ with $I_t= \int_0^t \sigma^{-1}_{t}\lambda_{t} dB_t$. By \cite[Proposition~2.11]{Christian-roughVolat}, $\mathbf{X}_t + (0,0,\frac12 \langle I\rangle_{t})  \in C^{0,p\text{-var}}([0,T],G^{[p]}(\R))$ for any $p\in(2,3).$

Assume for now that $I$ is a smooth path (so the rough integral above is equivalent to the classical integral driven by $I$), and let $\{L_t\}_{t\in[0,T]}$ be the payoff process. Then, following the same analysis as in \cite[Section~5]{el1997reflected}, the price of the American option, which is denoted by $Y_{t},$ should satisfy the following RBSDE:
\begin{equation}\label{e:pricing of Amer opt}
\left\{\begin{aligned}
Y_{t} &= L_{T} + \int_{t}^{T} \bar{f}(r,Y_{r},Z_{r})dr - \int_{t}^{T}Z_{r}d\mathbf{X}_{r} - \int_{t}^{T}Z_{r}dW_{r} + K_{T} - K_{t},\\
Y_{t} &\ge L_{t},\quad \int_{0}^{T} (Y_{r} - L_{r}) dK_{r} = 0,\quad t\in [0,T],
\end{aligned}\right.
\end{equation}
where $\bar{f}(t,y,z) := - r_{t}y - z\theta_{t} + (R_{t} - r_{t})\left[y -  z\sigma^{-1}_{t}\right]^{-}.$ 

When the rough path $\mathbf{X}$ is not smooth, we introduce the optimal stopping problem for rough RBSDE \eqref{e:pricing of Amer opt}\footnote{The rigorous derivation of the relationship between American option pricing and the optional stopping within the rough volatility model lies beyond the scope of this paper; we address this topic in our future work.}. 
This will be done with the help of the stability of PDEs, similar to the proof of Theorem~\ref{thm:rough PDE}.
	
\begin{proposition}
Assume that $b_{t},$ $R_t,$ $r_{t},$ $\sigma_{t},$ $\sigma^{-1}_{t},$ and $\lambda_t$ are deterministic and bounded functions, and assume that $L_{t} = l(t,W_{t})$ for some function $l(\cdot,\cdot)\in C_{b}([0,T]\times\R^{d},\mathbb{R})$. Let $\left\{\mathrm{X}^{n}:[0,T]\rightarrow \R,n\ge 1\right\}$ be a sequence of  
smooth paths such that their lifts $\mathbf{X}^{n}\rightarrow \mathbf{X} + (0,0,\frac{1}{2}\langle I\rangle_{t})$ in $C^{0,p\text{-var}}([0,T],G^{[p]}(\R)),$ with some $p\in (2,3).$ For $(t,x)\in[0,T]\times \R^{d},$ set $W^{t,x}_{s} := W_{s} - W_{t} + x$ for $s\in[t,T].$ For $n\ge 1,$ denote by $(Y^{n;t,x},Z^{n;t,x},K^{n;t,x})$ the unique solution to the following RBSDE: 
\begin{equation}\label{e:Amer Opt}
\left\{\begin{aligned}
Y_{s} &= l(T,W^{t,x}_{T}) + \int_{s}^{T} \bar{f}(r,Y_{r},Z_{r})dr - \int_{s}^{T} Z_{r}d\mathrm{X}^{n}_{r} - \int_{s}^{T}Z_{r}dW_{r} + K_{T} - K_{s},\\
Y_{s} &\ge l(s,W^{t,x}_{s}),\quad \int_{t}^{T} (Y_{r} - l(r,W^{t,x}_{r})) dK_{r} = 0,\quad s\in[t,T].
\end{aligned}\right.
\end{equation}
Then, there exists a continuous adapted process $Y_{t}$ such that a.s.,
\begin{equation*}
\lim_{n\rightarrow \infty}|Y^{n;0,0}_{t} - Y_{t}|=0 \text{ uniformly in }t\in[0,T].
\end{equation*}
Consequently, we have 
\begin{equation}\label{e:Amer option and optimal stopp}
Y_{t} = \lim_{n\rightarrow \infty}\esssup_{\tau\in \mathcal{T}_{t}}B^{n}_{t}(\tau,L_{\tau}),
\end{equation}
where $(B^{n}_{t}(\tau,L_{\tau}),\Gamma^{n}_{t}(\tau,L_{\tau}))$ denotes the unique solution to the following BSDE:
\begin{equation*}
B^{n}_{t} = l(\tau,W_{\tau}) + \int_{t\land \tau}^{\tau}\bar{f}(r,B^{n}_{r},\Gamma^{n}_{r})dr - \int_{t\land \tau}^{\tau}\Gamma^{n}_{r}d\mathrm{X}^{n}_{r} - \int_{t\land \tau}^{\tau}\Gamma^{n}_{r}dW_{r},\quad t\in[0,T].
\end{equation*}
\end{proposition}

\begin{proof}
By assumption, 
$\mathrm{X}^{n}_{s}\rightarrow I_{s}$ uniformly in $s\in[0,T]$, as $n\rightarrow \infty.$ Denote 
\begin{equation*}
\tilde{l}^{n}(t,x) := l\left(t,x - \mathrm{X}^{n}_{t}\right),\ \text{ and }\  \tilde{l}^{\infty}(t,x) := l\left(t,x - I_{t}\right).
\end{equation*}
For $n\ge 1$ and $n=\infty,$ let $v^{n}(t,x)$ be the unique viscosity solution to the following PDE (the existence of such $v^n$ can be seen, e.g., \cite{KLQT})
\begin{equation*}
\left\{\begin{aligned}
&d v(t,x) + \left(\frac{1}{2}\Delta v(t,x) + \bar{f}\big(t,v(t,x),\partial_{x}v(t,x)\big)\right) dt = 0,\quad t\in(0,T),\\
&v(T,x) = \tilde{l}^{n}(T,x) ,\quad v(t,x)\ge \tilde{l}^{n}(t,x).
\end{aligned}\right.
\end{equation*}
By a similar proof of \cite[Lemma~6]{friz2014rough}, for $n\ge 1$, we see that 
\[u^{n}(t,x) := v^{n}\left(t,x + \mathrm{X}^{n}_{t}\right)\] 
is the unique viscosity solution to the following PDE: 
\begin{equation*}
\left\{\begin{aligned}
&d u(t,x) + \left(\frac{1}{2}\Delta u(t,x) + \bar{f}\big(t,u(t,x),\partial_{x}u(t,x)\big)\right)dt - \partial_{x}u(t,x)d\mathrm{X}^{n}_{t}  = 0,\quad t\in(0,T)\\
&u(T,x) = l(T,x) ,\quad v(t,x)\ge l(t,x).
\end{aligned}\right.
\end{equation*}
Then, for $n\ge 1$, by Theorem~\ref{thm:rough PDE}, we have $u^{n}(t,x) = Y^{n;t,x}_{t}.$ Hence, by the Markovian property of RBSDE~\eqref{e:Amer Opt}, it follows that
\begin{equation*}
Y^{n;0,0}_{s} = u^{n}(s,W_{s}) = v^{n}\left(s,W_{s} + \mathrm{X}^{n}_{s}\right),
\end{equation*}
which can be proved by the penalization approximation result shown in \cite[Theorem~5.6]{el1997non}, and by the Markovian property of BSDEs (see, e.g., \cite[Remark~2.1]{pardoux1998backward}). Noting that $\tilde{l}^{n}\rightarrow \tilde{l}^{\infty}$ locally uniformly in $(t,x)$, by the same proof as in the step 1 and step 2 of Theorem~\ref{thm:rough PDE},  $v^{n}\rightarrow v^{\infty}$ locally uniformly in $(t,x).$ Hence, we have that a.s.,
\begin{equation*}
v^{n}\left(s,W_{s} + \mathrm{X}^{n}_{s}\right)\rightarrow v^{\infty}\left(s,W_{s} + I_{s}\right)\text{ uniformly in }s\in[0,T],\text{ as }n\rightarrow \infty.
\end{equation*}
Thus, letting $Y_{s}:= v^{\infty}\left(s,W_{s} + I_{s}\right),$ we derive that a.s., $Y^{n;0,0}_{s}\rightarrow Y_{s}$ uniformly in $s\in[0,T].$ 
Furthermore, by Theorem~\ref{thm:optimal stopping}, it follows that $Y^{n;0,0}_{t} = \esssup\limits_{\tau\in\mathcal{T}_{t}} B^{n}_{t}(\tau,L_{\tau}),$ which yields \eqref{e:Amer option and optimal stopp}.
\end{proof}

\section*{Acknowledgments}

Li and HZ are supported by the Fundamental Research Funds for the Central Universities and NSF of Shandong (No. ZR2023MA026).
Li's research was supported by the National Natural Science Foundation of China (No. 12301178), the Natural Science Foundation of Shandong Province for Excellent Young Scientists Fund Program (Overseas)  (No. 2023HWYQ-049) and the Natural Science Foundation of Shandong Province (No. ZR2023ZD35). HZ is partially supported by NSF of China (Grant Numbers 12031009) and Young Research Project of Tai-Shan (No.tsqn202306054).

\renewcommand\thesection{Appendix A}
\section{Results on RBSDEs and rough BSDEs}\label{append:A}
\renewcommand\thesection{A}

In this section, we recall some basic definitions and results concerning reflected BSDEs with superlinear quadratic coefficient in \cite{KLQT} and results concerning BSDEs with rough drivers in \cite{DF}. 
Assume the obstacle process $L$ and the terminal value $\xi$ satisfy Assumption~\ref{(H_L,xi)}. Consider the following RBSDE  
\begin{equation}\label{e:RBSDE}
\left\{\begin{aligned}
&Y_t=\xi+\int_t^T f(s,S_s,Y_s,Z_s)ds-\int_t^T Z_s dW_s +K_T-K_t,\\
&Y_{t}\ge L_{t}, \ \ t\in[0,T],\ \ \int_{0}^{T} 
(Y_{r} - L_{r}) dK_{r} = 0,
\end{aligned}\right.
\end{equation}
where $S$ is given in \eqref{e:diffusion process}. We assume

\begin{assumptionp}{(H1)}\label{(H1)}
Let $f:\Omega\times[0,T]\times\R^{d}\times\R\times\R^{d}\rightarrow\R$ be $\mathcal{P}\otimes\mathcal{B}(\R^d)\otimes\mathcal{B}(\R)\otimes\mathcal{B}(\R^d)$-measurable. Assume there is a strictly positive increasing continuous function   $c$ satisfying $\int_0^\infty \frac{dy}{c(y)}=\infty$ so that
\item[(A1)] there exists a constant $C_f>0$ such that $\mathbb{P}$-a.s.,
\begin{align*}
|f(t,x,y,z)|\leq c(|y|)+C_{f}|z|^2, \text{ for all }(t,x,y,z)\in[0,T]\times\R^{d} \times \R\times \R^{d}; 
\end{align*}
\item[(A2)] for every $M>0,$ there exists a constant $C_{M}>0$ such that $\mathbb{P}$-a.s.,
\begin{equation*}
|D_z f(t,x,y,z)|\leq C_{M}(1+|z|), \text{ for all }(t,x,y,z)\in[0,T]\times\R^{d} \times [-M,M]\times \R^{d};
\end{equation*}
\item[(A3')] for every $\varepsilon,M>0,$ there exist two constants $C_{\varepsilon,M},h_{\varepsilon,M}>0$ such that $\mathbb{P}$-a.s.,
\[D_y f(t,x,y,z)\leq  C_{\varepsilon,M} + \varepsilon |z|^{2}, \text{ for all }(t,x,y,z)\in[T-h_{\varepsilon,M},T]\times\R^d \times[-M,M]\times\mathbb{R}^d.\]

\end{assumptionp}

\begin{definition}\label{def:RBSDE}
We say that $(Y,Z,K)$ is a solution of the RBSDE \eqref{e:RBSDE} associated with terminal value $\xi,$ obstacle $L$ and coefficient $f,$ which is denoted by Eq.~$(\xi,L,f),$ if
\begin{itemize}
\item[(a)] it satisfies Eq.~\eqref{e:RBSDE} for every $t\in[0,T]$ a.s..
\item[(b)] $Y\in \mathrm{H}^{\infty}_{[0,T]},$ $Z\in \mathrm{H}^{2}_{[0,T]}(\R^{d}),$ and $K\in \mathrm{I}^{1}_{[0,T]}.$
\end{itemize}
\end{definition}

Kobylanski et al. \cite[Theorem 1]{KLQT} gives the following a prior estimate of $Y.$
\begin{lemma}\label{lem:A Prior Estimate}
Assume that $(b,\sigma)$ satisfies Assumption~\ref{(H_b,sigma)}, $f$ satisfies (A1) in \ref{(H1)}, $(L,\xi)$ satisfy \ref{(H_L,xi)}, and $(Y,Z,K)\in \mathrm{H}_{[0,T]}^{\infty}(\mathbb{R})\times \mathrm{H}_{[0,T]}^{2}(\mathbb{R}^{d})\times \mathrm{I}_{[0,T]}^{1}(\mathbb{R})$ is a solution of the Eq.~\eqref{e:RBSDE} with parameter $(\xi, f, L),$ we have
\begin{equation}\label{e:prior estimate}
|Y_t|\le U_{0} =: \bar{M}, 
\end{equation}
with $U:[0,T]\rightarrow \R$ uniquely solving the following backward ODE 
\[U_t = b + \int_{t}^{T}c(U_{r})dr, \]
where $b:= C_{\xi}\vee \esssup_{(t,\omega)\in[0,T]\times\Omega} L_t(\omega),$ and $C_{\xi}$ is the constant given in \ref{(H_L,xi)}.

\end{lemma}

Recall $\mathcal{T}_t$ and $L'$ defined at the beginning of Subsection~\ref{subsec:Optimal Stopping}. For each $\tau\in\mathcal{T}_t,$ denote by $(\mathcal{B}_s(\tau,L'_\tau),\pi_s(\tau,L'_\tau), t\leq s\leq \tau)$ the unique solution (if it exists) of the following BSDE 
\begin{equation}\label{e:B_s(tau,L') without rough path}
\left\{\begin{aligned}
dY_{s} &= - f(s,S_{s},Y_{s},Z_{s})ds + Z_{s}dW_{s},\quad s\in[t,\tau],\\
Y_{\tau}&=L'_{\tau}.
\end{aligned}\right.
\end{equation}

In the context of our paper, note that the coefficient $f$ only meets the condition~(A3') in Assumption~\ref{(H1)} rather than the condition \cite[(H4), (iii)]{KLQT}. The latter condition requires that for every $\varepsilon>0$, there exists $C_{\varepsilon}>0$ such that $D_{y}f(t,x,y,z)\le C_{\varepsilon} + \varepsilon|z|^2$ for all $t\in[0,T],$ while (A3') in \ref{(H1)} only require such a condition on a small interval. Thus, for the convenience of the reader, we give the following characterization, comparison theorem, and well-posedness of RBSDEs. The following Proposition and its proof are based on \cite[Proposition~3.1]{KLQT}.\

\begin{proposition}\label{prop:charac of RBSDE}
Assume \ref{(H1)}, \ref{(H_L,xi)} , and \ref{(H_b,sigma)} hold. Suppose that $(Y,Z,K)$ is a solution of the reflected BSDE \eqref{e:RBSDE} with parameter $(\xi,f,L).$ Then there exists $\varepsilon>0$ depending only on $\bar{M},C_f,C_{\bar{M}}$
such that for each $t\in[T - h_{\varepsilon,\bar{M}},T],$ 
\begin{equation*}
Y_t=\mathcal{B}_t(D_t,L'_{D_{t}})=\esssup_{\tau\in\mathcal{T}_t} \mathcal{B}_t(\tau,L'_{\tau}),
\end{equation*}
where $\bar{M}$ is the constant defined in Lemma~\ref{lem:A Prior Estimate}, $h_{\varepsilon,\bar{M}}$ is given in Assumption~\ref{(H1)}, and $D_t=\inf\{u\geq t; Y_u=L'_u\}.$
\end{proposition}

\begin{proof}
Since $Y_{D_{t}} = L'_{D_{t}}$ by definition of $D_{t},$ we derive that $(Y_{s\land D_{t}},t\le s\le T)$ is a solution of the BSDE associated with terminal time $D_{t}$ and parameter $(L'_{D_{t}},f).$ In addition, by the comparison theorem for BSDEs with quadratic growth coefficients proved by \cite[Theorem~A.2]{DF}, there exists $\varepsilon>0$  depending only on $\bar{M},C_f,C_{\bar{M}}$
such that $(Y_{s\land D_{t}},t\le s\le T)$ is the unique solution whenever $t\in [T-h_{\varepsilon,\bar{M}},T].$ Thus, for each $t\in [T-h_{\varepsilon,\bar{M}},T]$ We have
\begin{equation*}
Y_{t} = \mathcal{B}_{t}(D_{t},L'_{D_{t}}).
\end{equation*}
Fix $t\in [T-h_{\varepsilon,\bar{M}},T].$ Now it remains to show that $Y_{t} \ge \mathcal{B}_{t}(\tau,L'_{\tau})$ for each $\tau \in \mathcal{T}_{t}.$ Fix $\tau\in\mathcal{T}_{t}.$ Note that the pair of processes $(Y_{s\land \tau},Z_{s}\mathbf{1}_{\{s\in[t,\tau]\}})$ satisfies \begin{equation*}
- dY_{s} = f(s,S_{s},Y_{s},Z_{s})ds + dK_{s} - Z_{s} dW_{s},\quad s\in[t,\tau].
\end{equation*} 
In other words, $(Y_{s\land \tau},Z_{s}\mathbf{1}_{\{s\in[t,\tau]\}})$ is a solution of the BSDE associated with terminal time $\tau,$ terminal value $Y_{\tau},$ and driver $f(r,x,y,z)dr + dK_{r}.$ Note that $Y_{\tau} \ge L'_{\tau}.$ Following \cite[Theorem~A.2]{DF} again, we have
\begin{equation*}
Y_{t}\ge \mathcal{B}_{t}(\tau,L'_{\tau}),
\end{equation*}
and the proof is complete.
\end{proof}

The following comparison theorem and its proof are based on \cite[Proposition~3.2]{KLQT}.

\begin{proposition}\label{prop:RBSDE Comparison}
Suppose Assumption~\ref{(H_b,sigma)} holds. Let $f^1$ and $f^2$ satisfy Assumption~\ref{(H1)}, and $(\xi^{1},L^{1})$ and $(\xi^{2},L^{2})$ satisfy Assumption~\ref{(H_L,xi)}. Let $(Y^{1},Z^{1},K^{1})$ (resp. $(Y^{2},Z^{2},K^{2})$) be a solution of Eq.~\eqref{e:RBSDE} with parameter $(\xi^{1},L^{1},f^{1})$ (resp. $(\xi^{2},L^{2},f^{2})$) and assume that 
\begin{equation*}
\xi^{1}\le \xi^{2},\ L^{1}_{t}\le L^{2}_{t},\text{ and }f^{1}(t,x,y,z)\le f^{2}(t,x,y,z)\text{ for all }(x,y,z)\in\R^{d}\times\R\times\R^{d}\ dt\otimes d\mathbb{P}\text{-a.s..}
\end{equation*}
Then there exists $\varepsilon>0$  depending only on $\bar{M},C_f,C_{\bar{M}}$ such that
\begin{equation*}
Y^{1}_{t} \le Y^{2}_{t} \text{ for all }t\in[T-h_{\varepsilon,\bar{M}},T]\text{ a.s..}
\end{equation*}
\end{proposition}

\begin{proof}
Let $\varepsilon>0$ be the same as in Proposition~\ref{prop:charac of RBSDE}. Denote $L_t^{1\,\prime} := L^{1}_{t}\mathbf{1}_{\{t<T\}} + \xi^{1} \mathbf{1}_{\{t=T\}}$ and $L_t^{2\,\prime} := L^{2}_{t}\mathbf{1}_{\{t<T\}} + \xi^{2} \mathbf{1}_{\{t=T\}}.$ For each $t\in[T-h_{\varepsilon,\bar{M}},T]$  and $\tau\in\mathcal{T}_{t},$ denote by $\mathcal{B}^{1}_{t}(\tau,L^{1}_{\tau\,\prime})$ (resp. $\mathcal{B}^{2}_{t}(\tau,L_{\tau}^{2\,\prime})$) the unique solution to Eq.~\eqref{e:B_s(tau,L') without rough path} associated with terminal time $\tau$ and parameter $(L_{\tau}^{1\,\prime},f^{1})$ (resp. $(L_{\tau}^{2\,\prime},f^{2})$). Fix $t\in [T-h_{\varepsilon,\bar{M}},T].$ Let $D^{1}_{t}:=\inf\left\{u\ge t; Y^{1}_{u} = L_{u}^{1\,\prime}\right\}\in\mathcal{T}_{t}.$ By Proposition~\ref{prop:charac of RBSDE} we have 
\begin{equation*}
Y^{1}_{t} = \mathcal{B}^{1}_{t}(D^{1}_{t},L^{1\,\prime}_{D^{1}_{t}}),\text{ and }Y^{2}_{t} = \esssup_{\tau\in\mathcal{T}_{t}}\mathcal{B}^{2}_{t}(\tau,L_{\tau}^{2\,\prime}).
\end{equation*}
Note that in the proof of Proposition~\ref{prop:charac of RBSDE}, the choice of $\varepsilon$ also ensures that the comparison of BSDEs hold on $[T-h_{\varepsilon,\bar{M}},T].$ Also note that $f^{1}\le f^{2}$ and $L_{D^{1}_{t}}^{1\,\prime}\le L_{D^{1}_{t}}^{2\,\prime}.$ Thus, we have
\begin{equation*}
Y^{1}_{t} = \mathcal{B}^{1}_{t}(D^{1}_{t},L_{D^{1}_{t}}^{1\,\prime}) \le  \mathcal{B}^{2}_{t}(D^{1}_{t},L_{D^{1}_{t}}^{2\,\prime}) \le Y^{2}_{t}.
\end{equation*}
\end{proof}

\begin{proposition}\label{prop:RBSDE willposedness}
Suppose \ref{(H1)}, \ref{(H_L,xi)}, and \ref{(H_b,sigma)} hold. Then there exists an $\varepsilon>0$ depending only on $\bar{M},C_f,C_{\bar{M}}$ such that Eq.~\eqref{e:RBSDE} with parameter $(\xi,f,L)$ admits a unique solution on $[T-h_{\varepsilon,\bar{M}},T],$ where $\bar{M}$ is the constant defined in Lemma~\ref{lem:A Prior Estimate} and $h_{\varepsilon,\bar{M}}$ is given by Assumption~\ref{(H1)}.
\end{proposition}	
\begin{proof}
The existence of the solution is shown by \cite[Theorem 3]{KLQT}. We claim that uniqueness is a consequence of Proposition~\ref{prop:RBSDE Comparison}. Indeed, letting $(Y^{1},Z^{1},K^{1})$ and $(Y^{2},Z^{2},K^{2})$ be two solutions such that $Y^{1} = Y^{2} =: Y$ on $[T-h_{\varepsilon,\bar{M}},T],$ we have 
\begin{equation*}
\begin{aligned}
\int_{T - h_{\varepsilon,\bar{M}}}^{t}(Z^{1}_{r} - Z^{2}_{r})dW_{r} &= \int_{T - h_{\varepsilon,\bar{M}}}^{t}(f(r,S_{r},Y_{r},Z^{1}_{r}) - f(r,S_{r},Y_{r},Z^{2}_{r}))dr \\
&\quad - (K^{1} - K^{2})_{t} + (K^{1} - K^{2})_{T - h_{\varepsilon,\bar{M}}},\quad t\in[T - h_{\varepsilon,\bar{M}},T],
\end{aligned}
\end{equation*}
in which a martingale is equal to a continuous process of finite variation. This implies that $Z^{1}_{t} = Z^{2}_{t}$ in $\mathrm{H}^{2}_{[T-h_{\varepsilon,\bar{M}},T]}$ and then $K^{1}_{t} - K^{1}_{T-h_{\varepsilon,\bar{M}}} = K^{2}_{t} - K^{2}_{T-h_{\varepsilon,\bar{M}}}$ for all $t\in [T-h_{\varepsilon,\bar{M}},T]$ a.s..
\end{proof}

\cite[Theorem~4]{KLQT} gives the stability of reflected BSDEs.

\begin{lemma}\label{lem:continu solution map}
Suppose Assumption~\ref{(H_b,sigma)} holds. Let $\{\xi^n\}_{n\in\mathbb{N}}$ be a family of terminal value, $\{L^n\}_{n\in\mathbb{N}}$ be a family of obstacles satisfying Assumption~\ref{(H_L,xi)}, and $\{f^n\}_{n\in\mathbb{N}}$ be a family of coefficients satisfying Assumption \ref{(H1)} such that 
\begin{itemize}
\item[(i)] there exists a constant $C>0$ such that, for each $n,$
$$|\xi^n|\leq C, \  |L^n_t|\leq C, \ t\in[0,T];$$
\item[(ii)] there exists a function $c(\cdot)$ of the form $c(y)=a(1+|y|)$ with $a>0$ such that for each $n\in\mathbb{N},$
\begin{displaymath}
|f^n(t,x,y,z)|\leq c(|y|)+C|z|^2,  \ (t,x,y,z)\in[0,T]\times\mathbb{R}^d\times \mathbb{R}\times\mathbb{R}^{d};
\end{displaymath}
\item[(iii)] the sequence $\{f^n\}$ converges to $f^{0}$ locally uniformly in $(t,x,y,z)\in[0,T]\times\mathbb{R}^d\times \mathbb{R}\times\mathbb{R}^{d}$ a.s., $\{\xi^n\}$ converges to $\xi^{0}$ a.s.,  and $L^{n}$ converges to $L^{0}$ uniformly in $t\in[0,T]$ a.s..
\end{itemize}
For each $n\in\mathbb{N},$ let $(Y^n,Z^n,K^n)$ be the unique solution of the reflected BSDE Eq.~\eqref{e:RBSDE} with $(\xi^n,f^n,L^n).$ Then 
\begin{equation*}
\begin{aligned}
(Y^{n},K^{n}) &\rightarrow (Y,K) \text{ uniformly }\mathbb{P} \text{-a.s.,}\\
Z^n &\rightarrow Z \text{ in }\mathrm{H}^{2}_{[0,T]}(\mathbb{R}^d).
\end{aligned}
\end{equation*}
\end{lemma}

\begin{remark}\label{rem:conditions of continuity of RBSDE}
Actually, Lemma \ref{lem:continu solution map} is slightly different from \cite[Theorem~4]{KLQT}, which requires that $|D_{y} f(t,x,y,z)|\le C_{\varepsilon,M} + \varepsilon |z|^{2}$ holds for all $\varepsilon>0$ and $t\in[0,T]$ (which is exactly condition (iii) in \cite[(H4)]{KLQT}). However, in Lemma \ref{lem:continu solution map}, we only need this inequality to be true for all $t\in[T-h_{\varepsilon,M},T]$, i.e., $(A3')$ in Assumption~\ref{(H1)}. Let us explain why the stability result still holds under this weaker condition. Going back to the proof of \cite[Theorem~4]{KLQT}, condition (iii) in \cite[(H4)]{KLQT} is used to ensure the comparison theorem of Eq.~\eqref{e:RBSDE} holds. However, Lemma~\ref{prop:RBSDE willposedness} shows that the requirement of $(A3')$ is sufficient to achieve the same result. Therefore, we can weaken the assumptions in \cite[Theorem~4]{KLQT} to establish Lemma~\ref{lem:continu solution map}. 
\end{remark}

The BSDEs with rough drivers is considered as follows 
\begin{equation}\label{e:rough BSDE}
Y_{t}= \xi + \int_{t}^{T}f(s,S_{s},Y_{s},Z_{s})ds + \int_{t}^{T}H(S_{s},Y_{s})d\mathbf{X}_{s} - \int_{t}^{T}Z_{r}dW_{r},\quad t\in[0,T].
\end{equation}
We first state the well-posedness of BSDEs with rough drivers proved by \cite[Theorem~3]{DF}.

\begin{proposition}\label{prop:BSDE with rough driver}
Assume $(H,f)$ satisfies Assumption~\ref{(Hpr)}, \ref{(H0)}, $(\sigma,b)$ satisfies Assumption~\ref{(H_b,sigma)}, and $\xi\in L^{\infty}(\mathcal{F}_{T}).$ Let $\mathbf{X}^{n},n\ge 1$ and $\mathbf{X}$ be given in Theorem~\ref{thm:well-posedness}. Let $(Y^{n},Z^{n},K^{n})$ be the corresponding solution of Eq.~\eqref{59-} with parameter $(\xi,f,H,\mathrm{X}^{n}).$ Then there exists a pair $(Y,Z)\in \mathrm{H}^{\infty}_{[0,T]}\times \mathrm{H}^{2}_{[0,T]}(\R^{d})$ satisfying
\begin{equation*}
Y^{n}_{\cdot}\rightarrow Y_{\cdot}\text{ uniformly on }[0,T]\text{ a.s., }\ Z^{n}\rightarrow Z\text{ in }\mathrm{H}^{2}_{[0,T]}(\mathbb{R}^{d}).
\end{equation*}
We call $(Y,Z)$ be the solution of Eq.~\eqref{59-} with parameter $(\xi,f,H,\mathbf{X}).$
Furthermore, $(Y,Z)$ is independent of choice of $\left\{\mathbf{X}^{n},n\ge 1\right\}.$ 
\end{proposition}

\begin{remark}\label{rem:superlienar}
In Assumption~\ref{(H0)}, we assume the following super-linear growth condition in $y\in\R$ 
\begin{equation*}
|f(t,x,y,z)|\le c(|y|) + C_{0}|z|^{2}\ \text{ with }\ \int_{0}^{\infty}\frac{dy}{c(y)} = \infty,
\end{equation*}
which is weaker than \cite[(F1)]{DF} required by \cite[Theorem~3]{DF}. The assumption \cite[(F1)]{DF} is used to make an a prior estimate in the form of \eqref{e:bounde of Yn}, which can be obtained by using \cite[Corollary~2.2]{kobylanski2000backward}. In our case, we obtain the same a prior estimate by \cite[Theorem~1]{lepeltier1998existence}, which only requires the super-linear growth in $y.$
\end{remark}

Now we stated the comparison theorem for BSDEs with rough drivers as follows, which is a direct consequence of \cite[Theorem~2.6]{kobylanski2000backward} and Proposition~\ref{prop:BSDE with rough driver}.
\begin{proposition}\label{prop:comparison-BSDE with rough driver}
Assume for each $j=1,2,$ the parameter $(\xi^{j},f^{j},H,\mathbf{X})$ satisfy the same assumptions in Proposition~\ref{prop:BSDE with rough driver}. For $i=1,2$, let $(Y^{i},Z^{i})$ be the solution of Eq.~\eqref{e:rough BSDE} with parameter $(\xi^{j},f^{j},H,\mathbf{X}).$ Suppose that $\xi^{1} \le \xi^{2}$ and
\begin{equation*}
f^{1}(t,x,y,z)\le f^{2}(t,x,y,z)\ \ d\mathbb{P}\times dt\text{ a.e.,}\ \forall (x,y,z)\in\R^{d}\times\R\times\R^{d}.
\end{equation*}
Then, we have
\begin{equation*}
Y^{1}_{t} \le Y^{2}_{t}\ \text{ for all }\ t\in[0,T]\ \text{ a.s.. } 
\end{equation*}
\end{proposition}

\renewcommand\thesection{Appendix B}
\section{Some facts about PDEs with obstacles}\label{append:B}
\renewcommand\thesection{B}
In this section, we recall some results on PDEs with obstacles, and then we present the comparison theorem, which is essential for Section~\ref{subsec:PDE obstacle}. Throughout this section, we always suppose that $f\in C([0,T]\times \R^{d} \times \R \times \R^{d},\R),$ $b\in C([0,T]\times\R^{d},\R^{d}),$ $\sigma\in C([0,T]\times\R^{d},\R^{d\times d}),$ $l\in C([0,T]\times\R^{d}, \R),$ and $g\in C(\R^{d},\R)$ are deterministic. 
Consider the following obstacle problem
\begin{equation}\label{e:PDE-Obstacle}
\left\{\begin{aligned}
&D_{t}u(t,x) + \frac{1}{2} \text{Tr}\left\{\sigma\sigma^{\top}D^{2}_{xx} u\right\}(t,x) + (D_{x} u^{\top} b)(t,x) + f\big(t,x,u(t,x),(\sigma^{\top}D_{x}u)(t,x) \big) = 0,\\
&u(T,x) = g(x),\quad u(t,x)\ge l(t,x),\quad t\in [0,T].
\end{aligned}\right.
\end{equation}
Denote the set of symmetric $d\times d$ matrices by $\mathcal{S}(d).$ We say that $(a,p,\Lambda)\in \R\times \R^d \times \mathcal{S}(d)$ belongs to the semi-jets $P^{2,+}u(t,x)$ ($ P^{2,-}u(t,x)$ resp.) if, as $y\rightarrow x$, $s\rightarrow t,$
\begin{equation*}
u(s,y) \le (\ge\text{ resp.}) u(t,x) + a(s-t) + \langle p, y-x\rangle + \frac{1}{2}\langle\Lambda (y-x), (y-x)\rangle + o(|s-t| + |y-x|^{2}).
\end{equation*}

\begin{definition}\label{def:viscos solution}

(a). Assume  $u: (0,T]\times\mathbb{R}^{d}\rightarrow\mathbb{R}$ 
is upper semicontinuous. We call $u$ a viscosity subsolution of \eqref{e:PDE-Obstacle} if for every $x\in\R^d,$ $u(T,x)\le g(x),$ and if for every $(t_0,x_0)\in (0,T)\times\mathbb{R}^d$ such that $u(t_0,x_0)> l(t_0,x_0),$ and for every $(a,p,\Lambda)\in P^{2,+}u(t_{0},x_{0}),$ the following inequality is fulfilled
\begin{equation*}
a + \frac{1}{2} \text{Tr}\left\{\sigma\sigma^{\top}(t_0,x_0)\Lambda\right\} + p^{\top}b(t_0,x_0) + f\big(t_0,x_0,u(t_0,x_0),\sigma^{\top}(t_0,x_0)p\big)\ge 0.
\end{equation*}

(b). Assume $v:(0,T]\times \R^{d} \rightarrow \R$ 
is lower semicontinuous. We call $v$ a viscosity supersolution of \eqref{e:PDE-Obstacle} if for every $x\in\R^d,$ $v(T,x)\ge g(x),$ and if for every $(t,x)\in(0,T]\times\mathbb{R}^d,$ $v(t,x)\ge l(t,x),$ and if for every $(t_0,x_0)\in(0,T)\times \R^d,$ and for every $(a,p,\Lambda)\in P^{2,-}v(t_{0},x_{0}),$ the following inequality is fulfilled
\begin{equation*}
a + \frac{1}{2} \text{Tr}\left\{\sigma\sigma^{\top}(t_0,x_0)\Lambda\right\} + p^{\top}b(t_0,x_0) + f\big(t_0,x_0,v(t_0,x_0),\sigma^{\top}(t_0,x_0)p\big)\le 0.
\end{equation*}

(c). We call  $u\in C((0,T]\times\mathbb{R}^d;\mathbb{R})$
a viscosity solution of \eqref{e:PDE-Obstacle} if it is both a viscosity subsolution and a viscosity supersolution.

\end{definition}

Consider the following system of SDEs and reflected BSDEs 
\begin{equation}\label{e:system of RBSDEs}
\left\{\begin{aligned}
&Y^{t,x}_{s} = g(S^{t,x}_{T}) + \int_{s}^{T}f(r,S^{t,x}_r,Y^{t,x}_r,Z^{t,x}_r)dr  - \int_{s}^{T}Z^{t,x}_r dW_{r} + K^{t,x}_{T} - K^{t,x}_{s}, \\
& Y^{t,x}_{s}\ge l(s,S^{t,x}_{s}), \quad \int_{s}^{T} \left(Y^{t,x}_{r} - l(r,S^{t,x}_{r})\right) dK^{t,x}_{r} = 0 , \quad s\in[t,T], \\
&S^{t,x}_{s} = x + \int_{t}^{s} b(r,S^{t,x}_{r}) dr + \int_{t}^{s} \sigma(r,S^{t,x}_{r}) dW_r,\quad s\in[t,T].
\end{aligned}
\right.
\end{equation}

According to \cite[Lemma~5.1, Theorem~5, Corollary~2]{KLQT}, we have the following existence and uniqueness for the viscosity solution of obstacle problem.

\begin{proposition}\label{prop:exist of PDE}
Assume that $f$ satisfies Assumption~\ref{(H0)}, $(l,g) \in C_b([0,T]\times \R^{d},\R) \times C_b(\R^{d},\R)$ such that $g(x)\ge l(T,x),$ and $(b,\sigma)$ satisfies \ref{(H_b,sigma)}. In addition, suppose that $f$ satisfies 
\begin{equation}\label{e:f_x'}
|D_{x}f(t,x,y,z)|\le C'( 1 + |z|^{2} ), 
\end{equation} 
where $C'>0$ is a constant. Let $Y^{t,x}_{\cdot}$ be the unique solution of Eq.~\eqref{e:system of RBSDEs}. Then $u(t,x):=Y^{t,x}_{t}\in C([0,T]\times\R^d,\R),$ and $u$ is the unique viscosity solution of Eq.~\eqref{e:PDE-Obstacle}.
\end{proposition}

 According to Lemma~\ref{lem40}, the growth of $z\mapsto D_y \widetilde{f}^{n,T'}(t,x,y,z)$ in the proof of Theorem~\ref{thm:well-posedness} does not satisfy the condition~(A3) in Assumption~\ref{(H0)}, but satisfies a weaker condition~(A3') in \ref{(H1)}. For this reason, we will propose the comparison theorem for the obstacle problem under Assumption~\ref{(H1)}. And it will be proved in the same spirit of \cite[Theorem C.1]{DF},
by showing that the parameter of Eq.~\eqref{e:PDE-Obstacle} can be transformed to the $f$ satisfying the following structure conditions 
\begin{equation}\label{e:structure cond}
\begin{aligned}
&|f(t,x,y,z)|\le \bar{C}(1 + |y|) (1+ |z|^{2}),\\
&|D_x f(t,x,y,z)|\le \bar{C}(1+|z|^2),\\	
&D_y f(t,x,y,z)\le -\bar{K}(1 + |z^2|),\\
&|D_z f(t,x,y,z)|\le \bar{C}(1 + |z| + |y||z|),
\end{aligned}
\end{equation}
for all $(t,x,y,z)\in[0,T]\times\R^d\times\R\times\R^d.$
We outline that the key of \eqref{e:structure cond} is the third condition, which corresponds to the ``proper coefficient" defined in \cite{crandall1992user}. First we show some properties of the test function.

\begin{lemma}\label{lem:psi^delta,zeta}
Let  $u,v:(0,T]\times\R^{d}\rightarrow \R.$ Suppose that $u$ (resp. $v$) is upper (resp. lower) semi-continuous. Assume further that \begin{equation*}
\sup_{(t,x,x')\in(0,T]\times \R^d} \left[u(t,x) - v(t,x')\right] <\infty,
\end{equation*} 
and
\begin{equation*}
\lim_{t\rightarrow 0}\left[u(t,x) - v(t,x)\right] = -\infty,\quad \text{ uniformly on }\R^d.
\end{equation*}
Set 
\begin{equation}\label{e:def of psi}
\psi^{\delta,\zeta}(t,x,x') := u(t,x) - v(t,x') - \frac{1}{\delta^2}|x-x'|^2 - \zeta\left(|x|^2 + |x'|^2\right).
\end{equation} 
For every $\delta,\zeta>0$ choose (depending on $(\delta,\zeta)$)
\begin{equation}\label{e:def of hat t}
(\hat{t},\hat{x},\hat{x}') \in \underset{(t,x,x')\in(0,T]\times\R^d\times\R^d}{\text{arg\ max}}\psi^{\delta,\zeta}(t,x,x').
\end{equation} 
Then we have
\begin{itemize}
\item[(i).] 
\begin{equation}\label{e:|x-hat x|/delta^2 -> 0}
   \limsup_{\zeta\rightarrow 0}\limsup_{\delta\rightarrow 0} \left[\frac{|\hat{x} - \hat{x}'|^2}{\delta^2} + \zeta(|\hat{x}|^{2} + |\hat{x}'|^2)\right] = 0;
\end{equation}
\item[(ii).] 
\begin{equation*}
\limsup_{\zeta\rightarrow 0}\limsup_{\delta\rightarrow 0}\psi^{\delta,\zeta}(\hat{t},\hat{x},\hat{x}') = \liminf_{\zeta\rightarrow 0}\liminf_{\delta\rightarrow 0}\psi^{\delta,\zeta}(\hat{t},\hat{x},\hat{x}') = \sup_{(t,x)\in(0,T]\times\R^d} \left[u(t,x) - v(t,x)\right];
\end{equation*}
\item[(iii).] 
\begin{equation}\label{e:u - v = theta}
\begin{aligned}
\limsup_{\zeta\rightarrow 0}\limsup_{\delta\rightarrow 0}\left[u(\hat{t},\hat{x}) - v(\hat{t},\hat{x}')\right] &= \liminf_{\zeta\rightarrow 0}\liminf_{\delta\rightarrow 0}\left[u(\hat{t},\hat{x}) - v(\hat{t},\hat{x}')\right] \\
&= \sup_{(t,x)\in(0,T]\times\R^d} \left[u(t,x) - v(t,x)\right].
\end{aligned}
\end{equation}

\end{itemize}
\end{lemma}
\begin{proof}
From \eqref{e:def of hat t} it follows that for all $\delta,\zeta>0$
\begin{equation*}
\begin{aligned}
\frac{|\hat{x}-\hat{x}'|^2}{\delta^2} + \zeta \left(|\hat{x}|^2 + |\hat{x}'|^2\right)&\le - 
 u(T,0) + v(T,0)
+ u(\hat{t},\hat{x}) - v(\hat{t},\hat{x}')
\\
&\le  -u(T,0) + v(T,0)
+ \sup_{(t,x,x')\in(0,T]\times\R^d\times \R^d}\left[u(t,x) - v(t,x')\right].
\end{aligned}
\end{equation*}
Hence, there exists a constant $C>0$ such that 
\begin{equation}\label{e:bound of hat x and hat x'}
|\hat{x} - \hat{x}'|\le C\delta,\quad |\hat{x}|,|\hat{x}'|\le C/\zeta^{\frac{1}{2}}.
\end{equation}
Furthermore, for every fixed $\zeta>0,$ we have $\lim\limits_{\delta\rightarrow 0}|\hat{x} - \hat{x}'| = 0$ and then 
\begin{equation}\label{e:limsup u - v <= theta}
\limsup_{\delta\rightarrow 0} \left[u(\hat{t},\hat{x}) - v(\hat{t},\hat{x}')\right] \le \sup_{(t,x)\in(0,T]\times\R^d} [u(t,x) - v(t,x)] =: \theta.
\end{equation}
Let $(t^n,x^n)_{n\ge 1}$ be a sequence of point in $(0,T]\times\R^d$ such that 
\begin{equation*}
u(t^n,x^n) - v(t^n,x^n)\ge \theta - \frac{1}{n}.
\end{equation*}
By \eqref{e:def of hat t}, we have 
\begin{equation}\label{e:u(t,x) - v(t,x) <= psi}
u(t^n,x^n) - v(t^n,x^n) - 2\zeta |x^{n}|^2 \le u(\hat{t},\hat{x}) - v(\hat{t},\hat{x}') - \frac{|\hat{x} - \hat{x}'|^2}{\delta^2} - \zeta\left(|\hat{x}|^2 + |\hat{x}'|^2\right).
\end{equation}
Thus, we have
\begin{equation*}
\frac{|\hat{x} - \hat{x}'|^2}{\delta^2} + \zeta(|\hat{x}|^{2} + |\hat{x}'|^2)\le u(\hat{t},\hat{x}) - v(\hat{t},\hat{x}') - \theta + \frac{1}{n} + 2\zeta |x^{n}|^2.
\end{equation*}
Letting $\delta\rightarrow 0,$ by \eqref{e:limsup u - v <= theta} we have
\begin{equation*}
\limsup_{\delta\rightarrow 0} \left[\frac{|\hat{x} - \hat{x}'|^2}{\delta^2} + \zeta(|\hat{x}|^{2} + |\hat{x}'|^2)\right] \le \frac{1}{n} + 2\zeta |x^{n}|^2.
\end{equation*}
In the above inequality, letting $\zeta\rightarrow \infty$ we have 
\begin{equation*}
\limsup_{\zeta\rightarrow 0}\limsup_{\delta\rightarrow 0} \left[\frac{|\hat{x} - \hat{x}'|^2}{\delta^2} + \zeta(|\hat{x}|^{2} + |\hat{x}'|^2)\right]\le \frac{1}{n},
\end{equation*}
and then letting $n\rightarrow 0$ it follows
\begin{equation*}
\limsup_{\zeta\rightarrow 0}\limsup_{\delta\rightarrow 0} \left[\frac{|\hat{x} - \hat{x}'|^2}{\delta^2} + \zeta(|\hat{x}|^{2} + |\hat{x}'|^2)\right] = 0,
\end{equation*}
which proves (i). It remains to show (ii) and (iii). From the fact $\psi^{\delta,\zeta}(t,x,x')\le u(t,x) - v(t,x')$ and \eqref{e:limsup u - v <= theta}, it follows that 
\begin{equation}\label{e:psi <= theta}
\begin{aligned}
\lim\sup_{\zeta\rightarrow 0}\limsup_{\delta\rightarrow 0}\psi^{\delta,\zeta}(\hat{t},\hat{x},\hat{x}') \le \limsup_{\zeta\rightarrow 0}\limsup_{\delta\rightarrow 0}\left[u(\hat{t},\hat{x}) - v(\hat{t},\hat{x}')\right]\le \theta.
\end{aligned}
\end{equation}
On the other hand, taking $\liminf_{\zeta\rightarrow 0}\liminf_{\delta\rightarrow 0}$ to \eqref{e:u(t,x) - v(t,x) <= psi} we have
\begin{equation*}
u(t^n,x^n) - v(t^n,x^n)\le \liminf_{\zeta \rightarrow 0}\liminf_{\delta\rightarrow 0} \psi^{\delta,\zeta}(\hat{t},\hat{x},\hat{x}')\le \liminf_{\zeta\rightarrow 0}\liminf_{\delta\rightarrow 0}\left[u(\hat{t},\hat{x}) - v(\hat{t},\hat{x}')\right].
\end{equation*}
Letting $n\rightarrow \infty$ we have 
\begin{equation}\label{e:theta <= psi}
\theta \le \liminf_{\zeta \rightarrow 0}\liminf_{\delta\rightarrow 0} \psi^{\delta,\zeta}(\hat{t},\hat{x},\hat{x}')\le \liminf_{\zeta \rightarrow 0}\liminf_{\delta\rightarrow 0}\left[u(\hat{t},\hat{x}) - v(\hat{t},\hat{x}')\right]. 
\end{equation}
Then (ii) and (iii) follow from \eqref{e:psi <= theta} and \eqref{e:theta <= psi}.

\end{proof}

Recall that for any $\bar M >0, \varepsilon>0,$ $\Theta(\bar M):=(\bar{M}, c(|\bar{M}|),C_{f},C_{\bar{M}}  )$ and $C_{\varepsilon,\bar{M}}$ are given in Assumption~\ref{(H1)}.

\begin{proposition}
\label{prop:comparison of PDE}
Assume that $(H,f)$ satisfies the assumptions \ref{(Hpr)} and \ref{(H1)}, $(l,g) \in C_b([0,T]\times \R^{d},\R) \times C_b(\R^{d},\R)$ such that $g(x)\ge l(T,x),$ and $(b,\sigma )$ satisfy \ref{(H_b,sigma)}.  
In addition, suppose that $f$ satisfies \eqref{e:f_x'}. Let $u$ (resp. $v$) be a viscosity subsolution (resp. supersolution) of Eq.~\eqref{e:PDE-Obstacle} with $\|u(\cdot,\cdot)\|_{\infty;(0,T]\times\R^{d}}\vee\|v(\cdot,\cdot)\|_{\infty;(0,T]\times\R^{d}}\le \bar{M}$ for some $\bar M>0.$ Then there exists $\varepsilon>0$ depending only on $\Theta(\bar M),$ and $h>0$ depending on $(\varepsilon,\Theta(\bar M) ,  C_{\varepsilon,\bar{M}})$ such that
\begin{equation*}
u(t,x)\le v(t,x),\ \text{ for all }(t,x)\in(T-h\land h_{\varepsilon,\bar{M}},T]\times \mathbb{R}^d.
\end{equation*}

\end{proposition}
\begin{proof}
\textbf{Step 1.} Assuming that the parameter $f$ satisfies the structure conditions~\eqref{e:structure cond}, we will prove our statement by contradiction under this assumption. Firstly, let $u_{\lambda}(t,x): = u(t,x) - \frac{\lambda}{t}.$ Note that $u_{\lambda}$ is also a viscosity subsolution, and note that $u_{\lambda} - v$ is bounded above. Since $u\le v$ follows from $u_{\lambda}\le v$ for all $0< \lambda < 1,$ it is sufficient to prove comparison theorem with the following additional condition 
\begin{equation}\label{e:hat t > 0}
\lim_{t\rightarrow 0}\left[u(t,x) - v(t,x)\right] = -\infty,\quad \text{uniformly on }\R^d.
\end{equation}
Suppose that 
\begin{equation*}
\sup\limits_{(t,x)\in(0,T]\times\R^{d}}\left[u(t,x) - v(t,x)\right] =: \theta >0.
\end{equation*} 
For $\delta,\zeta >0$, recall $\psi^{\delta,\zeta}$ and $(\hat{t},\hat{x},\hat{x}')$ are given by \eqref{e:def of psi} and \eqref{e:def of hat t}, respectively. Also, by \eqref{e:bound of hat x and hat x'} that there exists a constant $C$ such that
\begin{equation}\label{e:bound of hat x and hat x' '}
|\hat{x} - \hat{x}'|\le C\delta,\ |\hat{x}|,|\hat{x}'|\le C/|\zeta|^{1/2}.
\end{equation} 
By \eqref{e:hat t > 0}, there exists $t_0\in(0,T]$ such that for all $\delta,\zeta>0,$ $\hat{t}\in[t_0,T].$ Thus, we have 
\begin{equation*}
\hat{t}^{\zeta}:=\limsup_{\delta\rightarrow 0}\hat{t}(\delta,\zeta)\in[t_0,T].
\end{equation*}
Note that by \eqref{e:bound of hat x and hat x' '} $\hat{x},\hat{x}'$ are bounded for each $\zeta>0,$ and $|\hat{x} - \hat{x}'|\rightarrow 0$ as $\delta\rightarrow 0.$ By extracting a subsequence of $\delta$ (depends on $\zeta$), we may suppose that the sequences $\left\{\hat{x}(\delta,\zeta)\right\}_{\delta}$ and $\left\{\hat{x}'(\delta,\zeta)\right\}_{\delta}$ converge to a common limit $\hat{x}^\zeta \in \R^d$ with $|\hat{x}^{\zeta}|\le C/|\zeta|^{1/2}.$ In the following, we show the contradiction under three (mere) possible cases for the limit $(\hat{t}(\delta,\zeta),\hat{x}(\delta,\zeta),\hat{x}'(\delta,\zeta)).$ 


(1). There exists a subsequence of $\left\{\hat{t}^{\zeta}\right\}$ such that $\hat{t}^{\zeta} = T $ for all $\zeta$ of this subsequence. Since $u - v$ is lower semicontinuous, and since $u(T,x)\le g(x)\le v(T,x),$ it follows that for every $\zeta>0$ 
\begin{equation*}
\liminf_{\delta\rightarrow 0}\left[u(\hat{t}(\delta,\zeta),\hat{x}(\delta,\zeta)) - v(\hat{t}(\delta,\zeta),\hat{x}'(\delta,\zeta))\right] \le u(T,\hat{x}^{\zeta}) - v(T,\hat{x}^{\zeta})\le 0,
\end{equation*}
and thus
\begin{equation*}
\begin{aligned}
&\liminf_{\zeta\rightarrow 0}\liminf_{\delta\rightarrow 0}
\left[u(\hat{t}(\delta,\zeta),\hat{x}(\delta,\zeta)) - v(\hat{t}(\delta,\zeta),\hat{x}'(\delta,\zeta))\right] \\
&\le \liminf_{\zeta\rightarrow 0} \left[u(T,\hat{x}^{\zeta}) - v(T,\hat{x}^{\zeta}) \right]\le 0.
\end{aligned}
\end{equation*}
While from Lemma~\ref{lem:psi^delta,zeta} we have
\begin{equation*}
\liminf_{\zeta\rightarrow 0}\liminf_{\delta\rightarrow 0}
\left[u(\hat{t}(\delta,\zeta),\hat{x}(\delta,\zeta)) - v(\hat{t}(\delta,\zeta),\hat{x}'(\delta,\zeta))\right] =\theta,
\end{equation*}
which leads to a contradiction.

(2) There exists a subsequence $\zeta_n\downarrow 0,n\rightarrow \infty$ such that for each $\zeta_n,$ there exists a subsequence of $\left\{\left(\hat{t}(\delta,\zeta_n),\hat{x}(\delta,\zeta_n)\right)\right\}_{\delta}$  such that
\[u\big(\hat{t}(\delta,\zeta_n),\hat{x}(\delta,\zeta_n)\big) - l\big(\hat{t}(\delta,\zeta_n),\hat{x}(\delta,\zeta_n)\big)\le 0.\]
Since $v(t,x)\ge l(t,x),$ we have that for the subsequence of $\left\{\left(\hat{t}(\delta,\zeta_n),\hat{x}(\delta,\zeta_n)\right)\right\}_{\delta}$ mentioned above, 
\begin{equation*}
u(\hat{t}(\delta,\zeta_n),\hat{x}(\delta,\zeta_n)) - v(\hat{t}(\delta,\zeta_n),\hat{x}'(\delta,\zeta_n))\le l(\hat{t}(\delta,\zeta_n),\hat{x}(\delta,\zeta_n)) - l(\hat{t}(\delta,\zeta_n),\hat{x}'(\delta,\zeta_n)).
\end{equation*}
Taking $\liminf_{n\rightarrow \infty}\liminf_{\delta\rightarrow 0}$ in the above inequality, and using Lemma~\ref{lem:psi^delta,zeta}, the following inequalities follow from \eqref{e:bound of hat x and hat x' '} and the continuity of $l(\cdot,\cdot)$
\begin{equation*}
\begin{aligned}
\theta &= \liminf\limits_{n\rightarrow\infty}\liminf\limits_{\delta\rightarrow 0}\left[ u(\hat{t}(\delta,\zeta_n),\hat{x}(\delta,\zeta_n)) - v(\hat{t}(\delta,\zeta_n),\hat{x}'(\delta,\zeta_n)) \right] \\
&\le \sup_{n\in\mathbb{N}}\lim_{\delta\rightarrow 0}\sup_{\substack{t\in[0,T],|x - x'|\le C\delta\\ \ |x|,|x'|\le C/|\zeta_{n}|^{1/2}}}|l(t,x) - l(t,x')| = 0 .
\end{aligned}
\end{equation*}
This leads to a contradiction. 

(3) There exists a subsequence $\zeta_n\downarrow 0,n\rightarrow \infty$ such that for each $\zeta_n,$ $\hat{t}^{\zeta_n}<T$ and there exists a subsequence of $\left\{(\hat{t}(\delta,\zeta_n),\hat{x}(\delta,\zeta_n)\right\}_{\delta}$ such that
\begin{equation*}
u\big(\hat{t}(\delta,\zeta_n),\hat{x}(\delta,\zeta_n)\big) - l\big(\hat{t}(\delta,\zeta_n),\hat{x}(\delta,\zeta_n)\big) > 0.
\end{equation*}
Now we show a contradiction for the case (3). For each $n,$ denote by $\{\delta^n_m\}_{m}$ the subsequence of $\delta$ derived from the subsequence of $\left\{(\hat{t}(\delta,\zeta_n),\hat{x}(\delta,\zeta_n)\right\}_{\delta}$ mentioned above.
Let $G:[0,T]\times \R^{d} \times \R\times\R^{d}\times\mathcal{S}(d)\rightarrow \R$ be defined by 
\begin{equation*}
G(t,x,y,p,\Lambda) := f(t,x,y,\sigma^{\top}(t,x)p) + \frac{1}{2}\text{Tr}\left\{\Lambda\sigma(t,x)\sigma^{\top}(t,x)\right\} + p^{\top}b(t,x).
\end{equation*}
Set $\hat{p}:=2(\hat{x} - \hat{x}')/\delta^2 + 2\zeta\hat{x}$ and $\hat{q}:=2(\hat{x} - \hat{x}')/\delta^2 - 2\zeta\hat{x}'.$ By \cite[Eq. (30)]{DF}, for each $\delta \in \{\delta^n_m\}_m$ and $\zeta \in \{\zeta_n\}_n,$
there exist $\Lambda,\Gamma\in \mathcal{S}(d)$ such that
(i)
\begin{equation}\label{e:Lambda Gamma}
\begin{pmatrix}
\Lambda& 0\\
0& -\Gamma
\end{pmatrix}
\le
\frac{4}{\delta^2}
\begin{pmatrix}
I_d& -I_d\\
-I_d& I_d
\end{pmatrix}
+ 4\zeta
\begin{pmatrix}
I_d& 0\\
0& -I_d
\end{pmatrix},
\end{equation}
where $I_d$ represents the unit matrix in $\mathbb{R}^{d\times d};$ (ii)
\begin{equation}\label{e:G>=G}
G(\hat{t},\hat{x},u(\hat{t},\hat{x}),\hat{p},\Lambda)\ge 0\ge G(\hat{t},\hat{x}',v(\hat{t},\hat{x}'),\hat{q},\Gamma).
\end{equation}
Multiplying by $(\sigma(\hat{t},\hat{x}),\sigma(\hat{t},\hat{x}'))$ and its transpose from the left side and right side respectively, \eqref{e:Lambda Gamma} yields that for some constant $C>0$
\begin{equation}\label{e:Lambda - Gamma}
\begin{aligned}
\text{Tr}\left\{\Lambda \sigma\sigma^{\top}(\hat{t},\hat{x})\right\} - \text{Tr}\left\{\Gamma \sigma\sigma^{\top}(\hat{t},\hat{x}') \right\} &\le \frac{4|\sigma(\hat{t},\hat{x}) - \sigma(\hat{t},\hat{x}')|}{\delta^2} + 4\zeta (|\sigma(\hat{t},\hat{x})|^2 + |\sigma(\hat{t},\hat{x}')|^2)\\
&\le C \left(\frac{|\hat{x} - \hat{x}'|^2}{\delta^2} + \zeta \right).
\end{aligned}
\end{equation}
Also note that  (recall $C_S$ is the constant in \ref{(H_b,sigma)})
\begin{equation}\label{e:hat p - hat q}
\begin{aligned}
\hat{p}^{\top}b(\hat{t},\hat{x}) - \hat{q}^{\top}b(\hat{t},\hat{x}') &\le \left(\hat{p}^{\top} - \hat{q}^{\top}\right)b(\hat{t},\hat{x}) + \hat{q}^{\top}\left(b(\hat{t},\hat{x}) - b(\hat{t},\hat{x}')\right)\\
& = 2C_{S}\left(\zeta(|\hat{x}| + |\hat{x}'|) + \frac{|\hat{x} - \hat{x}'|^2}{\delta^2} + \zeta|\hat{x}'|\cdot|\hat{x} - \hat{x}'|\right).
\end{aligned}
\end{equation}
Combining \eqref{e:G>=G}, \eqref{e:Lambda - Gamma}, and \eqref{e:hat p - hat q} we have  
\begin{equation}\label{e:f(u) - f(v) >= 0}
\begin{aligned}
&f\big(\hat{t},\hat{x},u(\hat{t},\hat{x}),\sigma^{\top}(\hat{t},\hat{x})\hat{p}\big) - f\big(\hat{t},\hat{x}',v(\hat{t},\hat{x}'),\sigma^{\top}(\hat{t},\hat{x}')\hat{q}\big) \\
&\ge - \left(\text{Tr}\left\{\Lambda \sigma\sigma^{\top}(\hat{t},\hat{x})\right\} - \text{Tr}\left\{\Gamma \sigma\sigma^{\top}(\hat{t},\hat{x}') \right\} + \hat{p}^{\top} b(\hat{t},\hat{x}) - \hat{q}^{\top}b(\hat{t},\hat{x}')\right)\\
&\ge  -O\left(\frac{|\hat{x} - \hat{x}'|^2}{\delta^2} + \zeta (|\hat{x}|^2  + |\hat{x}'|^2 + 1) \right).
\end{aligned}
\end{equation}
On the other hand, \cite[pp.~42]{DF} shows that for every $\delta,\zeta>0,$ whenever $\delta^2<\frac{\bar{C}\sqrt{\theta}}{3\bar{K}}$ it follows that 
\begin{equation}\label{e:f(u) - f(v) <= -K/3 theta}
\begin{aligned}
&f\big(\hat{t},\hat{x},u(\hat{t},\hat{x}),\sigma^{\top}(\hat{t},\hat{x})\hat{p}\big) - f\big(\hat{t},\hat{x}',v(\hat{t},\hat{x}'),\sigma^{\top}(\hat{t},\hat{x}')\hat{q}\big) \\
&\le -\frac{\bar{K}}{3}(u(\hat{t},\hat{x}) - v(\hat{t},\hat{x}')) + \left(2\frac{|\hat{x} - \hat{x}'|^2}{\delta^2} + 2 C_{S}\zeta |\hat{x}| \cdot |\hat{x} - \hat{x}'| + C_{S}|2\zeta \hat{x} + 2\zeta \hat{x}'|\right)^2  \\
&\qquad\qquad\qquad\qquad\qquad\qquad \times \frac{\bar{C}}{6\bar{K}(u(\hat{t},\hat{x}) - v(\hat{t},\hat{x}'))}.
\end{aligned}
\end{equation}
Finally, taking the limit of $\{\delta^{n}_m\}_{m}$ and then letting $n\rightarrow \infty$ in  \eqref{e:f(u) - f(v) >= 0} and \eqref{e:f(u) - f(v) <= -K/3 theta}, it follows form the equality~\eqref{e:|x-hat x|/delta^2 -> 0} and \eqref{e:u - v = theta} that (where $(\hat{t},\hat{x},\hat{x}') = (\hat{t}(\delta^n_m,\zeta_n),\hat{x}(\delta^n_m,\zeta_n),\hat{x}'(\delta^n_m,\zeta_n))$)
\begin{equation*}
0\le \limsup_{n\rightarrow \infty}\limsup_{m\rightarrow \infty}\left[f\big(\hat{t},\hat{x},u(\hat{t},\hat{x}),\sigma^{\top}(\hat{t},\hat{x})\hat{p}\big) - f\big(\hat{t},\hat{x}',v(\hat{t},\hat{x}'),\sigma^{\top}(\hat{t},\hat{x}')\hat{q}\big)\right] \le - \frac{\bar{K}}{3}\theta,
\end{equation*}
which leads to a contradiction.\\

\textbf{Step 2} Now we transform the coefficients of Eq.~\eqref{e:PDE-Obstacle} to meet conditions in \eqref{e:structure cond} when $T$ is small enough. Similar arguments can be found in \cite[Theorem C.1]{DF} and \cite[Theorem 6]{KLQT}. For $A,\lambda >0$ to be determined later, let 
\[\varphi(\tilde{y}):=\frac{1}{\lambda}\ln\left(\frac{e^{\lambda A \tilde{y} }+ 1}{A}\right):\mathbb{R}\rightarrow\left(-\frac{\ln(A)}{\lambda},\infty\right).\]
Set 
$
\tilde{M} := \bar{M} + 1,
$
and let $K>0$ with
$A >e^{2\lambda \tilde{M}e^{KT}}.$
Then for every $t\in[0,T]$, $\left\{e^{Kt}(y - \tilde{M}):y\in[-\bar{M},\bar{M}]\right\}$ is contained in the range of $\varphi.$ Denoting by $\tilde{y} := \varphi^{-1}(e^{Kt}(y - \tilde{M}))$ we have that for every $t\in[0,T]$
\begin{equation*}
\inf\left\{\varphi'(\tilde{y}):y\in[-\bar{M},\bar{M}]\right\} = \inf\left\{\frac{Ae^{\lambda A \tilde{y}}}{1 + e^{\lambda A \tilde{y}}}:y\in[-\bar{M},\bar{M}]\right\}>0,
\end{equation*}
 and
\begin{equation}\label{e:[-bar M,bar M] is equal to}
y\in[-\bar{M},\bar{M}] \Longleftrightarrow \tilde{y}\in[\varphi^{-1}(-e^{Kt}(2\bar{M}+1)),\varphi^{-1}(-e^{Kt})].
\end{equation}
Set 
$
\tilde{u}(t,x) := \varphi^{-1}(e^{Kt}(u(t,x) - \tilde{M})). 
$
Then Eq.~\eqref{e:PDE-Obstacle} is transformed to the following PDE (where $\sigma,b$ take values at $(t,x)$)
\begin{equation*}
\left\{\begin{aligned}
&D_{t}\tilde{u}(t,x) + \frac{1}{2}\text{Tr}\left\{\sigma\sigma^{\top}D^{2}_{xx} \tilde{u}(t,x)\right\}   + D_{x}\tilde{u}(t,x)^{\top}b 
+ \tilde{f}\big(t,x,\tilde{u}(t,x), \sigma^{\top}  D_{x} \tilde{u}(t,x) \big) = 0\\
&\tilde{u}(T,x) = \tilde{g}(x),\quad \tilde{u}(t,x)\ge \tilde{l}(t,x),
\end{aligned}\right.
\end{equation*}
where $\tilde{g}(x) := \varphi^{-1}(e^{Kt}(g(x)-\tilde{M})),$ $\tilde{l}(t,x) := \varphi^{-1}(e^{Kt}(l(t,x) - \tilde{M})),$ and
\[\tilde{f}(t,x,\tilde{y},\tilde{z}) := \frac{\varphi''(\tilde{y})}{2\varphi'(\tilde{y})}|\tilde{z}|^{2} + \frac{1}{\varphi'(\tilde{y})}e^{Kt}f\left(t,x,\varphi(\tilde{y})e^{-Kt}+\tilde{M},\varphi'(\tilde{y})e^{-Kt}\tilde{z}\right) - K\frac{\varphi(\tilde{y})}{\varphi'(\tilde{y})}.\]
By the condition (A3') in \ref{(H1)}, for every $\varepsilon>0,$ there exists $h_{\varepsilon,\bar{M}}>0$ such that whenever $T\le h_{\varepsilon,\bar{M}},$ we have for all $(t,x,y,z)\in [0,T]\times \R^d\times [-\bar{M},\bar{M}]\times \R^d,$
\begin{equation}\label{e:Dy f <=}
\begin{aligned}
&|f(t,x,y,z)|\le C + C|z|^2,\\
&D_y f(t,x,y,z)\le \varepsilon|z|^2 + C_{\varepsilon},\\
&|D_z f(t,x,y,z)|\le C + C|z|, 
\end{aligned}
\end{equation} 
where $C$ is a constant
only depends on $c(\cdot),$ $C_{f},$ $\bar{M},$ and $C_{\bar{M}};$ and $C_{\varepsilon}$ is a constant independent of $K.$
Fix an $\varepsilon>0$ that is to be determined later. By \eqref{e:[-bar M,bar M] is equal to} and \eqref{e:Dy f <=}, for each $(t,x,\tilde{z})\in[0,T]\times \R^d\times \R^d$ and $\tilde{y}\in [\varphi^{-1}(-e^{Kt}(2\bar{M}+1)),\varphi^{-1}(-e^{Kt})],$ it follows that  
\begin{equation}\label{e:Fy}
\begin{aligned}
&D_{\tilde{y}} \tilde{f}(t,x,\tilde{y},\tilde{z})\\
&= -\frac{|\varphi''(\tilde{y})|^2 - \varphi'''(\tilde{y})\varphi'(\tilde{y})}{2|\varphi'(\tilde{y})|^2}|\tilde{z}|^{2}  - \frac{\varphi''(\tilde{y})}{|\varphi'(\tilde{y})|^2}e^{Kt}f\left(\Theta\right)  +  D_y f\left(\Theta\right)\\
&\quad  + \frac{\varphi''(\tilde{y})}{\varphi'(\tilde{y})}e^{Kt}\langle D_z f\left(\Theta\right),\tilde{z}\rangle- K + K\frac{\varphi(\tilde{y})\varphi''(\tilde{y})}{|\varphi'(\tilde{y})|^{2}}\\
&\le -\frac{|\varphi''(\tilde{y})|^2 - \varphi'''(\tilde{y})\varphi'(\tilde{y})}{2|\varphi'(\tilde{y})|^2}|\tilde{z}|^{2} + C|\varphi''(\tilde{y})|e^{-Kt}|\tilde{z}|^{2} + C\frac{|\varphi''(\tilde{y})|}{|\varphi'(\tilde{y})|^2}e^{Kt}\\
&\quad + \varepsilon e^{-2Kt}|\varphi'(\tilde{y})|^{2}|\tilde{z}|^{2} + C_\varepsilon + C\left(\frac{|\varphi''(\tilde{y})|}{|\varphi'(\tilde{y})|}|\tilde{z}| + |\varphi''(\tilde{y})|e^{-Kt}|\tilde{z}|^{2}\right) - K + K\frac{\varphi''(\tilde{y})\varphi(\tilde{y})}{|\varphi'(\tilde{y})|^{2}}\\
&\le \left(-\frac{|\varphi''(\tilde{y})|^2 - \varphi'''(\tilde{y})\varphi'(\tilde{y})}{2|\varphi'(\tilde{y})|^2}  
 + 3C |\varphi''(\tilde{y})|e^{-Kt}
+ \varepsilon |\varphi'(\tilde{y})|^{2}e^{-2Kt}\right)|\tilde{z}|^2 \\ 
&\quad+ 
 2C\frac{|\varphi''(\tilde{y})|}{|\varphi'(\tilde{y})|^2}e^{Kt}
+ C_\varepsilon + \left(-1 + \frac{\varphi''(\tilde{y})\varphi(\tilde{y})}{|\varphi'(\tilde{y})|^2}\right)K.
\end{aligned}
\end{equation}
where 
$
\Theta = (t,x,y,z) :=(t,x,\varphi(\tilde{y})e^{-Kt}+\tilde{M},\varphi'(\tilde{y})e^{-Kt}\tilde{z}),
$
 and the second inequality we use the fact that
\begin{equation*}
C\frac{|\varphi''(\tilde{y})|}{|\varphi'(\tilde{y})|}|\tilde{z}|\le C\frac{|\varphi''(\tilde{y})|}{|\varphi'(\tilde{y})|^2}e^{Kt} + C|\varphi''(\tilde{y})|e^{-Kt}|\tilde{z}|^2.
\end{equation*}
Note that  for $y\in[-\bar{M},\bar{M}],$ 
\begin{equation}\label{e:y - tilde M<-1}
\varphi''(\tilde{y})>0\text{ and }\varphi(\tilde{y}) = e^{Kt}(y-\tilde{M})\le -e^{Kt}.
\end{equation}
Combined with the inequality \eqref{e:Fy}, we obtain that
\begin{equation*}
\begin{aligned}
D_{\tilde{y}} \tilde{f}(t,x,\tilde{y},\tilde{z})&\le \left(-\frac{|\varphi''(\tilde{y})|^2 - \varphi'''(\tilde{y})\varphi'(\tilde{y})}{2|\varphi'(\tilde{y})|^2} + 
3C |\varphi''(\tilde{y})|e^{-Kt}
+ \varepsilon |\varphi'(\tilde{y})|^{2}e^{-2Kt}\right)|\tilde{z}|^2 \\ 
&\quad+ \frac{|\varphi''(\tilde{y})|}{|\varphi'(\tilde{y})|^2}e^{Kt}
 \left(2C - K\right)
+ C_\varepsilon -K,
\end{aligned}
\end{equation*}
in which 
\begin{equation}\label{e:K}
-\theta_1:=\frac{|\varphi''(\tilde{z})|}{|\varphi'(\tilde{z})|^2}e^{Kt}
\left(2C - K\right)
+ C_\varepsilon -K<0\text{ when } K \text{ is large enough.}
\end{equation}
To show the term involving $|\tilde{z}|^{2}$ negative, notice that
\begin{equation*}
\begin{aligned}
&-|\varphi''(\tilde{y})|^2 + \varphi'''(\tilde{y})\varphi'(\tilde{y}) + 
 6C \varphi''(\tilde{y})|\varphi'(\tilde{y})|^2 e^{-Kt}
 +  2\varepsilon |\varphi'(\tilde{y})|^{4}e^{-2Kt}\\ 
&\le -|\varphi''(\tilde{y})|^2 + \varphi'''(\tilde{y})\varphi'(\tilde{y}) + 
6C \varphi''(\tilde{y})|\varphi'(\tilde{y})|^2
+  2\varepsilon |\varphi'(\tilde{y})|^{4}\\
&= \frac{(-\lambda^2 A^4 + 
 6C \lambda A^4)e^{3\lambda A\tilde{y}}
+  2\varepsilon A^4 e^{4\lambda A\tilde{y}}}{|1 + e^{\lambda A\tilde{y}}|^4}.
\end{aligned}
\end{equation*} 
 Set 
$
\underline{M}(t):=\varphi^{-1}(-e^{Kt}(2\bar{M} + 1))$ and $ \overline{M}(t):=\varphi^{-1}(-e^{Kt}).
$
By \eqref{e:[-bar M,bar M] is equal to} and \eqref{e:y - tilde M<-1}, we have $y - \tilde{M}\le 0$ when $\tilde{y}\in[\underline{M}(t),\overline{M}(t)].$ Assume $\varepsilon\le A^{-1}.$ Then it follows that for all $\tilde{y}\in[\underline{M}(t),\overline{M}(t)],$
\begin{equation*}
\varepsilon e^{4\lambda A \tilde{y}} \le e^{-\ln(A)} e^{4\ln(A e^{\lambda(y -\bar{M})e^{Kt}}-1)}\le e^{3\ln(A e^{\lambda(y -\bar{M})e^{Kt}}-1)} = e^{3\lambda A \tilde{y}}.
\end{equation*}
Furthermore, for $\tilde{y}\in[\underline{M}(t),\overline{M}(t)]$ by choosing $\lambda$ with $-\lambda^2  +   6C \lambda
+ 2 < 0$, we have
\begin{equation}\label{e:lambda}
-\theta_2:=\sup_{\tilde{y}\in[\underline{M}(T),\overline{M}(0)]}\frac{(-\lambda^2 A^4 + 
6C \lambda A^4
)e^{3\lambda A\tilde{y}} + 2 A^4 e^{3\lambda A\tilde{y}}}{2 |1 + e^{\lambda A\tilde{y}}|^4 \cdot|\varphi'(\tilde{y})|^{2}}<0.
\end{equation}
Thus, noting that $[\underline{M}(t),\overline{M}(t)]\subset [\underline{M}(T),\overline{M}(0)],$
\eqref{e:Fy} implies that
\begin{equation}\label{e:proper}
D_{\tilde{y}} \tilde{f}(t,x,\tilde{y},\tilde{z})\le -\theta_{1} - \theta_{2}|\tilde{z}|^2,\ \text{ for all }\tilde{y}\in[\underline{M}(t),\overline{M}(t)].
\end{equation}
In conclusion, we can choose the parameter $(A,K,\varepsilon,\lambda)$ in the following way.
Firstly, choose $\lambda>0$ such that \eqref{e:lambda} holds with
$-\lambda^2 + 
 6C\lambda
+ 2 < 0;$
secondly, let 
$
\varepsilon = e^{-4\lambda \tilde{M} e};$
thirdly, find $K(\bar{M},C)>0$ depends on $\bar{M}$ and $C$ such that \eqref{e:K} holds with
$-K +  2C 
< 0$    and $  -K + C_{\varepsilon} < 0;$
finally, let 
$A := e^{4\lambda \tilde{M} e^{KT}}.$
Then the following conditions 
\[D_{y} \tilde{f}(t,x,\tilde{y},\tilde{z})\le \varepsilon|\tilde{z}|^2 + C_{\varepsilon}\ \text{ and }\  \varepsilon \le A^{-1}\   \text{ are true whenever }\  T\le K^{-1}\land h_{\varepsilon,\bar{M}}.
\] 
In addition, when $T\le h\land h_{\varepsilon,\bar{M}}$ with $h:=K^{-1},$ \eqref{e:proper}, \eqref{e:f_x'}, and Assumption~\ref{(H1)} yield that there exists some $\bar{C}>0$ such that for all $(t,x,\tilde{z})\in[0,T]\times \R^d\times \R^d$ and $\tilde{y}\in [\underline{M}(t),\overline{M}(t)]$
\begin{equation*}
\begin{aligned}
&|\tilde{f}(t,x,\tilde{y},\tilde{z})|\le \bar{C} (1+ |\tilde{z}|^{2}),\\
&|D_x \tilde{f}(t,x,\tilde{y},\tilde{z})|\le \bar{C} (1+|z|^2),\\	
&D_{\tilde{y}} \tilde{f}(t,x,\tilde{y},\tilde{z})\le -\bar{K}(1 + |z^2|),\\
&|D_{\tilde{z}} \tilde{f}(t,x,\tilde{y},\tilde{z})|\le \bar{C} (1 + |\tilde{z}| ).
\end{aligned}
\end{equation*}
Let
\begin{equation*}
\bar{f}(t,x,\tilde{y},\tilde{z}) := \left\{\begin{aligned}
&\tilde{f}\big(t,x,\underline{M}(t),\tilde{z}\big) - \bar{K}(1+|\tilde{z}|^2)(\tilde{y} - \underline{M}(t)),\ \tilde{y} < \underline{M}(t);\\
&\tilde{f}(t,x,\tilde{y},\tilde{z}),\ \tilde{y}\in[\underline{M}(t),\overline{M}(t)];\\
&\tilde{f}\big(t,x,\underline{M}(t),\tilde{z}\big) - \bar{K}(1+|\tilde{z}|^2)(\tilde{y} - \overline{M}(t)),\ \tilde{y} > \overline{M}(t).
\end{aligned}
\right.
\end{equation*}
Then $\bar{f}$ satisfies the condition~\eqref{e:structure cond}.
By step~1, we have 
\begin{equation*}
u(t,x)\le v(t,x),\ \text{ for all }(t,x)\in(0,T]\times \mathbb{R}^d.
\end{equation*}
For an arbitrary terminal time $T,$ by taking the transition $\tilde{u}(\cdot, x):=u(\cdot+T-h\land h_{\varepsilon,\bar{M}},x),$ then 
\begin{equation*}
u(t,x)\le v(t,x),\ \text{ for all }(t,x)\in(T-h\land h_{\varepsilon,\bar{M}},T]\times \mathbb{R}^d,
\end{equation*}
which completes the proof.
\end{proof}

In Section~\ref{subsec:PDE obstacle} we need the semi-relaxed limits preserve the  terminal value at time $T.$ The following proposition is an adaption of \cite[Proposition~2]{caruana2011rough} in the obstacle setting.

\begin{proposition}\label{prop:preserve terminal value}
Denote $(f^{0},l^{0},g^{0}) := (f,l,g).$ Suppose that for $n\ge 0$ $f^{n}:[0,T]\times\R^d\times\R\times\R^{d}\rightarrow \R,$  $l^{n}:[0,T]\times\R^d\rightarrow \R,$ and $g^{n}:\R^d \rightarrow \R$ are continuous. Assume $f^n(\cdot)\rightarrow f^{0}(\cdot),$ $l^{n}(\cdot)\rightarrow l^{0}(\cdot),$ and $g^{n}(\cdot)\rightarrow g^{0}(\cdot)$ locally uniformly, and assume $g^{n}(x)\ge l^{n}(T,x)$ for all $n\ge 0.$
In addition, suppose for $n\ge 1$, $v^n$ is a viscosity solution of the obstacle problem for PDE~\eqref{e:PDE-Obstacle} with parameter $(f^n, l^n, g^n)$. Assume further $\|v^n\|_{\infty;(0,T]\times\R^d}$
is uniformly bounded in $n\ge 1.$ Then $\overline{v}(T,x) = \underline{v}(T,x) = g^{0}(x),$ where $\overline{v} = \limsup\limits_{n\rightarrow \infty}{}^{*}v^{n}$ and $\underline{v} = \limsup\limits_{n\rightarrow \infty}{}_{*}v^{n}$ are defined in \eqref{e:semi-relaxed}.
\end{proposition}

\begin{proof}
Clearly, $\overline{v}(T,x)\ge g^{0}(x).$ 
 Suppose that $\overline{v}(T,x) > g^{0}(x),$ 
and thus there exist $\delta>0$ and $x_0\in\R^d$ such that
\begin{equation*}
\overline{v}(T,x_0) - g^{0}(x_{0}) = \delta.
\end{equation*}
Assume without loss of generality that $g^{0}(x_0) = 0.$ By the definition of $\overline{v},$ there exists $n_k\rightarrow \infty$ and $(s_k,y_k)\rightarrow (T,x_0)$ such that
\begin{equation*}
\lim_{k\rightarrow\infty}v^{n_k}(s_k,y_k) = \overline{v}(T,x_0) = \delta.
\end{equation*}  
Note that $(g^{n} , l^{n}  )$ converges to $(g^{0} , l^0)$ locally uniformly, and $l^{0}(T,x_0)\le g^{0}(x_0).$ Then there exists $\rho \in (0,T)$ such that $g^{n}(x)<\frac{\delta}{2}$ and $l^{n}(t,x)<\frac{\delta}{2}$ whenever $|t-T| \vee|x-x_0| < \rho$ and $n$ is large enough. Define the test function 
\begin{equation*}
\phi(t,x) := \lambda (T-t) + K |x-x_{0}|^2,
\end{equation*} 
where $K := \sup\limits_{n\ge 1}\|v^n\|_{\infty;(0,T]\times \R^d}/\rho^2$ and  $\lambda > \sup\limits_{n\ge 1}\|v^n\|_{\infty;(0,T]\times \R^d}/\rho$ will be chosen later. Let 
\begin{equation*}
(\hat{t}^n,\hat{x}^n)\in \text{arg\ max}\left\{v^n(t,x) - \phi(t,x):(t,x)\in
 (0,T]\times\R^d
\right\}.
\end{equation*}
Denote $O_{\rho}(T,x_0) := \left\{(t,x)
\in(0,T]\times\R^d
;|t-T|\vee |x-x_0|\le \rho\right\}.$ First we claim that there exists $k'\ge 1$ such that $(\hat{t}^{n_k},\hat{x}^{n_k})\in O_{\rho}$ for all $k > k'.$ Actually, for $(t,x) \in (0,T]\times\R^d \setminus O_{\rho}$
and $n\ge 1,$
\begin{equation*}
v^{n}(t,x) - \phi(t,x) \le \sup_{n\ge 1}\|v^n\|_{\infty;(0,T]\times\R^d} - \lambda \rho \land K|\rho|^2 \le 0. 
\end{equation*}
While
\begin{equation}\label{e:>= delta/2}
v^{n_k}(\hat{t}^{n_k},\hat{x}^{n_k}) - \phi(\hat{t}^{n_k},\hat{x}^{n_k})\ge v^{n_k}(s_{k},y_k) - \phi(s_k,y_k)\ge   \frac{\delta}{2},\ \text{when $k$ is sufficiently large.}
\end{equation}
In addition, we claim that $\hat{t}^{n_k}\neq T$ when $k$ is sufficiently large. Indeed, if $\hat{t}^{n_k} = T,$
\begin{equation*}
v^{n_k}(T,\hat{x}^{n_k}) - \phi(T,\hat{x}^{n_k}) = g^{n_k}(\hat{x}^{n_k}) - \phi(T,\hat{x}^{n_k}) \le g^{n_k}(\hat{x}^{n_k}) < \frac{\delta}{2}\ \text{ when $k$ is large,}
\end{equation*}
which leads to a contraction with \eqref{e:>= delta/2} if $\hat{t}^{n_k} = T.$ Therefore, we have $v^{n_k}(\hat{t}^{n_k},\hat{x}^{n_k}) > l^{n_k}(\hat{t}^{n_k},\hat{x}^{n_k})$ when $k$ is large enough since $l^{n_k}(\hat{t}^{n_k},\hat{x}^{n_k})<\frac{\delta}{2}$ while \eqref{e:>= delta/2} holds for sufficiently large $k.$
Hence, by choosing $k$ large enough, we have 
\begin{equation*}
\hat{t}^{n_k} < T \text{ and } v^{n_k}(\hat{t}^{n_k},\hat{x}^{n_k}) > l^{n_{k}}(\hat{t}^{n_k},\hat{x}^{n_k}).
\end{equation*}
Recall that $I_d$ is the unit matrix in $\R^{d\times d}.$  Then by the viscosity subsolution property of $v^{n_k},$ and by the fact that $(-\lambda,2K(\hat{x}^{n_k} - x_0),K I_{d})\in P^{2,+} v^{n_k}(\hat{t}^{n_k},\hat{x}^{n_k}),$ we have
\begin{equation}\label{e:0 >= lambda - G^n}
0 \le -\lambda + G^{n_{k}}\big(\hat{t}^{n_k},\hat{x}^{n_k},v^{n_k}(\hat{t}^{n_k},\hat{x}^{n_k}),2K(\hat{x}^{n_k} - x_0),K I_d\big),
\end{equation}
where $G^{n}(t,x,y,p,\Gamma) := \frac{1}{2} \text{Tr}\left\{\sigma(t,x)\sigma^{\top}(t,x)\Gamma\right\} + p^{\top} b(t,x) + f^{n}\left(t,x,y,\sigma^{\top}(t,x)p\right).$
Note that $(\hat{t}^{n_k},\hat{x}^{n_k})$ remains in the bounded set $O_{\rho}$,  $\rho$  is independent of $\lambda,$ and $\|v^{n_k}\|_{\infty;(0,T]\times\R^d}$ are uniformly bounded in $k\ge 1.$ Then combined with the fact that $G^{n_k}(\cdot)$ converges to $G^{0}(\cdot)$ locally uniformly, we have that $G^{n_k}(\hat{t}^{n_k},\hat{x}^{n_k},v^{n_k}(\hat{t}^{n_k},\hat{x}^{n_k}),2K(x_0 - \hat{x}^{n_k}),K I_d)$ are uniformly  bounded  in $k\ge 1.$ It follows that \eqref{e:0 >= lambda - G^n} leads to a contradiction by taking $\lambda$ large enough. Thus we prove that $\overline{v}(T,x) = g^{0}(x),$ and $\underline{v}(T,x) = g^{0}(x)$ can be proved similarly.
\end{proof}

\renewcommand\thesection{Appendix C}
\section{Some elementary facts about increasing processes}\label{append:C}
\renewcommand\thesection{C}

\begin{lemma}\label{lem:X <= Y}
Assume that $X,Y:[0,T]\times \Omega\rightarrow \R$ are two optional processes. Suppose that for each stopping time $\tau\le T$,
\begin{equation}\label{e:condition}
X_{\tau} \le Y_{\tau} \text{ a.s..}
\end{equation}
Then we have a.s., $
X_{t}\le Y_{t} \text{ for all }t\in[0,T].$  
\end{lemma}
\begin{proof}
Denote  
\begin{equation*}
\mathcal{A} := \left\{(t,\omega)\in[0,T]\times\Omega; X_{t} > Y_{t}\right\}.
\end{equation*}
Since $X$ and $Y$ are optional, the random set $\mathcal{A}$ is also optional. Thus by the optional section theorem (see Dellacherie-Meyer \cite[Chapter~IV (no.83)]{Meyer}), for every $\varepsilon>0,$ there exists a stopping time $\tau'$ such that 
\begin{itemize}
\item[(i)] $\{\big(\tau'(\omega),\omega\big);\omega\in \Omega,\ \tau'(\omega)< \infty\}\subset \mathcal{A}$;
\item[(ii)]$
\mathbb{P}\left\{\omega\in\Omega;\tau'(\omega)<\infty\right\} \ge \mathbb{P}\left\{\omega\in\Omega;\text{ there exists $t<\infty$ such that }(t,\omega)\in \mathcal{A} \right\} - \varepsilon.$
\end{itemize}
Assume that $\mathcal{A}$ is not an evanescent random set, i.e., 
\begin{equation*}
\mathbb{P}\left\{\omega\in\Omega;\text{ there exists $t<\infty$ such that }(t,\omega)\in \mathcal{A} \right\} > 0.
\end{equation*} 
Choosing $\varepsilon$ small enough, it satisfies that
\begin{equation*}
\mathbb{P}\left\{\omega\in\Omega;\tau'(\omega)<\infty\right\}>0 \text{ and } X_{\tau'\land T} > Y_{\tau'\land T}\text{ on }\left\{\omega\in\Omega;\tau'(\omega)<\infty\right\},
\end{equation*}
which contradicts the condition~\eqref{e:condition}.
\end{proof}

Now we apply the above lemma to prove that $K_{\cdot}$ is an a.s. increasing process if it is increasing on  every stopping times.

\begin{lemma}\label{lem:app of sction theorem}
Assume $K:[0,T]\times\Omega\rightarrow \R$ is an optional process with $K_{0} = 0.$ Suppose that for every two stopping times $\tau_1$ and $\tau_2$ that $\tau_1 \le \tau_2\le T,$ it satisfies 
\begin{equation}\label{e:cond2}
K_{\tau_1} \le K_{\tau_2}\text{ a.s..}
\end{equation}
Then, we have a.s.,
\begin{equation*}
K_{t_1} \le K_{t_2}, \text{ for all }0\le t_1 \le t_2\le T.
\end{equation*}
\end{lemma}
\begin{proof}
For simplicity of notations, we denote $K_{t} = 0$ if $t<0.$ Denote by $\mathbb{Q}$ be the set of rational numbers. Define
\begin{equation*}
\widetilde{K}_{t}:=  \sup_{s\in\mathbb{Q},s\le  t}K_{s}.
\end{equation*}
Clearly $\widetilde{K}_t$ is also optional. For an arbitrary stopping time $\tau\le T$, we have by \eqref{e:cond2}
\begin{equation*}
\widetilde{K}_{\tau} = \sup_{s\in\mathbb{Q},s\le  \tau}K_{s} \le \sup_{s\in\mathbb{Q}}K_{s\land\tau} \le K_{\tau} \text{ a.s.. }
\end{equation*}
Thus by Lemma~\ref{lem:X <= Y} we have a.s.
\begin{equation}\label{e:1}
\widetilde{K}_{t} \le K_{t}, \text{ for all }t\in[0,T].
\end{equation}
On the other hand, for every $n\in\mathbb{N}$, set 
$
K^{n}_{t}:= K_{t - \frac{1}{n}}.
$
Clearly, $K^{n}_{t}$ is optional. For an arbitrary stopping time $\tau\le T$, 
\begin{equation*}
K^{n}_{\tau} = K_{\tau - \frac{1}{n}}\le  \sup_{s\in\mathbb{Q}}\left[K_{s}\mathbf{1}_{\{\tau - \frac{1}{n}<s\le \tau\}}\right] \le \sup_{s\in\mathbb{Q},s\le\tau}K_{s} = \widetilde{K}_{\tau}\text{ a.s.,}
\end{equation*}
where we have used the following fact in the first inequality
\begin{equation*}
K_{\tau-\frac{1}{n}}\mathbf{1}_{\{\tau-\frac{1}{n} < s \le \tau\}} \le K_{s\vee(\tau-\frac{1}{n})}\mathbf{1}_{\{\tau-\frac{1}{n} < s\le \tau\}} = K_{s}\mathbf{1}_{\{\tau-\frac{1}{n} < s\le \tau\}} \text{ a.s..}
\end{equation*}
Thus, by Lemma~\ref{lem:X <= Y} again, we have a.s.,
\begin{equation}\label{e:2}
K^{n}_{t} \le \widetilde{K}_{t}, \text{ for all }t\in[0,T].
\end{equation}
For every $n\in\mathbb{N},$ it follows from \eqref{e:1} and \eqref{e:2} that a.s.,
\begin{equation*}
K_{t_1} = K^{n}_{t_1 + \frac{1}{n}} \le \widetilde{K}_{t_1 + \frac{1}{n}} \le \widetilde{K}_{t_2}  \le K_{t_2} \text{ for all } 0\le t_1 < t_2 \le T \text{ and }|t_2 - t_1|> \frac{1}{n}.
\end{equation*}
By taking the intersection of the above cases in $n\in\mathbb{N},$ we have a.s.
\begin{equation*}
K_{t_1}\le K_{t_2} \text{ for all } 0\le t_1 < t_2 \le T,
\end{equation*}
which completes our proof.
\end{proof}

\bibliographystyle{plain}

\end{document}